\UseRawInputEncoding

\documentclass{article}
\usepackage{verbatim}
\usepackage{amsmath,amssymb,amsthm,mathtools}
\usepackage{graphicx}
\usepackage{caption}
\usepackage[dvipsnames]{xcolor}
\usepackage{caption}
\usepackage{subcaption}
\usepackage{multirow}
\usepackage[normalem]{ulem} 
\usepackage{tikz}
\usetikzlibrary{arrows}
\usepackage{dsfont}
\usepackage{enumitem}

\usepackage[numbers,sort&compress]{natbib}

\mathtoolsset{showonlyrefs=true}

\newtheorem{theorem}{Theorem}[section]
\newtheorem{corollary}{Corollary}[section]
\newtheorem{lemma}{Lemma}[section]
\newtheorem{proposition}{Proposition}[section]
\newtheorem{condition}{Condition}
\newtheorem{assumption}{Assumption}
\newtheorem{definition}{Definition}
\newtheorem{example}{Example}
\newtheorem{construction}{Construction}
\newtheorem{notation}{Notation}

\DeclareMathOperator*{\argmin}{arg\,min}

\theoremstyle{remark}
\newtheorem{remark}{Remark}

\title{Multi-dimensional state space collapse\\ in non-complete resource pooling scenarios}
\author{Ellen Cardinaels
\and Sem Borst
\and Johan S.H. van Leeuwaarden}

\begin{document}

\maketitle

\begin{abstract}
The present paper establishes an explicit multi-dimensional state space collapse (SSC) for parallel-processing systems with arbitrary compatibility constraints between servers and job types.
This breaks major new ground beyond the SSC results and queue length asymptotics in the literature which are largely restricted to complete resource pooling (CRP) scenarios where the steady-state queue length vector concentrates around a line in heavy traffic.
The multi-dimensional SSC that we establish reveals heavy-traffic behavior which is also far more tractable than the pre-limit queue length distribution, yet exhibits a fundamentally more intricate structure than in the one-dimensional case, providing useful insight into the system dynamics.
In particular, we prove that the limiting queue length vector lives in a $K$-dimensional cone of which the set of spanning vectors is random in general, capturing the delicate interplay between the various job types and servers. For a broad class of systems we provide a further simplification  which shows that the collection of random cones constitutes a fixed $K$-dimensional cone, resulting in a $K$-dimensional SSC. The dimension~$K$ represents the number of critically loaded subsystems, or equivalently, capacity bottlenecks in heavy-traffic, with $K=1$ corresponding to conventional CRP scenarios.
Our approach leverages probability generating function (PGF) expressions for Markovian systems operating under redundancy policies.
\end{abstract}



\maketitle

\section{Introduction}

In the present paper we analyze the heavy-traffic behavior of parallel-processing systems with redundancy policies in scenarios that go beyond the conventional complete resource pooling (CRP) condition.
In particular, we provide the first explicit characterization of a multi-dimensional state space collapse and the associated (scaled) steady-state queue length vector in fairly general \emph{non}-CRP scenarios.

While there are several formulations in the literature, broadly speaking the CRP condition entails that the system does not experience any local capacity bottlenecks, and operates in heavy traffic as if there is only a single global resource constraint in force.
Thus, the system behaves as a single-server queue with fully pooled resources in a critical-load regime, and the steady-state queue length vector typically concentrates around a line, thus exhibiting a state space collapse (SSC) which yields far greater tractability compared to the pre-limit queue length distribution.

As alluded to above, the CRP condition not only provides an appealing design objective but is also instrumental in facilitating a detailed analysis of the heavy-traffic behavior.
Indeed, \emph{non}-CRP scenarios have eluded an explicit derivation of the limiting queue length distribution in any degree of generality so far, yet such scenarios may naturally arise in various situations.

In order to illustrate this, let us focus for the moment on the simplest possible setting of a system with two job types and two servers.
For convenience, suppose that type-$i$ jobs arrive as a Poisson process of rate $\lambda_i = (1 - \epsilon) \mu_i$ and have independent and exponentially distributed service requirements, with $\mu_i$ denoting the processing speed of server~$i$, $i = 1, 2$.
The parameter $\epsilon \in (0, 1)$ may be interpreted as a relative capacity slack in a baseline scenario where type-$i$ jobs can only be processed by server~$i$, $i = 1, 2$, and both queues are dimensioned to operate at a relative load of $1 - \epsilon$.
Denote by $(Q_1, Q_2)$ the steady-state queue length vector, with $Q_i$ counting the number of type-$i$ jobs, $i = 1, 2$.
In this complete \emph{partitioning} scenario, the system decomposes into two independent M/M/1 queues, and for any non-anticipating and non-preemptive service discipline it holds that
\begin{equation}
\epsilon (Q_1, Q_2) \overset{d}{\rightarrow} (U, U') \hspace*{.2in} \mbox{ as } \epsilon \downarrow 0,
\end{equation}
with $\overset{d}{\rightarrow}$ denoting convergence in distribution, and $U$ and $U'$ representing two independent and exponentially distributed random variables with unit mean.

In contrast, in a complete \emph{sharing} scenario where both job types can be processed by either server, the system operates as a single M/M/1 queue with arrival rate $\lambda_1 + \lambda_2 = (1-\epsilon)(\mu_1+\mu_2)$ and service rate $\mu_1 + \mu_2$.
In that case, it holds for any non-anticipating and non-preemptive service discipline that does not distinguish between the job types that
\begin{equation}
\epsilon (Q_1, Q_2) \overset{d}{\rightarrow} (p_1, p_2) U \hspace*{.2in} \mbox{ as } \epsilon \downarrow 0,
\end{equation}
with $p_i = \frac{\lambda_i}{\lambda_1 + \lambda_2} = \frac{\mu_i}{\mu_1 + \mu_2}$ the fraction of type-$i$ jobs, $i = 1, 2$, and $U$ an exponentially distributed random variable with unit mean.
Thus the (scaled) queue length vector concentrates around a line with slope $(p_1, p_2)$ in heavy traffic, which is a manifestation of a (one-dimensional) SSC.
Note that the (scaled) total number of jobs has the same unit exponential distribution as the number of jobs of each individual type in a complete \emph{partitioning} scenario, and that the (scaled) number of type-$i$ jobs is now stochastically smaller by a factor $p_i \leq 1$, $i = 1, 2$, reflecting the performance gains in the complete \emph{sharing} scenario.

In many situations, however, it may not be feasible for all job types to be handled by all servers.
Indeed, some servers may either be able to only handle generic jobs, or be customized to handle highly specialized jobs, while other servers may be highly flexible but costly or scarcely available.
In the above setup, suppose that type-$1$ jobs can be handled by either server~$1$ or~$2$, while type-$2$ jobs can be handled by server~$2$ only.
In hospital environments, one could imagine that brain or heart surgery patients can only be assigned to a specialized ward with intensive care, while orthopedic treatment patients can be accommodated anywhere.
This scenario is commonly referred to as the `N-model' in the literature in view of the compatibility graph between the job types and the servers as depicted in Figure~\ref{fig:Nmodel}.

As a special case of our results for systems with arbitrary compatibility constraints operating under redundancy policies (as further specified below), it follows for a FCFS service discipline that
\begin{equation}
\label{eq:Nmodel_limiet}
\epsilon (Q_1, Q_2) \overset{d}{\rightarrow} (p_1 U, p_2 U + U') \hspace*{.2in} \mbox{ as } \epsilon \downarrow 0,
\end{equation}
with $U$ and $U'$ two independent and exponentially distributed random variables with unit mean.
Observe that the queue length vector no longer concentrates around a line, but lives in a two-dimensional cone spanned by the vectors $(p_1, p_2)$ and $(0, 1)$.
Thus, the two queue lengths are no longer perfectly correlated as in the complete \emph{sharing} scenario, but are still coupled in a subtle manner, and not completely independent as in the complete \emph{partitioning} scenario.
Indeed, the system exhibits a SSC, since the cone is a subspace of the full two-dimensional state space of the pre-limit queue length, even though it has the same dimension in the particular case of the N-model.
Further note that the (scaled) total number of jobs is distributed as $U + U'$ in the limit, just like in the complete \emph{partitioning} scenario, and that the (scaled) number of type-$2$ jobs is stochastically larger now, since they only have access to server~$2$ and face competition from type-$1$ jobs at that server.
In contrast, the (scaled) number of type-$1$ jobs is stochastically reduced, since they still enjoy exclusive access to server~$1$, but can also compete for access to server~$2$.

\begin{figure}[h]
\centering
\begin{subfigure}[b]{0.49\textwidth}
\centering
\includegraphics[scale=1]{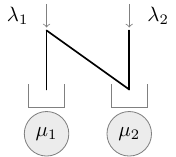}
\caption{Sketch of the system.}
\end{subfigure}
\begin{subfigure}[b]{0.49\textwidth}
\centering
\includegraphics[scale = 1]{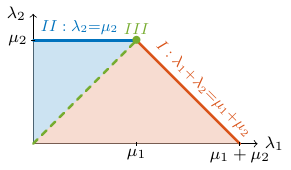}
\subcaption{Visual representation of the stability region.}\label{fig:Nmodel_stab}
\end{subfigure}
\caption{Visualization of the N-model.}\label{fig:Nmodel}
\end{figure}

In the present paper we establish an explicit multi-dimensional SSC and heavy-traffic limits similar to~\eqref{eq:Nmodel_limiet} for Markovian parallel-processing systems with arbitrary compatibility constraints operating under so-called redundancy policies.
Such compatibility constraints are prevalent in a wide variety of societal and technological systems due to heterogeneity in service demands and resources.
Prominent examples include data centers and cloud networks (due to data locality issues),
ride-sharing platforms and housing allocation agencies (due to spatial proximity considerations),
multi-skill service systems with varying degrees of server flexibility,
and organ transplantation networks.


We use probability generating function (PGF) expressions to provide an explicit characterization for the distribution of the limiting queue length vector, which reveals that this queue length vector lives in a $K$-dimensional cone spanned by a set of possibly random vectors, capturing the subtle interaction between the various job types and servers. For a broad class of systems we provide a further simplification, from which we deduce that the collection of random cones forms a fixed $K$-dimensional cone. The dimension~$K$ then represents the number of critical components, or equivalently, capacity bottlenecks in heavy traffic, with $K = 1$ corresponding to conventional CRP scenarios. If we denote the spanning vectors of the fixed $K$-dimensional cone by $\boldsymbol{\alpha}_1, \dots, \boldsymbol{\alpha}_K$, then the limiting (scaled) queue length vector can be represented as a linear combination of these vectors with $K$~independent exponentially distributed random variables associated with the $K$~critically loaded subsystems acting as scalar coefficients. Thus, the number of jobs of a particular type~$i$ is a weighted sum of these random variables, with the weight factors $\alpha_{i,1}, \dots, \alpha_{i, K}$ representing the load fractions that type-$i$ jobs account for in each of the subsystems.

In the N-model discussed above, there are $K = 2$ critical components: the subsystem consisting of type-$2$ jobs and server~$2$, and the `subsystem' comprising the entire system.
Thus type-$1$ jobs only belong to the second `subsystem' where they account for a fraction of the load~$p_1$, translating into the term $p_1 U$ in~\eqref{eq:Nmodel_limiet}, while type-$2$ jobs belong to both subsystems, making up the full load in the first one and accounting for a fraction of the load~$p_2$ in the second one, yielding the term $p_2 U + U'$ in~\eqref{eq:Nmodel_limiet}.
Further observe that the complete \emph{sharing} scenario satisfies the CRP condition with $K = 1$, while the complete \emph{partitioning} scenario amounts to a somewhat degenerate case with $K = 2$ critical components which are in fact entirely independent.\\

Heavy-traffic analysis has a long and rich history, dating back to the pioneering work of Kingman \cite{kingman1961,kingman1962} focusing on the single-server queue.
Subsequent work pursued heavy-traffic limits of an increasingly wide range of stochastic systems, such as queueing networks \cite{Bramson1998,Laws1992,Kelly1993,Williams1998},
parallel-server systems under various static and dynamic assignment policies
\cite{Bell2001,Bell2005,Harrison1999,Harrison1998,Stolyar2005,Mandelbaum2004,Williams2000},
stochastic processing networks~\cite{Dai2008},
bandwidth-sharing networks~\cite{Kang2009}
and generalized switches \cite{hurtado2020transform,Hurtadolange2019,Stolyar2004,Shah2014}.
As testified by these papers, heavy-traffic theory provides a powerful approach to analyze (scaled) queue lengths and delays, yielding valuable insight in the key performance characteristics of a variety of complex systems which would be mostly intractable otherwise.

The above set of references is far from exhaustive and a detailed review of the huge literature on heavy-traffic limits is beyond the scope of the present paper.
Broadly speaking, however, the bulk of the literature pertains to process-level limits and/or systems that satisfy the CRP condition mentioned earlier.
To the best of our knowledge, explicit results for \emph{steady-state queue length distributions} in \emph{non-CRP scenarios} have remained extremely scarce so far.

For example, Kang \textit{et al.}~\cite{Kang2009} derive process-level diffusion limits for Markovian bandwidth-sharing networks operating under $\alpha$-fair rate allocation policies, which generally do not satisfy the CRP condition.
For the special case of linear topologies and $\alpha = 1$ (Proportional Fairness) it is shown that, under certain local traffic assumptions, the steady-state distribution of the limiting diffusion process exhibits a product-form. This result was later extended by Vlasiou~\textit{et al.}~\cite{Vlasiou2021}, allowing for phase-type job size distributions and more general topologies.
The authors in~\cite{Kang2009} conjecture that this product-form carries over to the heavy-traffic limit of the steady-state queue length distribution of the original system. A rigorous proof of such an interchange-of-limits was later provided by Wang~\textit{et al.}~\cite{Wang2022}.

A different strand of more recent research on generalized switch models operating under a MaxWeight scheduling algorithm \emph{does} directly target steady-state heavy-traffic results in non-CRP scenarios.
Hurtado-Lange \& Maguluri~\cite{HurtadoLange2022} adopt techniques from~\cite{Eryilmaz2012} to prove a multi-dimensional SSC, and show how this property can be used to derive (scaled) first moments of certain linear combinations of queue lengths.
They observe however that this approach cannot be readily extended to obtain \emph{distributional} results for the joint queue length vector due to the presence of unknown cross terms.
In order to address the latter challenge, Jhunjhunwala \& Maguluri \cite{Jhunjhunwala2022,Jhunjhunwala2023} extend the transform method developed in~\cite{hurtado2020transform} to non-CRP scenarios to provide an
implicit characterization of the limiting queue length distribution for the class of input-queued switches, along with an explicit characterization in the special case of the above-described N-model under additional symmetry and uniqueness assumptions. A further discussion will be provided in Subsection~\ref{subsec:discussion}, after presenting our results.
Varma \& Maguluri~\cite{Varma2021} focus on the performance impact of the structure of the compatibility constraints in the above context, and present algorithms for selecting the compatibility relations so as to obtain the desired dimension of the SSC or for identifying job server pairs that might decrease the dimension of the collapsed limiting distribution.

A further body of work that is targeted at steady-state heavy-traffic results pertains to Markovian parallel-server systems with a First-Come First-Served (FCFS-ALIS) assignment policy, which can equivalently be thought of as a Join-the-Smallest-Workload (JSW) policy.
Specifically, Af\`eche \emph{et al.}~\cite{Afeche2021} and Hillas \emph{et al.}~\cite{Hillas2023} show that in a heavy-traffic regime the system decomposes into several components or subsystems each experiencing a critical load, and present expressions for the expected waiting times of jobs in the various components. 

The model set-up that we adopt strongly resembles that in \cite{Afeche2021,Hillas2023} and also permits arbitrary compatibility constraints between servers and job types.
Rather than deriving expected waiting times, however, we focus on establishing a multi-dimensional SSC which yields distributional results for the joint queue length as well as delays (along with higher moments).\\



We consider redundancy policies which provide a natural class of job assignment mechanisms in the presence of compatibility constraints.
Their key feature is to make copies or replicas of every arriving job and to assign one copy to each of its compatible servers, with the aim to exploit the variability in queue lengths and/or service times encountered by the different copies of the same job.
As soon as one copy either terminates or initiates service the remaining copies are discarded, which is referred to as the cancel-on-completion (c.o.c.) or cancel-on-start (c.o.s.) version of the redundancy policy, respectively.
The latter version is equivalent to the JSW policy mentioned above.

Besides their natural fit to handle compatibility constraints, redundancy policies come with the remarkably luxury of mathematical tractability.
Specifically, product-form expressions for the stationary distribution of the system occupancy in Markovian settings were derived by Gardner \emph{et al.}~\cite{Gardner2016queueing} and Ayesta \emph{et al.}~\cite{ayesta2018unifying}, and are closely related to earlier results by Visschers \emph{et al.}~\cite{visschers2012product} and Adan \& Weiss~\cite{adan2014skill}.
An overview of related models involving job-server compatibility constraints that yield product-form expressions is provided by Gardner \& Righter~\cite{Gardner2020}.

At first sight, the availability of such product-form expressions may seem to defeat the entire rationale of a heavy-traffic analysis.
However, the involved state description of the system occupancy is so detailed that the product-form expressions unfortunately cannot be directly used for analyzing key performance metrics like queue lengths or delays \cite{Gardner2017scheduling,visschers2012product,Adan2018}.
Our approach therefore leverages more convenient PGF expressions for the joint queue length vector~\cite{Cardinaels2022} which pave the way for the analytical derivation of heavy-traffic limits and additionally lend themselves to a more insightful probabilistic interpretation in terms of geometrically distributed random variables.
Closer inspection then not only allows for an alternative probabilistic derivation of the heavy-traffic behavior, but also illuminates that the $K$~independent exponentially distributed random variables in the stochastic representation of the steady-state queue length vector reflect the contributions to the overall queue length from the $K$~critically loaded subsystems.
Building on the product-form expressions (as opposed to proving process-level heavy-traffic limits) offers the further advantage that convergence of the steady-state queue length distribution comes for free, without the need to prove an interchange-of-limits which usually is a significant challenge.

The remainder of this paper is organized as follows.
The detailed model description and preliminary results can be found in Sections~\ref{subsec:modelparameters} and~\ref{sec:preliminaries}, respectively. The main results for the c.o.c.\ policy are outlined in Section~\ref{sec:main_results}, while those for the c.o.s.\ policy are deferred to Appendix~\ref{app:main_results_cos}. The pre-limit characterization of the queue length vector and an alternative proof of the main heavy-traffic results, using stochastic arguments and a probabilistic interpretation of the PGF expressions are presented in Appendix~\ref{sec:alternative proof}.
We conclude with some topics for further research in Section~\ref{sec:outlook}.

\section{Model description}\label{subsec:modelparameters}
Consider a system consisting of $N$ parallel servers, each with their own dedicated waiting line, where jobs arrive according to a Poisson process with rate $N\lambda$, $\lambda> 0$. The servers, indexed as $1,\dots,N$, process assigned jobs in order of arrival at speed $\mu_n> 0$, for $n =1,\dots,N$. 
The average processing speed per server is denoted by $\mu \coloneqq (1/N)\sum_{n=1}^N \mu_n$ .

The jobs can be categorized into different classes or types based on the specific subset of servers they are compatible with. In particular, a job that is compatible with all servers in $S\subseteq\{1,\dots,N\}$ is labeled as a type-$S$ job. The fraction of jobs that receive a type-$S$ label is denoted by $p_S$, and the collection of all job types is referred to as $\mathcal{S} \coloneqq \{ S\subseteq\{1,\dots,N\}\colon p_S >0\}$.
The arrival process can equivalently be seen as $|\mathcal{S}|$ independent Poisson processes with rates $\{\lambda_S = N\lambda p_{S} \colon S\in \mathcal{S}\}$ and corresponding type labels.

A bipartite graph can be constructed to represent the compatibility constraints of the various job types, with a node for each job type  and each server. A job type $S$ is connected to server $n$ whenever $n\in S$, for all $S\in\mathcal{S}$ and $n=1,\dots,N$.

We denote the total number of type-$S$ jobs, and the total number of \emph{waiting} type-$S$ jobs by $Q_S$ and $\tilde{Q}_S$, respectively, for all $S\in\mathcal{S}$.

The arriving jobs are assigned to the various servers according to a redundancy policy. In particular, copies of the job are made for each of the compatible servers and forwarded to the respective waiting lines. 
Each copy is assumed to have an exponentially distributed service requirement with unit mean.
The remaining (or redundant) copies are removed once one copy either \emph{finishes} its service or \emph{initiates} its service. The former setting is referred to as cancel-on-completion (c.o.c.)\ and the latter as cancel-on-start (c.o.s.). The aim of the redundancy policy is to exploit the variability in the queue lengths at different servers and the service requirements of different copies of the same job, which are assumed to be independent in case of the c.o.c.\ mechanism \cite{Gardner2016queueing,ayesta2018unifying,visschers2012product,Adan2018}. 


\section{Preliminaries}\label{sec:preliminaries}
\subsection{Stability and pre-limit results}
In this section we elaborate on some of the known results for the system described in the previous section which will form the foundations for the analysis in Section~\ref{sec:main_results}.

First, the conditions below are necessary and sufficient in order for the system to be stable~\cite{Gardner2016queueing,visschers2012product}.

\begin{condition}[Stability]\label{cond:stab}
The system described in Section~\ref{subsec:modelparameters} is stable if for all non-empty subsets of job types $\mathcal{T} \subseteq \mathcal{S}$, 
$N\lambda p(\mathcal{T}) < \mu(\mathcal{T})$,
where $p(\mathcal{T})\coloneqq \sum_{S\in \mathcal{T}} p_S$ denotes the fraction of jobs with labels in $\mathcal{T}$ and $\mu(\mathcal{T})\coloneqq\sum_{n\in \mathcal{T}} \mu_n$ denotes the aggregate service rate of all servers that are compatible with at least one job type in $\mathcal{T}$.
\end{condition}

Under these stability conditions, product-form expressions for the steady-state distribution of the system state were derived by Gardner \emph{et al.}\ \cite{Gardner2016queueing}, Ayesta \emph{et al.}~\cite{ayesta2018unifying} and Visschers \emph{et al.}~\cite{visschers2012product}.
The downside is that these expressions involve a very detailed state descriptor and offer little insight in the system performance. In particular, the product-form expressions have a factor for each (waiting) job in the system and hence give some information about the steady-state distribution of the (ordered and centralized) waiting line. However, it is not immediately clear how to gain more global understanding of the system dynamics, for instance in terms of the compatibility constraints or the total numbers of jobs of the various types.

The joint probability generating function (PGF) of the numbers of jobs of the various types was derived in~\cite{Cardinaels2022} by suitably aggregating the states. For completeness, the PGFs for the c.o.c.\ and the c.o.s.\ mechanisms are provided in Propositions~\ref{prop:pgf} and~\ref{prop:pgfcos} (in Appendix~\ref{app:prelim_cos}), respectively.

\begin{proposition}[Probability generating function - c.o.c.]\label{prop:pgf}
The joint PGF of the numbers of jobs of the various types for the redundancy c.o.c.\ policy is given by 
\begin{equation}\label{eq:pgf}
\mathbb{E}\bigr[\prod\limits_{S \in {\mathcal S}} z_S^{Q_S}\bigr] = \frac{f(\boldsymbol{z})}{f(\boldsymbol{1})},
\end{equation}
where $\boldsymbol{z}$ and $\boldsymbol{1}$  are $|\mathcal{S}|$-dimensional vectors with entries $|z_S| \le 1$ and 
\begin{equation}\label{eq:f}
 f(\boldsymbol{z}) = 1{+} \sum\limits_{m=1}^{|\mathcal{S}|}\sum\limits_{\boldsymbol{S}\in \mathcal{S}_m} \prod\limits_{j=1}^m \frac{N\lambda p_{S_j}z_{S_j}}{\mu(S_1,\dots,S_j)}
 \biggl(1-\frac{N\lambda}{\mu(S_1,\dots,S_j)}\sum\limits_{i=1}^j p_{S_i}z_{S_i}\biggl)^{-1}.
\end{equation}
The $m$-dimensional vector $\boldsymbol{S}$ consists of $m$ different job types, and the set consisting of all these vectors is denoted by $\mathcal{S}_m$~\textup{\cite{Cardinaels2022}}.
\end{proposition}

To ease the notation throughout this paper we introduce the following notation when focusing on the first $j$ entries of the vector~$\boldsymbol{S}$,
\begin{equation}\label{eq:abbreviated}
p(\boldsymbol{S},j) \coloneqq \sum\limits_{i=1}^j p_{S_i}
\quad \text{and} \quad
\mu(\boldsymbol{S},j)  \coloneqq \mu(S_1,\dots,S_j) = \sum_{n\in \cup_{i=1}^j S_i} \mu_n.
\end{equation}

\subsection{Heavy-traffic regime and non-complete resource pooling}\label{subsec:HT_and_nonCRP}

We are mainly interested in the system performance when the average arrival rate per server, $\lambda$, approaches a critical value. 
\begin{definition}[Critical arrival rate]\label{def:critical_arrival_rate}
The critical arrival rate of the system, $\lambda^*$, for given values of $\mu_n$, $n=1,\dots,N$ and $p_S$, $S\in\mathcal{S}$, is defined as
$\frac{1}{N} \min_{\mathcal{T}\subseteq \mathcal{S}} \left\{ \mu(\mathcal{T})/p(\mathcal{T}) \right\}$, 
with $\mu(\mathcal{T})$ and $p(\mathcal{T})$ as defined in Condition~\ref{cond:stab}.
\end{definition}

Considering model parameters that satisfy the stability conditions (Condition~\ref{cond:stab}), when the arrival rate~$\lambda$ is increased, $(N\lambda p_S)_{S\in\mathcal{S}}$ will reach the boundary of the stability region when $\lambda$ becomes equal to the associated critical arrival rate~$\lambda^*$. To specify which part of the boundary of the stability region is reached, we define the critical subset(s) of job types.
\begin{definition}[Critical subset of types] \label{def:critical_subset}
The subset $\mathcal{T}\subseteq\mathcal{S}$ of job types is called critical if
$$\mathcal{T}\in\argmin_{\mathcal{T}\subseteq \mathcal{S}} \left\{ \mu(\mathcal{T})/p(\mathcal{T}) \right\},$$ 
with $\mu(\mathcal{T})$ and $p(\mathcal{T})$ as defined in Condition~\ref{cond:stab}. The collection of all critical subsets of job types is denoted by $\mathcal{CR(S)}$.
\end{definition}

A common assumption in the heavy-traffic literature is formalized below.
\begin{condition}[CRP]\label{cond:crp}
The system described in Section~\ref{subsec:modelparameters} satisfies the complete resource pooling \textup{(}CRP\textup{)} condition if 
there is a unique critical subset of job types, i.e., $\mathcal{CR}(\mathcal{S}) = \{\mathcal{T}^*\}$ for some $\mathcal{T}^*\subseteq\mathcal{S}$.
Furthermore we distinguish between \emph{weak} and \emph{strong} CRP depending on whether $\mathcal{T}^*\subsetneq\mathcal{S}$ or $\mathcal{T}^*=\mathcal{S}$, respectively.
\end{condition}
Notice that the condition for strong CRP is equivalent to 
$ p(\mathcal{T}) < \mu(\mathcal{T})/(N\mu)$,
for all strict subsets of job types $\mathcal{T} \subsetneq \mathcal{S}$. Moreover, the critical arrival rate is as large as possible, i.e., $\lambda^* = \mu$.
We refer to~\cite[Appendix EC.1]{Cardinaels2022} for an in-depth comparison between the different notions of the CRP condition in the literature specialized to the setting of a parallel-server system, .

\begin{definition}[non-CRP] All system settings as described in Section~\ref{subsec:modelparameters} that do not satisfy the CRP condition \textup{(}Condition~\ref{cond:crp}\textup{)}, are referred to as \emph{non-CRP scenarios}.
\end{definition}


For these non-CRP scenarios we introduce the notions of \emph{depth} of the critical subsets and \emph{$k$-critical vectors} to study the critical subsets of job types in more detail.

\begin{definition}[Depth of the critical subsets]\label{def:depthK} The depth $K$ of the critical subsets of job types is defined as
$
K\coloneqq \max\left\{ k \colon \mathcal{T}_1 \subsetneq \mathcal{T}_2 \subsetneq \dots \subsetneq \mathcal{T}_k \text{~with~} \mathcal{T}_1,\mathcal{T}_2,\dots,\mathcal{T}_k \in \mathcal{CR(S)} \right\}.
$
Thus $K$ represents the maximum number of critical subsets that can be nested into each other.
\end{definition}
Note that the stability region can be thought of as an intersection of half spaces using the inequalities in Condition~\ref{cond:stab}, where $K$ coincides with the number of half spaces that the critical arrival rate vector $(N\lambda^* p_S)_{S\in\mathcal{S}}$ reaches simultaneously. Moreover, depending on the model parameters, $K$ can take any integer value between 1 and $|\mathcal{S}|$.

The following connection between Condition~\ref{cond:crp} and Definition~\ref{def:depthK} can be made:
\begin{lemma}\label{lem:CRPvsK}
The following two statements are equivalent for given values of $\mu_n$, $n=1,\dots,N$, and $p_S$, $S\in\mathcal{S}$:
\textup{(}i.\textup{)} The \textup{(}weak\textup{)} CRP condition is satisfied; \textup{(}ii.\textup{)} The depth of the critical subsets $K=1$. 
\end{lemma}
\begin{proof}
To prove the equivalence, we exploit the definition and properties of the collection of all critical subsets~$\mathcal{CR}(\mathcal{S})$.
The first statement obviously implies the second statement since $|\mathcal{CR}(\mathcal{S})|=1$. For the reverse statement we show that $K$ will always be strictly larger than 1 if there is more than one critical subset of job types, i.e., $|\mathcal{CR}(\mathcal{S})|>1$. Let $\mathcal{T}_1,\mathcal{T}_2\in\mathcal{CR}(\mathcal{S})$ be two distinct critical subsets of job types. We distinguish two scenarios: $\mathcal{T}_1\cap\mathcal{T}_2 \in \mathcal{CR}(\mathcal{S})$ and 
$\mathcal{T}_1\cap\mathcal{T}_2 \notin \mathcal{CR}(\mathcal{S})$.
In the former setting it is clear that $\mathcal{T}_1\cap\mathcal{T}_2 \subsetneq \mathcal{T}_1$ and hence $K \ge 2$. We show that the latter setting contradicts the definitions of the critical arrival rate $\lambda^*$ and the collection of all critical subsets of job types (Defintion~\ref{def:critical_arrival_rate}) and the stability condition (Condition~\ref{cond:stab}). If $\mathcal{T}_1\cap\mathcal{T}_2 \notin \mathcal{CR}(\mathcal{S})$, then
\begin{equation}
\begin{array}{rcl}
N\lambda^* p(\mathcal{T}_1\cup\mathcal{T}_2) & = & N\lambda^* \left( p(\mathcal{T}_1) + p(\mathcal{T}_2) - p(\mathcal{T}_1\cap\mathcal{T}_2) \right) \\
& = &  \mu(\mathcal{T}_1) + \mu(\mathcal{T}_2) - N\lambda^*p(\mathcal{T}_1\cap\mathcal{T}_2) \\
& > &  \mu(\mathcal{T}_1) + \mu(\mathcal{T}_2) - \mu(\mathcal{T}_1\cap\mathcal{T}_2) \\
& \ge &   \mu(\mathcal{T}_1\cup\mathcal{T}_2). \\
\end{array}
\end{equation}
This violates the stability constraint induced by the subset of job types $\mathcal{T}_1\cup\mathcal{T}_2$. The latter inequality follows from the fact that $\mu(\mathcal{T}_1\cup\mathcal{T}_2) + \mu(\mathcal{T}_1\cap\mathcal{T}_2) \le \mu(\mathcal{T}_1) + \mu(\mathcal{T}_2)$. Indeed, if server $n$ is compatible with a job type in $\mathcal{T}_1\cap\mathcal{T}_2$, then it is compatible with a job type in $\mathcal{T}_1$ and a job type in $\mathcal{T}_2$ such that the term $\mu_n$ will occur twice on both sides of the inequality. If server $n$ is compatible with a job type in $\mathcal{T}_1\cup\mathcal{T}_2$ but not with a job type in $\mathcal{T}_1\cap\mathcal{T}_2$, then there exists at least one job type in $\mathcal{T}_1$ or $\mathcal{T}_2$ that server $n$ is compatible with. So, since server $n$ contributes once to the left hand side, and at least once to the right hand side, the above inequality holds.
This concludes the proof.
\end{proof}

As can be seen in Proposition~\ref{prop:pgf}, the expression for the PGF of the numbers of jobs of the various types depends on an enumeration of all vectors whose entries are distinct job types, i.e., all $m$-dimensional vectors $\boldsymbol{S}= [S_1,\dots,S_m]\in\mathcal{S}_m$ with $m=1,\dots,|\mathcal{S}|$.
We now introduce a different perspective on the ordered vectors of job types that focuses on their relation with the critical subsets rather than their lengths.

\begin{definition}[$k$-Critical vectors]\label{def:crit_vectors}
Let $\boldsymbol{T} = [T_1,T_2,...,T_m]\in\mathcal{S}_m$ be an ordered vector with $m$ distinct entries belonging to the set of job types~$\mathcal{S}$. Then $\boldsymbol{T}$ is called $k$-critical if 
$|\{ i \colon T_1 \cup \dots \cup T_i \in \mathcal{CR(S)}, i = 1,\dots, m \} | = k$.
In other words, $\boldsymbol{T}$ induces $k$ critical subsets when combining the respective job types one by one.
The collection of all $k$-critical vectors is denoted by $\mathcal{N}_k$ and $\mathcal{N}\coloneqq \cup_{k=0}^K \mathcal{N}_k$.
Moreover, for a fixed $k$-critical vector $\boldsymbol{T}\in\mathcal{N}_k$ we define 
$\text{CR}(\boldsymbol{T}) \coloneqq \{i\colon T_1\cup \dots \cup T_i \in \mathcal{CR(S)}\}$
as the set that contains all indices $i\in \{1,\dots,|\boldsymbol{T}|\}$ such that aggregation of the first $i$ entries of~$\boldsymbol{T}$ yields a critical subset of job types.
In particular, if such indices do not exist, i.e., $\text{CR}(\boldsymbol{T}) = \emptyset$, then $\boldsymbol{T}$ is referred to as a 0-critical vector.  
\end{definition}
Note that $\mathcal{N}_k$ is empty for all $k > K$ and that $\bigcup_{k=0}^K \mathcal{N}_k = \bigcup_{m=0}^{|\mathcal{S}|} \mathcal{S}_m$.

\begin{example}[N-model]\label{ex:Nmodel} 
Figure~\ref{fig:Nmodel} visualizes the N-model consisting of two job types and two servers and its stability region. The first job type is compatible with both servers, while the second job type is only compatible with the second server. Henceforth, the two job types are labeled as type-$\{1,2\}$ and type-$\{2\}$ jobs, respectively. We observe that there are three different heavy-traffic scenarios that can occur corresponding to the various edges and corner points of the stability region, labeled as $I$, $II$ and $III$ in Figure~\ref{fig:Nmodel_stab}. Table~\ref{tab:Nmodel_overview} summarizes these different heavy-traffic scenarios together with an illustration of the above definitions. 
\end{example}

\begin{table}[h]
\centering
\begin{tabular}{|c|c|c|c|c|l|} 
\hline
\textbf{Scenario} & \textbf{CRP} & $\boldsymbol{\mathcal{CR}(\mathcal{S})}$ & $\boldsymbol{\lambda^*}$ & $\boldsymbol{K}$ & \textbf{Critical vectors, $\boldsymbol{T\in\mathcal{N}_k}$} \\
\hline\hline
\multirow{2}{*}{I} & \multirow{2}{*}{Strong} & \multirow{2}{*}{$\{\mathcal{S}\}$} & \multirow{2}{*}{$\mu$} & \multirow{2}{*}{1}& $k=0$: $\emptyset, [\{1,2\}], [\{2\}] $ \\
& & & & & $k=1$: $ [\{1,2\},\{2\}], [\{2\},\{1,2\}] $ \\
     \cline{1-6}              
\multirow{2}{*}{II} & \multirow{2}{*}{Weak} & \multirow{2}{*}{$\{\{2\}\}$} & \multirow{2}{*}{$\mu_2/ (Np_{\{2\}})<\mu$} & \multirow{2}{*}{1}& $k=0$: $\emptyset, [\{1,2\}],[\{1,2\},\{2\}]  $ \\
                   & & & & &  $k=1$: $[\{2\}],[\{2\},\{1,2\}]  $ \\
       \cline{1-6}              
\multirow{3}{*}{III} & \multirow{3}{*}{No} & \multirow{3}{*}{$\{\{2\},\mathcal{S}\}$} & \multirow{3}{*}{$\mu$} & \multirow{3}{*}{2} & $k=0$: $\emptyset, [\{1,2\}]  $ \\
                   & & & & &  $k=1$: $[\{2\}],[\{1,2\},\{2\}] $ \\
                   & & & & &  $k=2$: $[\{2\},\{1,2\}] $   \\
\hline
\end{tabular}
\caption{An overview of the different heavy-traffic scenarios of the N-model for which $\mathcal{S} = \{\{1,2\},\{2\}\}$, using definitions of Subsection~\ref{subsec:HT_and_nonCRP}.}\label{tab:Nmodel_overview}
\end{table}

\subsection{CRP components}

In this paper we pursue a heavy-traffic analysis for the \emph{entire} boundary of the stability region and in particular extend the results from~\cite{Cardinaels2022} allowing for boundary points to lie on an intersection of multiple faces of the stability region. These intersections appear when multiple subsets of job types simultaneously approach the stability conditions (Condition~\ref{cond:stab}), i.e., there are several active capacity bottlenecks in heavy traffic. 
These kinds of scenarios yield a more intricate analysis and complex behavior, and are hence typically excluded from consideration in the literature.


To describe our results in Section~\ref{sec:main_results} and to connect them to the work of Af\`eche \emph{et al.}~\cite{Afeche2021} which focuses on the c.o.s.\ mechanism, we first review some of the concepts and notation from~\cite{Afeche2021}. Then we will explain how our set-up with critical subsets fits their vocabulary.

Given a bipartite graph structure representing the compatibility constraints between job types and servers, the collection of edges that yield a stable system is referred to as a \textit{matching}~$M$. From this matching a \textit{residual matching}~$\tilde{M}$ is obtained by solving a limiting ($\lambda = \lambda^*$) maximum-flow problem related to the system, then $\tilde{M}$ consists of all edges with a non-zero flow in the optimal solution.
Informally speaking, one can think of the residual matching~$\tilde{M}$ as a subset of the edges of the original matching~$M$ which will be dominant in the heavy-traffic regime. This residual matching~$\tilde{M}$ will induce a partitioning of the job types and servers in so-called \textit{CRP components}.
\begin{definition}[CRP components]\label{def:CRP_components}
Given the residual matching~$\tilde{M}$, the system can be partitioned in $K$ distinct connected components or CRP components.
They are labeled $\mathbb{C}_1, \dots, \mathbb{C}_K$, and each component consists of a subset of job types $\mathcal{C}_k$ and their \emph{asymptotically} compatible servers~$\mathcal{Z}_k$, i.e., $\mathbb{C}_k \coloneqq (\mathcal{C}_k,\mathcal{Z}_k)$.  
\end{definition}
\begin{remark}
Note that $\mathcal{Z}_k$ is a subset of the servers that the job types in~$\mathcal{C}_k$ are compatible with, i.e., $\mathcal{Z}_k  \subseteq \bigcup_{T\in\mathcal{C}_k} T$
as the residual matching~$\tilde{M}$ is a subset of the original matching~$M$.
\end{remark}

Using these CRP components, a \textit{directed acyclic graph} (DAG) can be obtained by creating an edge $(\mathbb{C}_i,\mathbb{C}_j)$ whenever there are job types in $\mathbb{C}_i$ that are compatible with any of the servers in $\mathbb{C}_j$~\cite[Lemma~2]{Afeche2021}. Building on this DAG, the different CRP components can be ordered.
\begin{definition}[Topological order] Let $\sigma = (\sigma(1),\dots,\sigma(K))$ be a permutation of $\{1,\dots,K\}$. Then $\left(\mathbb{C}_{\sigma(1)},\dots,\mathbb{C}_{\sigma(K)}\right)$ is called a topological order if for any directed arc $(\mathbb{C}_i,\mathbb{C}_j)$ associated with the above-mentioned DAG, it holds that $\sigma^{-1}(j) < \sigma^{-1}(i)$. In words, if component $\mathbb{C}_i$ can forward \textup{(}some of\textup{)} its jobs to component~$\mathbb{C}_j$, then component~$\mathbb{C}_j$ must occur earlier in the topological ordering. The set of all permutations yielding topological orders is denoted by $\Sigma_K$.
\end{definition}

\begin{example} \label{ex:N_model_CRP_components}
Consider the N-model in Example~\ref{ex:Nmodel}, assume that $(p_{\{1,2\}},p_{\{2\}}) = (\mu_1,\mu_2)/(\mu_1+\mu_2)$. Hence, scenario III will occur in the heavy-traffic regime \textup{(}Table~\ref{tab:Nmodel_overview}\textup{)}. The residual matching~$\tilde{M}$ is given by two edges, i.e., connecting job type~$\{1,2\}$ to server 1 and job type~$\{2\}$ to server 2. This disconnects the compatibility graph and results in two CRP components, $ \mathbb{C}_1 = (\mathcal{C}_1,\mathcal{Z}_1) = (\{1,2\}, \{1\})$ and $ \mathbb{C}_2 = (\mathcal{C}_2,\mathcal{Z}_2) =(\{2\}, \{2\})$. Since the job type in component $\mathbb{C}_1$ is compatible with a server in component $\mathbb{C}_2$, the associated DAG is simply given by the edge $(\mathbb{C}_1,\mathbb{C}_2)$. Hence, the only permutation yielding a valid topological ordering of the CRP components is $\Sigma_2 = \{(2,1)\}$.
\end{example}

Using the above definitions, we can trace the concepts of CRP components back to our critical subsets of job types (Definition~\ref{def:critical_subset}).
\begin{construction}\label{construction} ~\newline\vspace{-0.45cm}
\begin{enumerate}
\item Enumerate all $K$ distinct CRP components: $\mathbb{C}_k = (\mathcal{C}_k,\mathcal{Z}_k)$ with $k=1,\dots,K$.
\item Enumerate all topological orders: $\left(\mathbb{C}_{\sigma(1)},\dots,\mathbb{C}_{\sigma(K)}\right)$ for all $\sigma\in\Sigma_K$.
\item Define for all $k=1,\dots,K$ and $\sigma\in\Sigma_K$
\begin{equation}
T_{\sigma,k} \coloneqq \bigcup\limits_{i=1}^k \mathcal{C}_{\sigma(i)}
\quad \text{and}\quad
Z_{\sigma,k} \coloneqq \bigcup\limits_{i=1}^k \mathcal{Z}_{\sigma(i)}.
\end{equation}
\end{enumerate}
\end{construction}

\begin{remark}The obtained sets $(T_{\sigma,k})_{\sigma,k}$ do not all have to be different. For instance, if $\lambda^* = \mu$, then the full set of job types~$\mathcal{S}$ is a critical subset, or each job type is connected to at least one (compatible) server in the residual matching~$\tilde{M}$. Hence, $T_{\sigma,K} \equiv \mathcal{S}$ for all $\sigma\in\Sigma_K$ .
\end{remark}

\begin{theorem}\label{th:construction}
Construction~\ref{construction} yields all critical subsets of job types, i.e.,
$
\mathcal{CR}(\mathcal{S}) = \{T_{\sigma,k} \colon k=1,\dots,K, \sigma\in\Sigma_K \}.
$
\end{theorem}
The proof of Theorem~\ref{th:construction} builds on arguments in the proofs of \cite[Lemmas 3 and 4]{Afeche2021} which focus on the properties of the CRP components and is deferred to Appendix~\ref{app:proof_construction}. 



\begin{example}
Applying Construction~\ref{construction} to Example~\ref{ex:N_model_CRP_components}, we observe that we obtain $T_{\sigma,1} = \{\{2\}\}, T_{\sigma,2} = \{\{2\},\{1,2\}\} = \mathcal{S}$ for the job types and  $Z_{\sigma,1} = \{2\}, Z_{\sigma,2} = \{1,2\}$ for the servers. Indeed $\mathcal{CR}(\mathcal{S}) = \{T_{\sigma,1},T_{\sigma,2} \} = \{\{2\},\mathcal{S}\}$ as also depicted in Table~\ref{tab:Nmodel_overview}.
\end{example}

\begin{remark} \label{rem:K}
As the subset of job types of different CRP components, i.e., $(\mathcal{C}_k)_k$, are non-empty, we have that
$
T_{\sigma,1} \subsetneq T_{\sigma,2} \subsetneq \dots \subsetneq  T_{\sigma,K}
$
for any $\sigma\in\Sigma_K$. Hence, the notions of $K$ in Definition~\ref{def:depthK} (depth of the critical subsets) and Definition~\ref{def:CRP_components} (the number of CRP components) coincide.
\end{remark}


In the remainder of this manuscript we will adhere to the following notation.
\begin{notation} 
The full set of job types is always denoted by $\mathcal{S}$, and any subset of $\mathcal{S}$ will also be written in calligraphic font, e.g., $\mathcal{T}\subseteq \mathcal{S}$, while individual job types are written in regular Roman font, e.g., $T\in\mathcal{S}$. 
Vectors receive a bold letter and their elements are denoted by the same \textup{(}non-bold\textup{)} letter and a subscript, e.g., $\boldsymbol{T} = [T_1,T_2,\dots,T_j]$. We use $\mathbb{N}$ to denote the set of natural numbers, including zero.
\end{notation}

\section{Main results}\label{sec:main_results}
In this section we state the main heavy-traffic result in Theorem~\ref{cor:mixture_dist}, which applies to both the c.o.c.\ and the c.o.s.\ mechanism. 
As the notation in Section~\ref{sec:preliminaries} already alludes to, we assume that the fractions of jobs of each type remain fixed as the system approaches its heavy-traffic limit, i.e., $(\lambda_S)_{S\in\mathcal{S}} = N\lambda(p_S)_{S\in\mathcal{S}}$ and $\lambda\uparrow\lambda^*$.
A generalization of the main result where this assumption is relaxed is formulated in Theorem~\ref{th:general_results_extension} in Appendix~\ref{subsec:extension}.
In the proofs below, we will focus on the c.o.c.\ mechanism, and details for the c.o.s.\ mechanism are deferred to Appendix~\ref{app:main_results_cos}.\\

For notational convenience, we define  
\begin{equation}\label{eq:beta}
\beta(\boldsymbol{T}) \coloneqq \prod\limits_{i=1}^{|\boldsymbol{T}|}
\frac{N\lambda^*p_{T_i}}{\mu(\boldsymbol{T},i)}
 \prod\limits_{i\notin \text{CR}(\boldsymbol{T})} \left( 1 - \frac{N\lambda^* p(\boldsymbol{T},i)}{\mu(\boldsymbol{T},i)} \right)^{-1} 
\end{equation}
for any $\boldsymbol{T}\in\mathcal{N}$, 
\begin{equation} \label{eq:beta_N_K}
\beta(\mathcal{N}_k) \coloneqq \sum\limits_{\boldsymbol{T}\in\mathcal{N}_k} \beta(\boldsymbol{T})
\end{equation}
for $k=0,\dots,K$, and
\begin{equation}\label{eq:P_star_T}
\mathbb{P}^*(\boldsymbol{T}) \coloneqq \frac{\beta(\boldsymbol{T})}{\beta(\mathcal{N}_K)}
\end{equation}
for any $\boldsymbol{T}\in\mathcal{N}_K$. The value $\mathbb{P}^*(\boldsymbol{T})$ represents the limiting probability that the state of the system corresponds to the ordered vector of job types~$\boldsymbol{T}$, i.e., the oldest job in the system is of type $T_1$, the oldest job that is not of type $T_1$ is of type $T_2$, etc.

The expression for the PGF in Proposition~\ref{prop:pgf} can then be leveraged to prove the following convergence result. 

\begin{proposition}[Convergence of PGFs]\label{th:limit_gf}
With the definitions as in Section~\ref{sec:preliminaries},
the joint PGF of the vector of \textup{(}scaled\textup{)} queue lengths, i.e., $\mathbb{E}\bigr[\prod_{S\in\mathcal{S}} z_S^{Q_S}\bigr]$ with $z_S \coloneqq \exp\bigl(-(1-\frac{\lambda}{\lambda^*})t_S\bigl)$, converges to
\begin{equation}\label{eq:GF}
\sum\limits_{\boldsymbol{T}\in\mathcal{N}_K} \mathbb{P}^*(\boldsymbol{T})
\prod\limits_{i\in \text{CR}(\boldsymbol{T})}
\Bigl(1+ \sum_{j=1}^i t_{T_j}\frac{p_{T_j}}{p(\boldsymbol{T},i)}\Bigl)^{-1},
\end{equation}
as $\lambda \uparrow \lambda^*$ with $t_S \ge 0$ for all $S\in\mathcal{S}$.
\end{proposition}


Using L\'evy's Continuity Theorem~\cite{Jacod2000} and the above expression, it can be shown that the (scaled) numbers of jobs of the various types converge in distribution to a random vector $(X_S)_{S\in\mathcal{S}}$ associated with a mixture distribution. There is a mixture weight for each $\boldsymbol{T}\in\mathcal{N}_K$ and a corresponding mixture component which consists of a weighted sum of $K$~independent exponential random variables.
Define the indices $i_1,\dots,i_K$ such that  $\{i_1,\dots,i_K\} = \mathrm{CR}(\boldsymbol{T})$ with $i_1<i_2<\dots<i_K$ and $j_{S}(\boldsymbol{T})$ the position of $S$ in $\boldsymbol{T}$ such that $T_{j_{S}(\boldsymbol{T})} = S$ for any $S\in\mathcal{S}$ and $\boldsymbol{T}\in\mathcal{N}_K$.
Let $\left(I_{\boldsymbol{T}}\right)_{\boldsymbol{T}\in\mathcal{N}_K}
= \boldsymbol{e}(\tilde{\boldsymbol{T}})$ with probability $\mathbb{P}^*(\tilde{\boldsymbol{T}})$ for any $\tilde{\boldsymbol{T}}\in\mathcal{N}_K$ and define $\boldsymbol{e}(\tilde{\boldsymbol{T}})$ as a $|\mathcal{N}_K|$-dimensional unit vector with a one entry for the location corresponding to the ordered vector $\tilde{\boldsymbol{T}}$, $\tilde{\boldsymbol{T}}\in\mathcal{N}_K$. Then, the \textup{(}random\textup{)} coefficients are given by
\begin{equation}\label{eq:P_k_S_random}
P_{k,S} \coloneqq \sum\limits_{\boldsymbol{T}\in\mathcal{N}_K} I_{\boldsymbol{T}} \frac{p_S}{p(\boldsymbol{T},i_k)}\mathds{1}\{i_k \ge j_S(\boldsymbol{T})\},
\end{equation}
for any $S\in\mathcal{S}$ and $k=1,\dots,K$, such that
\begin{equation}\label{eq:X_s}
(X_S)_{S\in\mathcal{S}}
\overset{d}{=}
\Bigl( \sum\limits_{k=1}^K P_{k,S} U_k 
\Bigl)_{S\in\mathcal{S}},
\end{equation}
with $U_1,\dots,U_K$ independent and exponentially distributed random variables with unit mean.

\begin{theorem}[Convergence to a mixture distribution] \label{cor:mixture_dist}
The vector of \textup{(}scaled\textup{)} queue lengths, $(1-\frac{\lambda}{\lambda^*})(Q_S)_{S\in\mathcal{S}}$, converges in distribution to a random vector associated with the mixture distribution as specified above, i.e.,
\begin{equation}
\Bigl(1-\frac{\lambda}{\lambda^*}\Bigl) \left( Q_S \right)_{S\in\mathcal{S}}
\overset{d}{\rightarrow}
(X_S)_{S\in\mathcal{S}},
\end{equation}
as $\lambda\uparrow\lambda^*$.
\end{theorem}

Hence, the limiting (scaled) queue length vector can be written as a linear combination of the random vectors $\boldsymbol{P}_1,\dots,\boldsymbol{P}_K$ as in~\eqref{eq:P_k_S_random} with scalar coefficients given by $K$ independent and exponentially distributed random variables with unit mean. 
This means that the limiting vector of (scaled) queue lengths lives in a $K$-dimensional cone spanned by these random vectors $\boldsymbol{P}_1,\dots,\boldsymbol{P}_K$.

\begin{figure}[h]
\centering
\begin{subfigure}[b]{0.49\textwidth}
\centering
\includegraphics[scale=0.8]{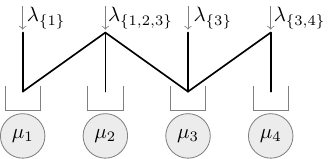}
\caption{Sketch of the system.}\label{fig:vb_algemeen_resultaat}
\end{subfigure}
\begin{subfigure}[b]{0.49\textwidth}
\centering
\includegraphics[scale = 0.8]{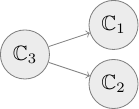}
\subcaption{Associated DAG.}\label{fig:DAG_algemeen_resultaat}
\end{subfigure}
\caption{An illustrative example with model parameters as defined in Example~\ref{ex:example_main_result}.}\label{fig:general_result}
\end{figure}

\begin{example} \label{example:mixture_dist} 
Consider the four-server system as depicted in Figure~\ref{fig:vb_algemeen_resultaat} with  arrival rates 
\begin{equation}\label{eq:example_arrival rates}
N\lambda(p_{\{1\}}, p_{\{1,2,3\}},p_{\{3\}}, p_{\{3,4\}}) = 4\lambda\Bigl(\frac{1}{4},\frac{1}{4},\frac{1}{6},\frac{1}{3}\Bigl) = \Bigl(\lambda,\lambda,\frac{2}{3}\lambda,\frac{4}{3}\lambda\Bigl)
\end{equation}
and identical processing speeds per server $(\mu_n\equiv\mu)$. Then the critical arrival rate $\lambda^* = \mu$ and the set of critical subsets is given by
$
\mathcal{CR}(\mathcal{S}) = \left\{
\{1\}, \{\{3\},\{3,4\}\}, \{\{1\},\{3\},\{3,4\}\}, \mathcal{S}
\right\}.
$
There are four critical vectors in $\mathcal{N}_K$, with $K=3$, which are presented in the first column of Table~\ref{tab:mixture_dist}. The associated mixture weights are computed using the expression in~\eqref{eq:beta}, and the various mixture components of the limiting random vector 
$(X_S)_{S\in\mathcal{S}}$ are obtained from~\eqref{eq:X_s}.
\end{example}

\begin{table}[h]
\centering
\begin{tabular}{|c|c|c|}
\hline
$\boldsymbol{T\in\mathcal{N}_K}$               & \textbf{mixt. weight} & \textbf{mixt. component of} $\boldsymbol{\left( X_{\{1\}}, X_{\{1,2,3\}},X_{\{3\}},X_{\{3,4\}}\right)}$   \\
\hline\hline
$\left[ \{1\},\{3\},\{3,4\},\{1,2,3\}\right]$ &  $4/9$            & \multirow{2}{*}{$\left( U_1 + \frac{1}{3}U_2+\frac{1}{4}U_3, \frac{1}{4}U_3, \frac{2}{9}U_2+\frac{1}{6}U_3, \frac{4}{9}U_2+ \frac{1}{3}U_3\right)$} \\
 $\left[ \{1\},\{3,4\},\{3\},\{1,2,3\}\right]$                                              &   $2/9$       &                   \\
 \hline 
$\left[\{3\},\{3,4\},\{1\},\{1,2,3\}\right]$                                                                                            &   $2/9$   & \multirow{2}{*}{$\left( \frac{1}{3}U_2+\frac{1}{4}U_3, \frac{1}{4}U_3, \frac{1}{3}U_1 + \frac{2}{9}U_2+\frac{1}{6}U_3, \frac{2}{3}U_1 + \frac{4}{9}U_2+ \frac{1}{3}U_3\right)$} \\
$\left[ \{3,4\},\{3\},\{1\},\{1,2,3\}\right]$                                                                                                                                          & $1/9$  &   \\
\hline                
\end{tabular}
\caption{The limiting mixture distribution $(X_S)_{S\in\mathcal{S}}$ in Example~\ref{example:mixture_dist} obtained after applying Theorem~\ref{cor:mixture_dist}, with $U_1, U_2$ and $U_3$ independent exponentially distributed random variables with unit mean.}\label{tab:mixture_dist}
\end{table}

Moreover, it follows that the total number of jobs converges to a random variable with an Erlang distribution with parameters 1 and $K$.

\begin{corollary}[Convergence of  the total number of jobs]\label{cor:total_num_jobs}
With the definitions as in Section~\ref{sec:preliminaries}, the PGF of the \textup{(}scaled\textup{)} total number of jobs, i.e., $\mathbb{E}\left[z^Q\right]$ with $z\coloneqq \exp( - (1-\frac{\lambda}{\lambda^*}) t)$, converges to $( 1+ t)^ {-K}$
as $\lambda\uparrow\lambda^*$ with $t\ge 0$. Hence, the \textup{(}scaled\textup{)} total number of jobs, $(1-\frac{\lambda}{\lambda^*})Q$, converges in distribution to a random variable with an Erlang distribution with parameters 1 and $K$.
\end{corollary}

The proofs of Proposition~\ref{th:limit_gf}, Theorem~\ref{cor:mixture_dist} and Corollary~\ref{cor:total_num_jobs} are deferred to Subsection~\ref{subsec:proof_part1}. \\

The complexity of the expressions for the mixture weights and mixture components makes it difficult in general to get grip on the system's performance in heavy traffic. 
Careful consideration of Example~\ref{example:mixture_dist} already suggests that the stochastic representation for $(X_S)_{S\in\mathcal{S}}$ in~\eqref{eq:X_s} can be simplified.
Indeed, two pairs of ordered vectors yield the same mixture component, so that the original mixture distribution in Table~\ref{tab:mixture_dist} can be rewritten as a distribution with only two mixture components and the adapted mixture weights $\frac{2}{3}$ and $\frac{1}{3}$, respectively.
This notion will be formalized later in Lemma~\ref{lem:mixture_components} in Subsection~\ref{subsec:proof_part2}.
In fact, building on the properties of the associated DAG, we can even show that the intricate mixture distribution in~\eqref{eq:X_s} is equivalent to a non-mixture distribution as specified later. In order to demonstrate that, we introduce the following concepts related to the associated DAG.  

\begin{definition}[Subgraphs and roots of the DAG]
For any $k=1,\dots,K$, consider the subgraph of the associated DAG rooted at the critical component $\mathbb{C}_k$ and let $V_k$ denote the collection of all job types that occur in this rooted subgraph.
Then
$
p(V_k) \coloneqq \sum_{S\in V_k} p_S
$ 
 represents the sum of the arrival probabilities of all job types that occur in $V_k$. 
 Let $\mathcal{R}_K$ denote the set of the root nodes of the DAG.
\end{definition}

Note that $\mathcal{R}_K$ is always non-empty as the associated DAG has no (directed) cycles by definition \cite[Lemma 2]{Afeche2021}.

The stochastic representation for $(X_S)_{S\in\mathcal{S}}$ in~\eqref{eq:X_s} can be simplified under the following assumption on the structure of associated DAG.

\begin{assumption}[Root partitioning of DAG] \label{cond:root}
 For any $r\in \mathcal{R}_K$, define $V_r$ as the set of nodes of the subgraph of the associated DAG with root node $r$. Assume that the root nodes yield a disjoint partitioning of the nodes in the DAG, i.e., for any $\mathcal{C}_k$ with $k=1,\dots,K$, there is precisely one $r\in \mathcal{R}_K$ such that $\mathcal{C}_k \in V_r$.
\end{assumption}

It is worth emphasizing that this assumption is satisfied in a wide range of systems and non-CRP scenarios, including but certainly not restricted to, so-called nested systems as also studied in~\cite{Anton2020,Gardner2017scheduling,Gardner2020}. Indeed, the compatibility constraints of the system in Example~\ref{example:mixture_dist} do not yield a nested system (Figure~\ref{fig:vb_algemeen_resultaat}), yet the associated DAG satisfies Assumption~\ref{cond:root} (Figure~\ref{fig:DAG_algemeen_resultaat}). A further discussion is provided in Appendix~\ref{app:root_partition}.

The simplified result then states that the vector of (scaled) queue lengths converges in distribution to a random vector where the entries are linear combinations of $K$ vectors $\boldsymbol{\hat{p}}_1,\dots,\boldsymbol{\hat{p}}_K$ whose (deterministic) entries are determined by (the subgraphs of) the DAG associated with the system. In particular, define 
\begin{equation}
\label{eq:prob_type_rooted_tree}
\left(\hat{p}_{k,S} \right)_{S\in\mathcal{S}}\coloneqq \Bigl( \frac{p_S}{p(V_k)}\mathds{1}\left\{S\in V_k \right\} \Bigl)_{S\in\mathcal{S}}
\end{equation}
for all $k=1,\dots,K$. The scalar coefficients of the linear combination are given by $K$ independent and exponentially distributed random variables with unit mean.
Thus, the joint queue length vector exhibits a \emph{state space collapse} onto a $K$-dimensional cone spanned by the vectors $\boldsymbol{\hat{p}}_1,\dots,\boldsymbol{\hat{p}}_K$, which is one-dimensional only if $K=1$, i.e., the CRP condition is satisfied.

\begin{theorem}[Convergence to a non-mixture distribution]\label{th:general_results}
If Assumption~\ref{cond:root} holds, then the vector of \textup{(}scaled\textup{)} queue lengths, $\bigl(1-\frac{\lambda}{\lambda^*}\bigl) \left( Q_S\right)_{S\in\mathcal{S}}$, converges in distribution to $( Y_S)_{S\in\mathcal{S}}$, i.e., 
\begin{equation}
\Bigl(1-\frac{\lambda}{\lambda^*}\Bigl)( Q_S)_{S\in\mathcal{S}}
\overset{d}{\rightarrow}
( Y_S)_{S\in\mathcal{S}},
\end{equation}
as $\lambda\uparrow \lambda^*$, with
\begin{equation}\label{eq:general_result_RV}
( Y_S)_{S\in\mathcal{S}}
\overset{d}{=}
\Bigl(\sum\limits_{k=1}^K \hat{p}_{k,S}   U_k \Bigl)_{S\in\mathcal{S}},
\end{equation}
$\hat{p}_{k,S}$ as defined in~\eqref{eq:prob_type_rooted_tree} and $U_1,\dots,U_K$ independent and exponentially distributed random variables with unit mean. 
\end{theorem}

As already established in Corollary~\ref{cor:total_num_jobs}, the (scaled) total number of jobs converges to a random variable with an Erlang distribution with parameters 1 and $K$ since $\sum_{S\in\mathcal{S}} \hat{p}_{k,S} = 1$ for all $k=1,\dots,K$, and hence
\begin{equation}
\Bigl(1-\frac{\lambda}{\lambda^*}\Bigl)\sum\limits_{S\in\mathcal{S}} Q_S
\overset{d}{\rightarrow}
\sum\limits_{k=1}^K  U_k.
\end{equation}

We now provide an example to illustrate the simplified result. A further discussion of Theorem~\ref{th:general_results} is provided in the following subsection.
\begin{example} \label{ex:example_main_result}
For the system described in Example~\ref{example:mixture_dist}, the associated DAG has edges $(\mathbb{C}_3,\mathbb{C}_1)$ and $(\mathbb{C}_3,\mathbb{C}_2)$ with $\mathbb{C}_1 = (\{1\},\{1\})$, $\mathbb{C}_2 = (\{\{3\},\{3,4\}\},\{3,4\})$ and $\mathbb{C}_3 = (\{1,2,3\},\{2\})$, see Figure~\ref{fig:DAG_algemeen_resultaat}. Hence, the job types belonging to nodes of the subgraphs rooted at the various CRP components are given by 
$V_1 = \{1\}, V_2 = \{\{3\},\{3,4\}\}, V_3 = \mathcal{S}$. Theorem~\ref{th:general_results} yields
\begin{equation}
\begin{array}{rl}
& \bigl(1-\frac{\lambda}{\mu} \bigl) \left( Q_{\{1\}}, Q_{\{1,2,3\}},Q_{\{3\}}, Q_{\{3,4\}} \right)\\
\overset{d}{\rightarrow}
&
\bigl( 
U_1 + U_3 p_{\{1\}},
 U_3 p_{\{1,2,3\}},
  U_2 \frac{p_{\{3\}}}{p_{\{3\}}+p_{\{3,4\}}}+ U_3 p_{\{3\}},
   U_2 \frac{p_{\{3,4\}}}{p_{\{3\}}+p_{\{3,4\}}}+ U_3p_{\{3,4\}}
   \bigl)\\
 = &  \left(U_1 + \frac{1}{4}U_3,\frac{1}{4}U_3, \frac{1}{3}U_2 + \frac{1}{6}U_3 , \frac{2}{3}U_2 + \frac{1}{3}U_3 \right)
\end{array}
\end{equation}
as $\lambda\uparrow\mu$ with $U_1,U_2,U_3$ independent and exponentially distributed random variables with unit mean.
This example highlights that the simplified result in Theorem~\ref{th:general_results} is \emph{not} obtained by showing that $\boldsymbol{P}_k = \boldsymbol{\hat{p}}_k$ for all $k$ under Assumption~\ref{cond:root}, but rather relies on subtle relationships between the model parameters and the properties of independent and exponentially distributed random variables. Indeed, observe that $\boldsymbol{\hat{p}}_1 = [1,0,0,0]$, while $\boldsymbol{P}_1 = [1,0,0,0]$ with probability~$\frac{2}{3}$ or $[0,0,\frac{1}{3},\frac{2}{3}]$ with probability~$\frac{1}{3}$.
In fact, the subscript $k$ in~\eqref{eq:P_k_S_random} specifically represents the depths of critical subsets induced by the ordered vectors $\boldsymbol{T}$, while the subscript $k$ in~\eqref{eq:prob_type_rooted_tree} corresponds to the different critical components which can be labeled arbitrarily.
\end{example}

The result in Theorem~\ref{th:general_results} is established by proving that the Laplace transforms of $(X_S)_{S\in\mathcal{S}}$ and $(Y_S)_{S\in\mathcal{S}}$ coincide after simplification, as formalized in the next proposition.

\begin{proposition}\label{prop:equality_GF}
For $t_S\ge 0$ for all $S\in\mathcal{S}$, it holds under Assumption~\ref{cond:root} that
\begin{equation}
\mathbb{E}\Bigr[ \prod\limits_{S\in\mathcal{S}} \mathrm{e} ^{-t_S X_S}\Bigr] = \mathbb{E}\Bigr[ \prod\limits_{S\in\mathcal{S}} \mathrm{e} ^{-t_S Y_S}\Bigr],
\end{equation}
with 
the former Laplace transform as in~\eqref{eq:GF} and 
\begin{equation}\label{eq:equality_GF}
\mathbb{E}\Bigr[ \prod\limits_{S\in\mathcal{S}} \mathrm{e} ^{-t_S Y_S}\Bigr] 
= \prod\limits_{k=1}^K \Bigl( 1+ \sum\limits_{S\in V_k} t_S \frac{p_S}{p(V_k)}\Bigl)^{-1},
\end{equation}
yielding $(X_S )_{S\in\mathcal{S}} \overset{d}{=} (Y_S )_{S\in\mathcal{S}}.$
\end{proposition}
The proof of Proposition~\ref{prop:equality_GF} is deferred to Subsection~\ref{subsec:proof_part2}.

To summarize, the proof of Theorem~\ref{th:general_results} consists of two major parts, as depicted at the top of the diagram in Figure~\ref{fig:diagram_results}. 
We first investigate the heavy-traffic limit of the PGF in Proposition~\ref{prop:pgf} (Proposition~\ref{th:limit_gf}). From this we can deduce that the (scaled) numbers of jobs of the various types converge to a random vector associated with a mixture distribution (Theorem~\ref{cor:mixture_dist}).
Second, we show that the Laplace transform of this mixture distribution under Assumption~\ref{cond:root} coincides with the Laplace transform of the random vector $(Y_S)_{S\in\mathcal{S}}$ as in~\eqref{eq:general_result_RV} (Proposition~\ref{prop:equality_GF}). The result in Theorem~\ref{th:general_results} then immediately follows, i.e., the (scaled) numbers of jobs of the various types converge to $(Y_S)_{S\in\mathcal{S}}$.

As visualized in the lower part of the diagram in Figure~\ref{fig:diagram_results}, we also provide an alternative proof of the heavy-traffic result in Theorem~\ref{cor:mixture_dist} in Appendix~\ref{sec:alternative proof}.
This alternative reasoning builds on the stochastic interpretation of the PGF in Proposition~\ref{prop:pgf}. In particular, we first derive a pre-limit characterization of the numbers of jobs of the various types in terms of weighted sums of geometrically distributed random variables whose parameters depend on the model parameters and the order in which the different job types occur. Subsequently we investigate the heavy-traffic behavior of each of these geometric random variables.

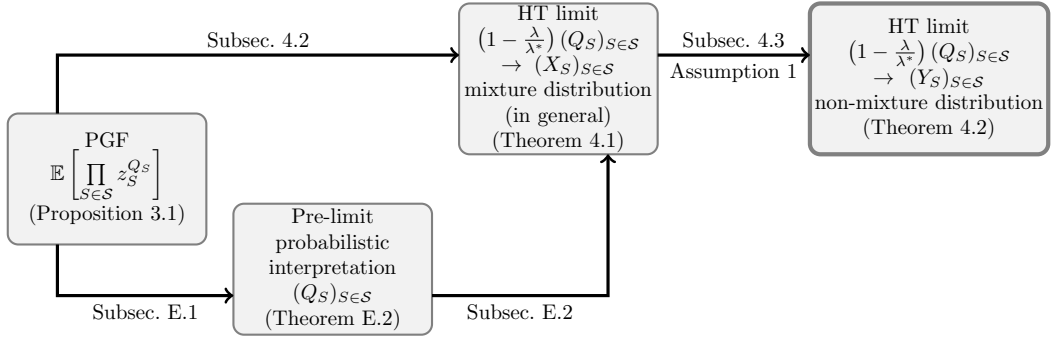
\begin{figure}
\centering
\addtolength{\leftskip} {-2cm}
\addtolength{\rightskip}{-2cm}
\begin{tikzpicture}[transform shape, scale = 0.8]
\def \k {5.3}
\def \l {7.5}
\def \m {2.5}
\def \rect {3.3}

\begin{scope}[shift = {(-\l,0)}]
\draw[rounded corners, draw = gray, thick, fill = gray!10] (0,0) rectangle (\rect,0.66*\rect) node[align = center, pos=.5,text width=\rect cm] {PGF\\ $\mathbb{E}\left[\prod\limits_{S\in\mathcal{S}} z_S^{Q_S}\right]$ (Proposition~\ref{prop:pgf})};
\end{scope}

\begin{scope}[shift = {(-0.5*\l,-0.6*\m)}]
\draw[rounded corners, draw = gray, thick, fill = gray!10] (0,0) rectangle (\rect,0.66*\rect) node[align = center, pos=.5,text width=\rect cm] {Pre-limit\\ probabilistic \\ interpretation  $(Q_S)_{S\in\mathcal{S}}$ \\(Theorem~\ref{th:probabilistic_interpretation})};
\end{scope}

\begin{scope}[shift = {(0,0.6*\m)}]
\draw[rounded corners, draw = gray, thick, fill = gray!10] (0,0) rectangle (\rect,0.76*\rect) node[align = center, pos=.5,text width=\rect cm] {HT limit\\ $\left( 1 - \frac{\lambda}{\lambda^*}\right)(Q_S)_{S\in\mathcal{S}}$\\$ \rightarrow (X_S)_{S\in\mathcal{S}}$\\
mixture distribution\\ (in general)\\ (Theorem~\ref{cor:mixture_dist})};
\end{scope}

\begin{scope}[shift = {(1.1*\k,0.6*\m)}]
\draw[rounded corners, draw = gray, ultra thick, fill = gray!10] (0,0) rectangle (1.2*\rect,0.76*\rect) node[align = center, pos=.5,text width=1.2*\rect cm] {HT limit\\ $\left( 1 - \frac{\lambda}{\lambda^*}\right)(Q_S)_{S\in\mathcal{S}}$\\$ \rightarrow (Y_S)_{S\in\mathcal{S}}$\\ non-mixture distribution\\ (Theorem~\ref{th:general_results})};
\end{scope}


\draw[->, very thick] (-\l+0.25*\rect,0) -- (-\l+0.25*\rect,0.5*\rect-\m) -- (-\l+0.25*\rect,0.5*\rect-\m) -- (-0.5*\l,0.5*\rect-\m) node [below, align = center,pos=.5,sloped]{Subsec.~\ref{sec:interpretation}} ;

\draw[->, very thick] (-0.5*\l+\rect,0.5*\rect-\m) -- (0.75*\rect,0.5*\rect-\m)
 node [below, align = center,pos=.5,sloped]{Subsec.~\ref{subsec:alternative_proof_HT}}
  -- (0.75*\rect,0.6*\m) ;


\draw[->, very thick] (-\l+0.25*\rect,0.66*\rect)  -- (-\l+0.25*\rect,0.5*\rect+0.6*\m) -- (0,0.5*\rect+0.6*\m) node [above, align = center,pos=.5,sloped]{Subsec.~\ref{subsec:proof_part1}};

\draw[->, very thick] (\rect,0.5*\rect+0.6*\m) -- (1.1*\k,0.5*\rect+0.6*\m) node [above, align = center,pos=.5,sloped]{Subsec.~\ref{subsec:proof_part2}} node [below, align = center,pos=.5,sloped]{Assumption~\ref{cond:root}};
\end{tikzpicture}
\caption{Schematic representation of the derived heavy-traffic (HT) results starting from the PGF.}\label{fig:diagram_results}
\end{figure}

\subsection{Discussion}\label{subsec:discussion}
In this subsection we provide further insight into the result stated in Theorem~\ref{th:general_results}.

Depending on the compatibility constraints and the model parameters, $K$ can take any value between 1 and $|\mathcal{S}|$. We now present two examples of these extremes.

For instance, when the strong CRP condition (Condition~\ref{cond:crp}) is satisfied, the associated DAG consists of a single node (i.e., $K=1$) which includes all job types. In this special case, the result in Theorem~\ref{th:general_results} simplifies to
\begin{equation}
\Bigl(1-\frac{\lambda}{\mu}\Bigl) \left( Q_S\right)_{S\in\mathcal{S}}
\overset{d}{\rightarrow} U\left(p_S\right)_{S\in\mathcal{S}},
\end{equation}
as $\lambda\uparrow\mu$, with $U$ a unit-mean exponential random variable. Hence, the system exhibits full state space collapse. Moreover, the limiting joint distribution coincides with that of a multi-class M/M/1 queue with arrival rate $N\lambda$, service rate $N\mu$ and class probabilities $(p_S)_{S\in\mathcal{S}}$.
This agrees with the results in~\cite[Theorem~1]{Cardinaels2022} which hold under both the weak and strong CRP condition.

In contrast, consider a setting with $\mathcal{S} = \{ \{n\} \colon n=1,\dots, N \}$, identical arrival rates $\lambda_{\{n\}} \equiv \lambda$ and identical processing speeds $\mu_n \equiv \mu$, i.e., each job type is compatible with only one server and each server has just one compatible job type. In this case, all subsets of job types are critical and the associated DAG consists of $K = |\mathcal{S}|= N$ isolated nodes such that $V_k = \{k\}$ for $k=1,\dots,K$. Hence, 
\begin{equation}
\Bigl(1-\frac{\lambda}{\mu}\Bigl) \left( Q_{\{n\}}\right)_{n=1}^N
\overset{d}{\rightarrow}\left(  U_n \right)_{n=1}^N,
\end{equation}
as $\lambda\uparrow\mu$, with $U_1,\dots,U_N$ independent and exponentially distributed random variables with unit mean.
This is consistent with the well-known result that the queue length of an M($\lambda)$/M($\mu$)/1 queue, when scaled with $1-\frac{\lambda}{\mu}$, converges to an exponentially distributed random variable with unit mean as $\lambda\uparrow\mu$. The above example can of course be seen as a collection of $N$ independent M($\lambda)$/M($\mu$)/1 queues, hence the joint distribution of the (scaled) numbers of jobs of the various types converges to the distribution of $N$ independent and exponential random variables.\\

As already alluded to above, the joint queue length vector exhibits a state space collapse onto a, possibly random, $K$-dimensional cone.
Even if only a partial state space collapse occurs, i.e., $K>1$, or even when $K=|\mathcal{S}|$, the limiting system will be more convenient to analyze. So, even when $K$ is large, and hence the dimension reduction is limited, it is more manageable to study a collection of independent exponential random variables than the pre-limit queue length distribution (possibly via the product-form expressions).\\

As observed in the introduction, explicit results for steady-state queue length distributions in non-CRP scenarios as presented in Theorem~\ref{th:general_results} have remained extremely scarce so far. To the best of our knowledge, no distributional results have been established in general settings, and we are only aware of some promising and revealing advances for special model instances. In particular, Maguluri \& Jhunjhunwala~\cite{Jhunjhunwala2022,Jhunjhunwala2023} extend the transform method to non-CRP scenarios to derive an implicit characterization of the limiting queue length distribution in terms of a certain functional equation for the class of input-queued switches operating in a time-slotted fashion under a MaxWeight scheduling algorithm.  In the special case of the N-model the functional equation can be solved under additional symmetry assumptions (equal service rates, and hence equal arrival rates in a non-CRP scenario) to obtain a result of the form
\begin{equation}
\Bigl(1-\frac{\lambda}{\mu} \Bigl)\left( Q_{\{1,2\}}, Q_{\{2\}} \right) \overset{d}{\rightarrow}\Bigl( \frac{1}{2}U_1,  \frac{1}{2}U_1 + U_2 \Bigl),
\end{equation}
as $\lambda \uparrow \mu$, with $U_1$ and $U_2$ independent unit-mean exponential random variables.

While the proof techniques and specific model attributes are quite different from our framework, it is striking to observe the close resemblance in the limiting queue length distribution in this particular case.  Since the above results for the MaxWeight algorithm do not extend to heterogeneous settings or more general compatibility graphs, it is not immediately clear though to what extent the similarity is due to the symmetry assumptions or might possibly apply more broadly.\\


Finally, notice that we did not assume the compatibility graph to be connected to obtain the results above. In case of a disconnected graph it might however be notationally easier to apply Theorem~\ref{th:general_results} (or Theorem~\ref{cor:mixture_dist}) to each connected component separately. 

\subsection{Proofs of Proposition~\ref{th:limit_gf}, Theorem~\ref{cor:mixture_dist} and Corollary~\ref{cor:total_num_jobs}}
\label{subsec:proof_part1}
In the remainder of this section it is assumed that all job types belong to (at least) one of the critical subsets and hence the whole system experiences criticality, so that $\lambda^* = \mu$. This assumption mainly serves to ease the notation and is non-essential; pointers to proofs and notation that need to be adapted otherwise can be found in Appendix~\ref{app:assumption_lam_star_mu}.

The proof of Proposition~\ref{th:limit_gf}, and the proofs of some of the remaining results, will rely on the following intermediate result.
\begin{lemma}\label{lem:intermediate_result}
Let $\boldsymbol{T}\in\mathcal{N}_k$ for some $k=0,1,\dots,K$, $\boldsymbol{t} \coloneqq (t_S)_{S\in\mathcal{S}} \ge 0$ and $\boldsymbol{z} \coloneqq (z_S)_{S\in\mathcal{S}}$ with $z_S = \exp \bigl( - \bigl( 1-\frac{\lambda}{\lambda^*} \bigl) t_S \bigl)$ for all $S\in\mathcal{S}$, and define 
\begin{equation}\label{eq:def_h_T_z}
h(\boldsymbol{T},\boldsymbol{z}) \coloneqq 
\prod\limits_{j=1}^{|\boldsymbol{T}|}\frac{N\lambda p_{T_j}z_{T_j}}{\mu(\boldsymbol{T},j)} 
\Bigl( 1- \frac{N\lambda \sum_{i=1}^j p_{T_i}z_{T_i}}{\mu(\boldsymbol{T},j)}
\Bigl)^{-1}.
\end{equation}
Then
\begin{equation}
\lim\limits_{\lambda\uparrow\lambda^*}  h(\boldsymbol{T},\boldsymbol{z}) \times \Bigl( 1-\frac{\lambda}{\lambda^*} \Bigl) ^k = h^*(\boldsymbol{T},\boldsymbol{t}),
\end{equation}
with
\begin{equation}
h^*(\boldsymbol{T},\boldsymbol{t}) \coloneqq \beta (\boldsymbol{T})  \times \prod\limits_{j\in\mathrm{CR}(\boldsymbol{T})} \Bigl( 1+ \frac{\sum_{i=1}^j p_{T_i}t_{T_i}}{p(\boldsymbol{T},j)}
\Bigl)^{-1} \in (0,\infty),
\end{equation}
and $\beta(\boldsymbol{T})$ as in~\eqref{eq:beta}.
\end{lemma}
Notice that $h(\boldsymbol{T},\boldsymbol{z})$ is precisely the term in $f(\boldsymbol{z})$ in~\eqref{eq:f} for a given ordered vector of job types~$\boldsymbol{T}$ and that $\mathbb{P}^*(\boldsymbol{T}) = \beta(\boldsymbol{T})/\beta(\mathcal{N}_K)$ for any $\boldsymbol{T}\in\mathcal{N}_K$ with $\mathbb{P}^*(\boldsymbol{T})$ as in~\eqref{eq:P_star_T}.

\begin{proof}[Proof of Lemma~\ref{lem:intermediate_result}]
First, we note that we can write $h(\boldsymbol{T},\boldsymbol{z})$ as the product of 
\begin{equation}
\begin{array}{rcl}
h^1(\boldsymbol{T},\boldsymbol{z}) &\coloneqq &
\prod\limits_{j=1}^{|\boldsymbol{T}|}\frac{N\lambda p_{T_j}z_{T_j}}{\mu(\boldsymbol{T},j)}
\prod\limits_{j\notin \mathrm{CR}(\boldsymbol{T})}
\Bigl( 1- \frac{N\lambda \sum_{i=1}^j p_{T_i}z_{T_i}}{\mu(\boldsymbol{T},j)}
\Bigl)^{-1}\\
h^2(\boldsymbol{T},\boldsymbol{z}) &\coloneqq &
\prod\limits_{j\in \mathrm{CR}(\boldsymbol{T})}
\Bigl( 1- \frac{N\lambda \sum_{i=1}^j p_{T_i}z_{T_i}}{\mu(\boldsymbol{T},j)}
\Bigl)^{-1}.\\
\end{array}
\end{equation}
Now, from Definition~\ref{def:crit_vectors} we know that \emph{none} of the factors of $h^1$ will diverge when $\lambda\uparrow\lambda^*$. In fact, $\lim_{\lambda\uparrow\lambda^*}  h^1(\boldsymbol{T},\boldsymbol{z}) = \beta(\boldsymbol{T})$, with $\beta(\boldsymbol{T})$ as in~\eqref{eq:beta}.
Next, from Definition~\ref{def:crit_vectors} we know that \emph{each} of the $k$ factors of $h^2$ will diverge when $\lambda\uparrow\lambda^*$. However, after applying l'H\^opital's rule, we observe that 
\begin{equation}\label{eq:intermed_res_hopital}
\begin{array}{rcl}
\lim\limits_{\lambda\uparrow\lambda^*}  h^2(\boldsymbol{T},\boldsymbol{z})  \times \bigl( 1-\frac{\lambda}{\lambda^*} \bigl) ^k
&=& \lim\limits_{\lambda\uparrow\lambda^*} \prod\limits_{j\in \mathrm{CR}(\boldsymbol{T})}
\Bigl( 1- \frac{N\lambda \sum_{i=1}^j p_{T_i}z_{T_i}}{\mu(\boldsymbol{T},j)}
\Bigl)^{-1} \bigl( 1-\frac{\lambda}{\lambda^*} \bigl)\\
&=& \prod\limits_{j\in\mathrm{CR}(\boldsymbol{T})} \Bigl( 1+ \frac{\sum_{i=1}^j p_{T_i}t_{T_i}}{p(\boldsymbol{T},j)}
\Bigl)^{-1} \in (0,\infty).
\end{array}
\end{equation}
Combining the above two  observations concludes the proof.
\end{proof}

\begin{proof}[Proof of Proposition~\ref{th:limit_gf}]
We need to show that the PGF in~\eqref{eq:pgf} converges to \eqref{eq:GF} when $\lambda\uparrow\lambda^*$ and when the number of jobs of the various types is scaled by $(1-\frac{\lambda}{\lambda^*})$. 
First, we rewrite the function~$f$ in~\eqref{eq:f} by rearranging the terms according to the number of critical sets of job types they correspond to. In particular,
\begin{equation}\label{eq:rewrite_f}
f(\boldsymbol{z}) = \sum\limits_{k=0}^K  f_k(\boldsymbol{z})
\quad \text{with} \quad
f_k(\boldsymbol{z}) = \sum\limits_{\boldsymbol{T}\in\mathcal{N}_k} 
h(\boldsymbol{T},\boldsymbol{z})
\end{equation}
and $h(\boldsymbol{T},\boldsymbol{z})$ as in~\eqref{eq:def_h_T_z}.
Note that we may write 
\begin{equation}
 \mathbb{E}\Bigr[\prod\limits_{S\in\mathcal{S}}z_S^{Q_S} \Bigr]
= \dfrac{\sum\limits_{k=0}^K  f_k(\boldsymbol{z})}{\sum\limits_{k=0}^K  f_k(\boldsymbol{1})}\\
= \dfrac{f_K(\boldsymbol{z})}{f_K(\boldsymbol{1})} \dfrac{1+\sum\limits_{k=0}^{K-1}\dfrac{f_k(\boldsymbol{z})}{f_K(\boldsymbol{z})}}{1+\sum\limits_{k=0}^{K-1}\dfrac{f_k(\boldsymbol{1})}{f_K(\boldsymbol{1})}}.
\end{equation}
We will first show that $f_k(\boldsymbol{z})/ f_K(\boldsymbol{z})$ converges to 0 as $\lambda\uparrow\lambda^*$ for all $k=0,\dots,K-1$ and $t_S\ge 0$ for all $S\in\mathcal{S}$. Then we prove that $f_K(\boldsymbol{z})/ f_K(\boldsymbol{1})$ indeed tends to~\eqref{eq:GF}. 
Recalling that $z_S \coloneqq \exp\left(-(1-\lambda/\lambda^*)t_S\right)$ for all $S\in\mathcal{S}$ and $t_S\ge 0$, 
this concludes the proof.

\textbf{Part 1:} \textit{Show that for all $k=0,\dots,K-1$ and $z_S \coloneqq \exp\left(-(1-\lambda/\lambda^*)t_S\right)$ with $t_S \ge 0$,} 
$$\lim_{\lambda\uparrow\lambda^*} f_k(\boldsymbol{z})/f_K(\boldsymbol{z}) = 0.$$
Due to Lemma~\ref{lem:intermediate_result}, we know that both $(1-\frac{\lambda}{\lambda^*})^k f_k(\boldsymbol{z})$ and $(1-\frac{\lambda}{\lambda^*})^K f_K(\boldsymbol{z})$ converge to constants as $\lambda\uparrow\lambda^*$. Hence, 
\begin{equation}
\Bigl(1-\frac{\lambda}{\lambda^*}\Bigl)^k f_K(\boldsymbol{z}) = \Bigl(1-\frac{\lambda}{\lambda^*}\Bigl)^{k-K} \biggl( \Bigl(1-\frac{\lambda}{\lambda^*}\Bigl)^K f_K(\boldsymbol{z})\biggl)
\end{equation}
diverges as $k<K$. We conclude that 
\begin{equation}\label{eq:limit_f_k}
\lim\limits_{\lambda\uparrow\lambda^*} \frac{f_k(\boldsymbol{z})}{f_K(\boldsymbol{z})} 
= \lim\limits_{\lambda\uparrow\lambda^*} \frac{\bigl(1-\frac{\lambda}{\lambda^*}\bigl)^k f_k(\boldsymbol{z})}{\bigl(1-\frac{\lambda}{\lambda^*}\bigl)^K f_K(\boldsymbol{z})} 
\times \lim\limits_{\lambda\uparrow\lambda^*} \Bigl(1-\frac{\lambda}{\lambda^*}\Bigl)^{K-k} 
=0.
\end{equation} 

\textbf{Part 2:} \textit{Show that for all $z_S \coloneqq \exp\left(-(1-\lambda/\lambda^*)t_S\right)$ with $t_S \ge 0$,}
\begin{equation}\label{eq:part2}
\lim\limits_{\lambda\uparrow\lambda^*} \frac{f_K(\boldsymbol{z})}{f_K(\boldsymbol{1})} = 
\sum\limits_{\boldsymbol{T}\in\mathcal{N}_K} \frac{\beta(\boldsymbol{T})}{ {\beta}(\mathcal{N}_K)} \prod\limits_{i\in \text{CR}(\boldsymbol{T})}\biggl(1+ \sum_{j=1}^i t_{T_j}\frac{p_{T_j}}{p(\boldsymbol{T},i)}\biggl)^{-1}
.
\end{equation} 
We will again focus on the function $(1-\frac{\lambda}{\lambda^*})^K f_K(\boldsymbol{z})$.
From Lemma~\ref{lem:intermediate_result}, we know that
\begin{equation}\label{eq:theorem_part2}
\lim\limits_{\lambda\uparrow\lambda^*} 
\Bigl(1-\frac{\lambda}{\lambda^*}\Bigl)^K f_K(\boldsymbol{z})
= \sum\limits_{\boldsymbol{T}\in\mathcal{N}_{K}}
\beta(\boldsymbol{T}) 
\prod\limits_{i\in\text{CR}(\boldsymbol{T})} 
 \Bigl(
1+\frac{N\lambda^*}{\mu(\boldsymbol{T},i)}\sum\limits_{j=1}^i p_{T_j}t_{T_j}
\Bigl)^{-1}.
\end{equation}
Note that the latter product evaluates as 1 if $\boldsymbol{z} = \boldsymbol{1}$, or alternatively if $\boldsymbol{t} = \boldsymbol{0}$, such that 
\begin{equation}
\lim\limits_{\lambda\uparrow\lambda^*} 
\Bigl(1-\frac{\lambda}{\lambda^*}\Bigl)^K f_K(\boldsymbol{1})
= {\beta}(\mathcal{N}_K).
\end{equation} 

This concludes the proof of Proposition~\ref{th:limit_gf}.
\end{proof}

\begin{proof}[Proof of Theorem~\ref{cor:mixture_dist}]
Considering~\eqref{eq:GF}, we notice that there is a weight $\mathbb{P}^*(\boldsymbol{T})$ associated with each vector $\boldsymbol{T}\in\mathcal{N}_K$.
Now, focusing on a fixed $\boldsymbol{T}\in\mathcal{N}_K$, we observe that
\begin{equation}
\prod\limits_{i\in\mathcal{CR}(\boldsymbol{T})}\Bigl(1+ \sum_{j=1}^i t_{T_j}\frac{p_{T_j}}{p(\boldsymbol{T},i)}\Bigl)^{-1} = \prod\limits_{k=1}^K \left( 1+ x_k(\boldsymbol{T}) \right)^{-1} = \prod\limits_{k=1}^K \mathbb{E}\left[ \mathrm{e}^{-U_k x_k(\boldsymbol{T}) } \right],
\end{equation} 
with $x_k(\boldsymbol{T}) \coloneqq  \sum_{j=1}^{i_k}t_{T_j}p_{T_j}/p(\boldsymbol{T},i_k)$. Moreover, we note that
\begin{equation}
\sum\limits_{k=1}^K U_k x_k(\boldsymbol{T}) = \sum\limits_{j=1}^{|\boldsymbol{T}|} t_{T_j} \sum\limits_{\substack{k=1,\dots,K \\ i_k \ge j}} U_k \frac{p_{T_j}}{p(\boldsymbol{T},i_k)} = \sum\limits_{S\in\mathcal{S}} t_S \sum\limits_{\substack{k=1,\dots,K \\ i_k \ge j_S(\boldsymbol{T})}} U_k \frac{p_{S}}{p(\boldsymbol{T},i_k)},
\end{equation}
where we altered the summation order in the last equality from $T\in\boldsymbol{T}$ to $S\in\mathcal{S}$ by keeping track of the position of job type $S$ in the $K$-critical vector $\boldsymbol{T}$ via $j_S(\boldsymbol{T})$. Note that the above expression is still valid if $S \notin \boldsymbol{T}$, in that case we can assume that $j_S(\boldsymbol{T}) = \infty$ and hence the corresponding inner sum evaluates to 0.
The convergence result then follows from L\'evy's Continuity Theorem~\cite{Jacod2000}.
\end{proof}

\begin{remark}
In case $K=1$, the statement of Theorem~\ref{cor:mixture_dist} coincides with the results in~\cite[Theorem~2]{Cardinaels2022}. The present paper not only allows for a more general setting (i.e., $K \ge 1$), but it also significantly reduces the complexity of the proof by, first, enumerating the vectors of job types according to the number of critical subsets ($\boldsymbol{T}\in \mathcal{N}_k$) instead of the length ($\boldsymbol{T}\in \mathcal{S}_k$), and second, by multiplying the numerators and denominator with $(1-\lambda/\lambda^*)^K$ (as in~\eqref{eq:limit_f_k}) instead of \emph{extracting} the possibly diverging components of the limiting PGF.
\end{remark}

\begin{proof}[Proof of Corollary~\ref{cor:total_num_jobs}]
Substitute $t_S = t$ for all $S\in\mathcal{S}$ in the expression obtained in Proposition~\ref{th:limit_gf}, and observe that $|\text{CR}(\boldsymbol{T})| = K$ for all $\boldsymbol{T}\in\mathcal{N}_K$. The simplification to $(1+t)^{-K}$ follows after straightforward manipulations.

We observe that this limiting PGF of the (scaled) total number of jobs coincides with that of a sum of $K$ independent exponential random variables with unit mean, or in other words the Laplace transform of a random variable with an Erlang distribution. Hence, L\'evy's Continuity Theorem~\cite{Jacod2000} implies that the non-negative random variable $(1-\frac{\lambda}{\lambda^*})Q$ converges in distribution to a random variable with an Erlang distribution and parameters $1$ and $K$.
\end{proof}

\subsection{Proof of Proposition~\ref{prop:equality_GF}}
\label{subsec:proof_part2}

In order to prove Proposition~\ref{prop:equality_GF} we will use the fact that some of the components of the mixture distribution of $(X_S)_{S\in\mathcal{S}}$ are the same and hence the corresponding mixture weights can be aggregated. This, together with further simplification, will result in an equivalent representation of the Laplace transform as in~\eqref{eq:equality_GF}.

Indeed, considering Example~\ref{example:mixture_dist} it can be seen that the $K$-ordered vectors 
$$\boldsymbol{T} = \left[ \{1\},\{3\},\{3,4\},\{1,2,3\} \right]\text{~and~}\boldsymbol{T'} = \left[ \{1\},\{3,4\},\{3\},\{1,2,3\} \right]$$ 
both induce the same mixture component of the heavy-traffic limit.
This is due to the fact that $\text{CR}(\boldsymbol{T}) = \text{CR}(\boldsymbol{T}') = \{1,3,4\}$ and the only difference between $\boldsymbol{T}$ and $\boldsymbol{T}'$ is a permutation of the job types that belong to the same CRP component, i.e., $\mathcal{C}_2 = \{\{3\},\{3,4\}\}$. A similar observation can be made for the remaining two $K$-ordered vectors in $\mathcal{N}_K$.

We first elaborate on the various mixture components and how the corresponding mixture weights can be aggregated.

\begin{definition}[$\sigma$-ordered vectors]\label{def:sigma_critcal_vectors}
Let $\sigma\in\Sigma_K$ be a topological ordering, then we define a set of ordered vectors as follows, $\boldsymbol{T}^{\sigma} \coloneqq \left\{ \boldsymbol{T} = \left[ \mathrm{perm}(\mathcal{C}_{\sigma(1)}),\dots, \mathrm{perm}(\mathcal{C}_{\sigma(K)}) \right] \right\}$,
where $\mathrm{perm}(\mathcal{C}_{k})$ denotes any permutation of the job types in the critical component $\mathbb{C}_k$. We refer to the vectors in~$\boldsymbol{T}^{\sigma}$ as $\sigma$-ordered vectors.
\end{definition}

\begin{example}
\label{exam:sigma_ordered_vectors}
In Example~\ref{ex:example_main_result} there are two possible topological orderings of the CRP components, i.e., $\Sigma_K = \left\{ (1,2,3),(2,1,3) \right\}$. With the CRP components defined as $\mathbb{C}_1 = (\{1\},\{1\})$, $\mathbb{C}_2 = (\{\{3\},\{3,4\}\},\{3,4\})$ and $\mathbb{C}_3 = (\{1,2,3\},\{2\})$, this results in the following two sets of $\sigma$-ordered vectors,
\begin{equation}
\begin{array}{rcl}
\boldsymbol{T}^{(1,2,3)} & = & \left\{ [\{1\},\{3\},\{3,4\},\{1,2,3\}], [\{1\},\{3,4\},\{3\},\{1,2,3\}]\right\}, \\
\boldsymbol{T}^{(2,1,3)} & = & \left\{ [\{3\},\{3,4\},\{1\},\{1,2,3\}], [\{3,4\},\{3\},\{1\},\{1,2,3\}]\right\}.
\end{array}
\end{equation}
\end{example}

Next, we argue that the above construction for all $\sigma\in\Sigma_K$ will yield a partitioning of the set of $K$-critical vectors, $\mathcal{N}_K$.

\begin{lemma} \label{lem:N_K_versus_Sigma}
It holds that
 $\mathcal{N}_K = \bigcup_{\sigma\in\Sigma_K} \boldsymbol{T}^{\sigma}$, 
with $\mathcal{N}_K$ as in Definition~\ref{def:crit_vectors} and $\boldsymbol{T}^{\sigma}$ as in Definition~\ref{def:sigma_critcal_vectors}.
\end{lemma}

\begin{proof}
Consider $\boldsymbol{T} \in \boldsymbol{T}^{\sigma}$ for some $\sigma\in\Sigma_K$. We have to show that $\boldsymbol{T}$ induces $K$ critical subsets when combining the respective job types one by one.
Due to Construction~\ref{construction}, we know that $T_{\sigma,1} \subsetneq T_{\sigma,2}\subsetneq \dots \subsetneq T_{\sigma,K}$, where 
 $T_{\sigma,k} = \bigcup_{i=1}^k \mathcal{C}_{\sigma(i)}$ 
are critical subsets of job types for all~$k$. Since, by definition of $K$, there cannot be more than $K$ induced critical subsets in any ordered vector in $\mathcal{N}$, we conclude that $\boldsymbol{T}\in\mathcal{N}_K$.

Now assume that $\boldsymbol{T} \in\mathcal{N}_K$ with $\mathrm{CR}(\boldsymbol{T}) = \{i_1,\dots,i_K\}$. Define $\mathcal{T}_k
\coloneqq \{T_1,\dots,T_{i_k} \}$ for $k=1,\dots,K$ and $\mathcal{T}_0 \coloneqq \emptyset$ to ease the notation. Since the critical components form a partitioning of the job types and all the critical subsets of job types $\mathcal{T}_1,\dots,\mathcal{T}_K$ can be written as a union of the critical components, we have that (after possibly relabeling the critical components) $\mathcal{T}_{k}\setminus\mathcal{T}_{k-1} = \mathcal{C}_k$ for $k=1,\dots,K$. 

We claim that $(\mathbb{C}_1,\dots,\mathbb{C}_K)$ is indeed a topological ordering and hence there exists a $\sigma\in\Sigma_K$ such that $\boldsymbol{T}\in\boldsymbol{T}^{\sigma}$. If the critical component $\mathbb{C}_i$ can forward some of its jobs to servers in critical component $\mathbb{C}_j$, then $\mathbb{C}_j$ must occur before $\mathbb{C}_i$ in the topological ordering. Assume by contradiction that $\mathbb{C}_i$ is positioned before $\mathbb{C}_j$ in the above-mentioned ordering. This implies that there is a job type~$S\in\mathcal{C}_i$ compatible with server $n\in\mathcal{Z}_j$. By construction, we have that 
$\mathcal{Z}_k \subseteq \cup_{T\in\mathcal{C}_k} T$, and from \cite[Lemma~3(i)]{Afeche2021} we know that $N\mu p\left(\mathcal{C}_k\right) = \mu\left(\mathcal{Z}_k\right)$ for all $k$, so that 
\begin{equation}
\mu(\mathcal{T}_i) = N\mu p(\mathcal{T}_i) = N\mu\sum\limits_{k=1}^ip(\mathcal{C}_k) = \sum\limits_{k=1}^i \mu(\mathcal{Z}_k) < \mu(\mathcal{T}_i),
\end{equation}
as server $n$ is not contained in $\mathcal{Z}_i$ (or $\mathcal{Z}_1,\dots,\mathcal{Z}_{i-1}$). 
This contradicts the criticality of the set $\mathcal{T}_i$. Thus, the above ordering is indeed a topological ordering. This concludes the proof.
\end{proof}

\begin{lemma}\label{lem:mixture_components}
The mixture components associated with the ordered vectors governed by the same topological ordering, i.e., all $\boldsymbol{T}\in\boldsymbol{T}^{\sigma}$ for some $\sigma\in\Sigma_K$, are the same.
\end{lemma}

\begin{proof}
Note that for all $\boldsymbol{T}\in\boldsymbol{T}^{\sigma}$ we have $\mathrm{CR}(\boldsymbol{T}) = \left\{ |\mathcal{C}_{\sigma(1)}|, |\mathcal{C}_{\sigma(1)}|+|\mathcal{C}_{\sigma(2)}|, \dots, \sum_{k=1}^K |\mathcal{C}_{\sigma(k)}| \right\}$. To ease the notation, write $ \mathrm{CR}(\boldsymbol{T}) = \{i_1,\dots,i_K\}$ with $i_0 = 0 < i_1< \dots < i_K$.

Consider the result and notation of Theorem~\ref{cor:mixture_dist} and focus on any job type $S\in\mathcal{S}$. Then there exists some $k'=0,1,\dots,K$ such that $i_{k'-1} < j_S(\boldsymbol{T}) \le i_{k'}$ for all  $\boldsymbol{T}\in\boldsymbol{T}^{\sigma}$. Moreover, note that $p(\boldsymbol{T},i_k) = \sum_{j=1}^k p(\mathcal{C}_{\sigma(j)})$ for all $\boldsymbol{T}\in\boldsymbol{T}^{\sigma}$.  Thus,
\begin{equation}
X_S \mid \boldsymbol{T}
\overset{d}{=} \sum\limits_{\substack{k=1,\dots,K \\ i_k \ge j_{S}(\boldsymbol{T})}} U_k \frac{p_{S}}{p(\boldsymbol{T},i_k)}
=\sum\limits_{k=k'}^{K} U_k \frac{p_{S}}{p(\boldsymbol{T},i_k)}
=\sum\limits_{k=k'}^{K} U_k \frac{p_{S}}{\sum_{j=1}^k p(\mathcal{C}_{\sigma(j)})},
\end{equation}
which no longer depends on the actual ordering in $\boldsymbol{T}$ but only on $\sigma\in\Sigma_K$.  This concludes the proof.
\end{proof}

As a consequence of the above two lemmas, if the topological ordering is unique, i.e., $|\Sigma_K|=1$ and hence the associated DAG is a line graph, then the representation of $(X_S)_{S\in\mathcal{S}}$ in~\eqref{eq:X_s} can easily be rewritten to obtain the non-mixture distribution of $(Y_S)_{S\in\mathcal{S}}$ in~\eqref{eq:general_result_RV}.
In all other settings, this simplification is not so straightforward.


Next we investigate the mixture weights. We will use the following
identity in order to aggregate the mixture weights $\left(\mathbb{P}^*(\boldsymbol{T})\right)_{\boldsymbol{T}\in\mathcal{N}_K}$ in~\eqref{eq:P_star_T}
over all ordered vectors induced by the same topological ordering of the CRP components 

\begin{lemma}\label{lem:Guido_extension}
Let $c_1,\dots, c_K$ be $K$ positive constants and $\Sigma_K$ be the set of all topological orderings of a DAG that satisfies Assumption~\ref{cond:root}, then
\begin{equation}
\sum\limits_{\sigma \in\Sigma_K}
\frac{1}{c_{\sigma(1)}(c_{\sigma(1)}+c_{\sigma(2)}) \dots (c_{\sigma(1)}+\dots+c_{\sigma(K)})}
=
\prod\limits_{k=1}^K \frac{1}{ \sum_{j\in V_k} c_j}.
\end{equation}
\end{lemma}
\begin{proof}
The result is shown by induction on~$K$, and is trivially true for $K=1$. The induction step follows after fixing the last value of the permutation, i.e., $i_K$, and applying the result for lower values than $K$. Note that only a root node of the DAG can be positioned at the end of the permutation. Let $\mathcal{R}_K$ denote the set of root nodes of the DAG and $\Sigma_K(k)$ denote the subset of $\Sigma_K$ where the last entry is given by $k$, with $k\in\mathcal{R}_K$.
Then,
\begin{equation}\label{eq:intermediate_induction}
\begin{array}{rl}
 &  \sum\limits_{\sigma \in\Sigma_K}
\frac{1}{c_{\sigma(1)}(c_{\sigma(1)}+c_{\sigma(2)}) \dots (c_{\sigma(1)}+\dots+c_{\sigma(K)})}\\

= & \sum\limits_{l\in\mathcal{R}_K} \sum\limits_{\sigma\in\Sigma_K(l)}
\frac{1}{c_{\sigma(1)}(c_{\sigma(1)}+c_{\sigma(2)}) \dots (c_{\sigma(1)}+\dots+c_{\sigma(K)})}\\

= & \sum\limits_{l\in\mathcal{R}_K}  \frac{1}{c_{\sigma(1)}+\dots+c_{\sigma(K)}}\sum\limits_{\sigma\in\Sigma_K(l)}
\frac{1}{c_{\sigma(1)}(c_{\sigma(1)}+c_{\sigma(2)}) \dots (c_{\sigma(1)}+\dots+c_{\sigma(K-1)})}.\\
\end{array}
\end{equation}
Let $V_k^{(l)}$ denote the subgraph rooted at node $k$ in the DAG with root node $l$ removed from the DAG, $k\neq l$. We can then apply the induction hypothesis to conclude that the final expression in~\eqref{eq:intermediate_induction} is equal to
\begin{equation}
\sum\limits_{l\in\mathcal{R}_K}  \frac{1}{c_{\sigma(1)}+\dots+c_{\sigma(K)}}\prod\limits_{k=1}^{K-1} \frac{1}{\sum_{j\in V_k^{(l)}} c_j}.
\end{equation}
Notice that the collection of subgraphs governed by the nodes in  $V_l, \{V_k^{(k)}\colon  k = 1,\dots, K, k\neq l \}$ is the same as the collection of subgraphs governed by $\{V_k \colon k=1,\dots, K \}$ as $l$ is a root node of the original DAG. So we obtain 
\begin{equation}
\begin{array}{rl}
& \sum\limits_{l\in\mathcal{R}_K}  \frac{1}{c_{\sigma(1)}+\dots+c_{\sigma(K)}}
\left(\prod\limits_{k=1}^{K} \frac{1}{\sum_{j\in V_k} c_j}\right) \left( \sum_{j\in V_l} c_j \right) 
= \prod\limits_{k=1}^{K} \frac{1}{\sum_{j\in V_k} c_j}
\frac{1}{\sum_{k=1}^{K} c_{\sigma(k)}}
\sum\limits_{l\in\mathcal{R}_K}  
 \sum\limits_{j\in V_l} c_j  \\
 =& \prod\limits_{k=1}^{K} \frac{1}{\sum_{j\in V_k} c_j}
\frac{1}{\sum_{k=1}^{K} c_{\sigma(k)}}
\sum\limits_{j=1}^K   c_j  
= \prod\limits_{k=1}^{K} \frac{1}{\sum_{j\in V_k} c_j},
\end{array}
\end{equation}
where we used that each node of the DAG is part of precisely one subgraph governed by the root nodes.
This concludes the proof.
\end{proof}

\begin{remark}\label{rem:Guido_1}
The identity in Lemma~\ref{lem:Guido_extension} is a generalization of the identity
\begin{equation}
\sum\limits_{(i_1,\dots,i_K) \in\mathrm{perm}(K)}
\frac{1}{c_{i_1}(c_{i_1}+c_{i_2}) \dots (c_{i_1}+\dots+c_{i_K})}
=
\prod\limits_{j=1}^K \frac{1}{ c_j},
\end{equation}
where $\mathrm{perm}(K)$ denotes the set of \emph{all} permutations of the integers $1,\dots,K$.
Alternatively, one can think of a DAG which solely consists of isolated vertices where $\Sigma_K$ includes all permutations. Such a DAG can occur when analyzing a set of $K$ isolated single-server queues.
\end{remark}

Using the above lemma, we can aggregate the mixture weights.

\begin{proposition}\label{prop:prob_t_sigma}
 Let $\sigma\in\Sigma_K$ with $\Sigma_K$ the set of all topological orderings of a DAG that satisfies Assumption~\ref{cond:root} and define
\begin{equation} \label{eq:hat_beta_sigma_K}
\hat{\beta}(\sigma) \coloneqq  \prod\limits_{k=1}^K \frac{1}{p\left(\mathcal{C}_{\sigma(1)},\dots,\mathcal{C}_{\sigma(k)}\right)}
\quad\text{and}\quad
\hat{\beta}(\Sigma_K) \coloneqq  \prod\limits_{k=1}^K \frac{1}{p(V_k)}.
\end{equation}
Then,
\begin{equation}
\mathbb{P}^*(\boldsymbol{T}^{\sigma}) \coloneqq \sum\limits_{\boldsymbol{T}\in\boldsymbol{T}^{\sigma}} \mathbb{P}^*(\boldsymbol{T}) = \frac{\hat{\beta}(\sigma)}{\hat{\beta}(\Sigma_K)}.
\end{equation}
\end{proposition}

\begin{example}
With the $\sigma$-ordered vectors as in Example~\ref{exam:sigma_ordered_vectors} we obtain the following weights using the above lemma. Let 
$\hat{\beta}((1,2,3)) = \frac{16}{3}$, $\hat{\beta}((2,1,3)) = \frac{8}{3}$ and $\hat{\beta}(\Sigma_K) = 8$ so that
$\mathbb{P}^*(\boldsymbol{T}^{(1,2,3)})= \frac{2}{3}$ and $\mathbb{P}^*(\boldsymbol{T}^{(2,1,3)})=\frac{1}{3}.$
Note that the same weights can be obtained after appropriately adding the weights of the $K$-critical vectors computed in Example~\ref{example:mixture_dist}. Due to Lemmas~\ref{lem:N_K_versus_Sigma} and~\ref{lem:mixture_components} we can simplify the mixture distribution in Example~\ref{example:mixture_dist} as follows:
\begin{equation}
\left( X_{\{1\}}, X_{\{1,2,3\}},X_{\{3\}},X_{\{3,4\}}\right) \overset{d}{=}
\left\{
\begin{array}{lcr}
\left( U_1 + \frac{1}{3}U_2+\frac{1}{4}U_3, \frac{1}{4}U_3, \frac{2}{9}U_2+\frac{1}{6}U_3, \frac{4}{9}U_2+ \frac{1}{3}U_3\right),  & & \text{w.p.~} \frac{2}{3},\\
\left( \frac{1}{3}U_2+\frac{1}{4}U_3, \frac{1}{4}U_3, \frac{1}{3}U_1 + \frac{2}{9}U_2+\frac{1}{6}U_3, \frac{2}{3}U_1 + \frac{4}{9}U_2+ \frac{1}{3}U_3\right),
 & & \text{w.p.~} \frac{1}{3}.\\
\end{array}
\right.
\end{equation}
\end{example}

\begin{proof}[Proof of Proposition~\ref{prop:prob_t_sigma}]
We first focus on a particular $\boldsymbol{T}\in\boldsymbol{T}^{\sigma}$ for some $\sigma\in\Sigma_K$, and emphasize how $\mathbb{P}^*(\boldsymbol{T})$ depends on $\mathrm{CR}(\boldsymbol{T})=\{i_1,\dots,i_K\}$ and $\sigma$ rather than the actual ordering of the job types within each critical component. Say that $\boldsymbol{T} = \left[ \boldsymbol{C}_{\sigma(1)},\dots,\boldsymbol{C}_{\sigma(K)}\right]$ with $\boldsymbol{C}_{\sigma(k)}$ some ordering of the set $\mathcal{C}_{\sigma(k)}$ for all $k$. So,
\begin{equation}
\mathbb{P}^*(\boldsymbol{T})   
=
\alpha 
\prod\limits_{k=1}^K \prod\limits_{i=1}^{|\mathcal{C}_{\sigma(k)}|}
 \frac{N\lambda^* p_{\sigma,k,i}}{\mu(\sigma,k,i)}
  \prod\limits _{i=1}^{|\mathcal{C}_{\sigma(k)}|-1}\left( 1-\frac{N\lambda^*p(\sigma,k,i)}{\mu(\sigma,k,i)}\right)^{-1},
\end{equation}
where $p_{\sigma,k,i}$ denotes the arrival fraction of the $i$th job type in $\boldsymbol{C}_{\sigma(k)}$,
\begin{equation}
\begin{array}{rcl}
p(\sigma,k,i) &=&  \sum\limits_{j=1}^{k-1} \sum\limits_{l=1}^{|\mathcal{C}_{\sigma(j)}|} p_{\sigma,j,l} + \sum\limits_{j=1}^{i} p_{\sigma,k,l} = p(\mathcal{C}_{\sigma(1)},\dots,\mathcal{C}_{\sigma(k-1)}) + \sum\limits_{j=1}^{i} p_{\sigma,k,l},\\
\mu(\sigma,k,i) &=& \sum\limits_{n=1}^N \mu_n \mathds{1}\left\{n \text{~is compatible with a job type in~} \cup_{j=1}^{k-1} \mathcal{C}_{\sigma(j)} \cup \{ C_{\sigma(k),1},\dots,C_{\sigma(k),i} \}  \right\},
\end{array}
\end{equation}
and $\alpha$ an appropriate normalization constant, i.e., $\alpha = \beta(\mathcal{N}_K)^{-1}$. Notice that for each $\boldsymbol{T}\in\mathcal{N}_K$ the same job types will occur and that all these job types contribute with factor $N\lambda^* p_S$ for those job types~$S$, irrespective of $\sigma$ such that $\boldsymbol{T}\in\boldsymbol{T}^{\sigma}$. Hence, we can update the normalization constant to $\alpha'$ and write
\begin{equation}
\begin{array}{rcl}
\mathbb{P}^*(\boldsymbol{T}) & =& 
\alpha' 
\prod\limits_{k=1}^K \prod\limits_{i=1}^{|\mathcal{C}_{\sigma(k)}|}
 \frac{1}{\mu(\sigma,k,i)}
  \prod\limits _{i=1}^{|\mathcal{C}_{\sigma(k)}|-1} \frac{\mu(\sigma,k,i)}{ \mu(\sigma,k,i) -N\lambda^*p(\sigma,k,i)}\\
& =& 
\alpha' 
\prod\limits_{k=1}^K 
 \frac{1}{\mu(\sigma,k,|\mathcal{C}_{\sigma(k)}|)}
  \prod\limits _{i=1}^{|\mathcal{C}_{\sigma(k)}|-1} \frac{1}{ \mu(\sigma,k,i) -N\lambda^*p(\sigma,k,i)}.
\end{array}
\end{equation}
Using the fact that $\{\mathcal{C}_{\sigma(1)},\dots, \mathcal{C}_{\sigma(k)}\}$ are critical subsets for all $k$ and hence that 
$$N\lambda^* p(\mathcal{C}_{\sigma(1)},\dots, \mathcal{C}_{\sigma(k)}) = N\lambda^* p(\sigma,k,|\mathcal{C}_{\sigma(k)}|) =  \mu(\sigma,k,|\mathcal{C}_{\sigma(k)}|)$$
yields
\begin{equation}
\mathbb{P}^*(\boldsymbol{T}) =
\alpha'' 
\prod\limits_{k=1}^K 
 \frac{1}{p(\mathcal{C}_{\sigma(1)},\dots, \mathcal{C}_{\sigma(k)})}
  \prod\limits _{i=1}^{|\mathcal{C}_{\sigma(k)}|-1} \frac{1}{ \mu(\sigma,k,i) -N\lambda^*p(\sigma,k,i)}\\
\end{equation}
for an updated normalization constant~$\alpha''$. 
Notice that the first part only depends on the chosen permutation~$\sigma$ of the sets of job types of the critical components, $\mathcal{C}_{\sigma(1)},\dots, \mathcal{C}_{\sigma(K)}$ and not on the order in which these types occur within~$\boldsymbol{T}$, i.e., $\boldsymbol{C}_{\sigma(1)},\dots, \boldsymbol{C}_{\sigma(K)}$. 
We now rewrite the denominator of the second part of the expression. The main idea to simplify this expression in the denominator is by neglecting those components that experience criticality. Before doing so we make the following observations:
\begin{enumerate}
\item  Each CRP component $\mathbb{C} = (\mathcal{C},\mathcal{Z})$ which corresponds to a leaf node in the DAG gives rise to a critical subset of job types. Indeed, this $\mathbb{C}$ can be positioned first in a topological ordering such that by Construction~\ref{construction} $\mathcal{C}$ is recovered as a critical set.
\item Let $\mathbb{C}_k = (\mathcal{C}_k,\mathcal{Z}_k)$ be a node in the DAG and $V_k$ be the subgraph of the DAG rooted at $\mathbb{C}_k$. The collection of job types in $V_k$ gives rise to a critical subset of job types, since we can generate a topological ordering with precisely those nodes associated with $V_k$ first. Moreover, $\mathbb{C}_k$ will be positioned last (with respect to the other CRP components induced by $V_k$) in this topological ordering since it is the root of the subgraph. Again, $V_k$ then induces a critical subset of job types by Construction~\ref{construction}.
\item Let $(\mathbb{C}_{\sigma(1)},\dots,\mathbb{C}_{\sigma(k)})$ be the first $k$ elements in a topological ordering. Then there exist an $L\in\{1,\dots,K\}$ and indices $\{j_1,\dots,j_L\}\subseteq\{1,\dots,K\}$ such that the nodes $\mathbb{C}_{\sigma(1)},\dots,\mathbb{C}_{\sigma(k)}$ can be partitioned into $L$ subgraphs rooted at nodes $\mathbb{C}_{j_1},\dots,\mathbb{C}_{j_L}$, i.e.,
$\bigcup_{l=1}^k \mathcal{C}_{\sigma(l)} = \bigcup_{l=1}^L V_{j_l}$. 
If this were not the case, the topological ordering would violate the DAG structure.
\end{enumerate}
Let $V_{\sigma(k)}$ denote the nodes of the subgraph rooted at $\mathbb{C}_{\sigma(k)}$. 
Given the above observations we can partition the job types as 
\begin{equation}
\begin{array}{rcl}
\{\mathcal{C}_{\sigma(1)},\dots,\mathcal{C}_{\sigma(k-1)}\}\cup\{C_{\sigma(k),1},\dots,C_{\sigma(k),i} \}
&=&
(\{\mathcal{C}_{\sigma(1)},\dots,\mathcal{C}_{\sigma(k-1)}\}\setminus V_{\sigma(k)})
\dot{\cup}
(V_{\sigma(k)} \setminus \mathcal{C}_{\sigma(k)}) \\
& & 
\dot{\cup}\{C_{\sigma(k),1},\dots,C_{\sigma(k),i} \}.
\end{array}
\end{equation}
It immediately follows that $$
N\lambda^* p(\sigma,k,i)  = N\lambda^* p(V_{\sigma(k)} \setminus \mathcal{C}_{\sigma(k)} ) + N\lambda^* p([C_{\sigma(k),1},\dots, C_{\sigma(k),i}]) + N\lambda^* p( \{\mathcal{C}_{\sigma(1)},\dots,\mathcal{C}_{\sigma(k)} \} \setminus V_{\sigma(k)}).$$
Moreover, due to the structure of the topological ordering (and the observations above), we have that $\mu(\sigma,k,i)  =\mu(V_{\sigma(k)} \setminus \mathcal{C}_{\sigma(k)} ) + \mu([C_{\sigma(k),1},\dots, C_{\sigma(k),i}]) + \mu( \{\mathcal{C}_{\sigma(1)},\dots,\mathcal{C}_{\sigma(k)} \} \setminus V_{\sigma(k)})$.
The third observation implies $N\lambda^* p( \{\mathcal{C}_{\sigma(1)},\dots,\mathcal{C}_{\sigma(k)} \} \setminus V_{\sigma(k)}) = \mu( \{\mathcal{C}_{\sigma(1)},\dots,\mathcal{C}_{\sigma(k)} \} \setminus V_{\sigma(k)})$,
such that 
\begin{equation}
\begin{array}{rcl}
\mu(\sigma,k,i) - N\lambda^* p(\sigma,k,i) &=&  \mu(V_{\sigma(k)} \setminus \mathcal{C}_{\sigma(k)} ) + \mu([C_{\sigma(k),1},\dots, C_{\sigma(k),i}])-N\lambda^* p(V_{\sigma(k)} \setminus \mathcal{C}_{\sigma(k)} )\\
& &+ N\lambda^* p([C_{\sigma(k),1},\dots, C_{\sigma(k),i}]), 
\end{array}
\end{equation}
which only depends on the job types in the rooted subgraph starting at node $\mathbb{C}_{\sigma(k)}$, i.e., $V_{\sigma(k)}$, and the ordering of the job types in just this node, i.e., $\boldsymbol{C}_{\sigma(k)}$. To slightly ease the notation, let us refer to the above notation as $\mu(V_{\sigma(k)},\boldsymbol{C}_{\sigma(k)},i) - N\lambda^* p(V_{\sigma(k)},\boldsymbol{C}_{\sigma(k)},i)$.

Combining the above yields
\begin{equation}\label{eq:last_step_P_star_T_sigma}
\begin{array}{rcl}
\mathbb{P}^*(\boldsymbol{T}^{\sigma}) & = & \sum\limits_{\boldsymbol{T}\in\boldsymbol{T}^{\sigma}} \mathbb{P}^*(\boldsymbol{T})\\

&= & \alpha'' \sum\limits_{\boldsymbol{T}\in\boldsymbol{T}^{\sigma}} \prod\limits_{k=1}^K \frac{1}{p(\mathcal{C}_{\sigma(1)},\dots, \mathcal{C}_{\sigma(k)})}
  \prod\limits _{i=1}^{|\mathcal{C}_{\sigma(k)}|-1} \frac{1}{ \mu(V_{\sigma(k)},\boldsymbol{C}_{\sigma(k)},i) - N\lambda^* p(V_{\sigma(k)},\boldsymbol{C}_{\sigma(k)},i)}\\
  
  &= & \alpha''  \prod\limits_{k=1}^K \frac{1}{p(\mathcal{C}_{\sigma(1)},\dots, \mathcal{C}_{\sigma(k)})}
   \sum\limits_{\boldsymbol{T}\in\boldsymbol{T}^{\sigma}} \prod\limits_{k=1}^K 
  \prod\limits _{i=1}^{|\mathcal{C}_{\sigma(k)}|-1} \frac{1}{ \mu(V_{\sigma(k)},\boldsymbol{C}_{\sigma(k)},i) - N\lambda^* p(V_{\sigma(k)},\boldsymbol{C}_{\sigma(k)},i)}\\
  
    &= & \alpha''  \prod\limits_{k=1}^K \frac{1}{p(\mathcal{C}_{\sigma(1)},\dots, \mathcal{C}_{\sigma(k)})} \cdot
    \prod\limits_{k=1}^K \sum\limits_{\boldsymbol{C}_{\sigma(k)}\in\mathrm{perm}(\mathcal{C}_{\sigma(k)})}
  \prod\limits _{i=1}^{|\mathcal{C}_{\sigma(k)}|-1} \frac{1}{ \mu(V_{\sigma(k)},\boldsymbol{C}_{\sigma(k)},i) - N\lambda^* p(V_{\sigma(k)},\boldsymbol{C}_{\sigma(k)},i)}\\
  
      &= & \alpha''  \prod\limits_{k=1}^K \frac{1}{p(\mathcal{C}_{\sigma(1)},\dots, \mathcal{C}_{\sigma(k)})} \cdot
    \prod\limits_{k=1}^K \sum\limits_{\boldsymbol{C}_{k}\in\mathrm{perm}(\mathcal{C}_{k})}
  \prod\limits _{i=1}^{|\mathcal{C}_{k}|-1} \frac{1}{ \mu(V_{k},\boldsymbol{C}_{k},i) - N\lambda^* p(V_{k},\boldsymbol{C}_{k},i)}\\
  
        &= & \alpha'''  \prod\limits_{k=1}^K \frac{1}{p(\mathcal{C}_{\sigma(1)},\dots, \mathcal{C}_{\sigma(k)})}.
\end{array}
\end{equation}
For the last step we used that the second part of the expression no longer depends on $\sigma$ and is in fact the same for all $\sigma\in\Sigma_K$. 
The normalization constant of the above expression for all $\sigma\in\Sigma_K$ can be rewritten into the desired format as depicted in~\eqref{eq:hat_beta_sigma_K} using the identity in 
Lemma~\ref{lem:Guido_extension}.
\end{proof}

We have now established all auxiliary results to prove Proposition~\ref{prop:equality_GF}.

\begin{proof}[Proof of Proposition~\ref{prop:equality_GF}]
Using the above results, we can show that the Laplace transforms of $(Y_S)_{S\in\mathcal{S}}$ and $(X_S)_{S\in\mathcal{S}}$ coincide.
We consider $(X_S)_{S\in\mathcal{S}}$ and condition on $\sigma\in\Sigma_K$ to aggregate the various mixture components (Lemmas~\ref{lem:N_K_versus_Sigma} and~\ref{lem:mixture_components}),
\begin{equation}
\mathbb{E}\Bigr[ \prod\limits_{S\in\mathcal{S}} \mathrm{e}^{-t_S X_S} \Bigr]
=
\sum\limits_{\sigma\in\Sigma_K} 
\mathbb{P}^*(\boldsymbol{T}^{\sigma})
\mathbb{E}\Bigr[ \prod\limits_{S\in\mathcal{S}} \mathrm{e}^{-t_S X_S} \mid \boldsymbol{T}^{\sigma} \Bigr],
\end{equation}
 for all $t_S\ge 0$. Next, we rely on~\eqref{eq:X_s} to obtain
 \begin{equation}
\sum\limits_{\sigma\in\Sigma_K} 
\mathbb{P}^*(\boldsymbol{T}^{\sigma})
\mathbb{E}\Bigr[ 
\prod\limits_{k=1}^K \exp \Bigl( 
-U_k \Bigl(
 \sum_{i=1}^k \sum_{S\in\mathcal{C}_{\sigma(i)}} t_S \frac{p_S}{p(\mathcal{C}_{\sigma(1)},\dots,\mathcal{C}_{\sigma(k)})} \Bigl) \Bigl) \Bigr].
 \end{equation}
 Since $U_1,\dots,U_K$ are independent and exponentially distributed with unit mean, we can write
 \begin{equation}
\sum\limits_{\sigma\in\Sigma_K} 
\mathbb{P}^*(\boldsymbol{T}^{\sigma})
\prod\limits_{k=1}^K
\Bigl( 1+ \sum_{i=1}^k \sum_{S\in\mathcal{C}_{\sigma(i)}} t_S \frac{p_S}{p(\mathcal{C}_{\sigma(1)},\dots,\mathcal{C}_{\sigma(k)})} \Bigl)^{-1}.
 \end{equation}
 Using the expression in Proposition~\ref{prop:prob_t_sigma} for $\mathbb{P}^*(\boldsymbol{T}^{\sigma})$, we obtain
 \begin{equation}
 \begin{array}{rl}
      &  \frac{1}{\hat{\beta}(\Sigma_K)}
\sum\limits_{\sigma\in\Sigma_K} 
\prod\limits_{k=1}^K
\Bigl( p(\mathcal{C}_{\sigma(1)},\dots,\mathcal{C}_{\sigma(k)}){+} \sum\limits_{i=1}^k \sum\limits_{S\in\mathcal{C}_{\sigma(i)}} t_S p_S \Bigl)^{-1}\\
     = & \frac{1}{\hat{\beta}(\Sigma_K)}
\sum\limits_{\sigma\in\Sigma_K} 
\prod\limits_{k=1}^K
\Bigl( \sum\limits_{i=1}^k \Bigl( p(\mathcal{C}_{\sigma(i)}){+}  \sum\limits_{S\in\mathcal{C}_{\sigma(i)}} t_S p_S \Bigl) \Bigl)^{-1}
.
 \end{array}
 \end{equation}
Using the identity in Lemma~\ref{lem:Guido_extension} with $c_k = p(\mathcal{C}_k) + \sum_{S\in\mathcal{C}_k} t_S p_S$ for all $k$ and substituting the expression for $\hat{\beta}(\Sigma_K)$ in~\eqref{eq:hat_beta_sigma_K}, results in
 \begin{equation}
 \frac{1}{\hat{\beta}(\Sigma_K)}
\prod\limits_{k=1}^K
\Bigl(  p(V_{k})+  \sum\limits_{S\in\mathcal{C}_{k}} t_S p_S  \Bigl)^{-1}
 =
\prod\limits_{k=1}^K
\Bigl( 1+  \sum\limits_{S\in V_{k}} t_S \frac{p_S}{ p(V_{k})}  \Bigl)^{-1}. 
 \end{equation}
This concludes the proof.
\end{proof}

\section{Outlook}\label{sec:outlook}
Our methods and results suggest two natural topics for further research.

First of all, it would be interesting to apply the above framework to slightly different models that also have product-form stationary distributions. 
One can think of, for instance, order-independent queues~\cite{Krzesinski2011} or redundancy policies operating in overloaded systems with abandonments~\cite{Gardner2020}. 
Although the product-form expressions seem fairly similar at first glance, they have a different way to describe the total rate at which jobs leave the system. This complicates the derivation of the PGF as 
aggregation of the states according to the $k$-critical vectors (Definition~\ref{def:crit_vectors}) will no longer result in a closed-form expression.
A similar analysis has also proven to be successful for related models without the queueing feature, e.g., matching models, which still maintain a product-form stationary distribution under some form of a CRP condition~\cite{Comte2022}. It would be interesting to explore whether these conditions could be relaxed.

From a broader perspective, it is worth observing that the stability conditions in Condition~\ref{cond:stab} are not only necessary but also sufficient for a far wider range of routing and scheduling policies~\cite{Foss1998}. Consequently, all these `maximally stable' or `throughput-optimal' policies have the same (number of) critically loaded subsystems as redundancy policies, and might therefore potentially also exhibit quite similar heavy-traffic behavior. This seems especially plausible for the celebrated Join-the-Shortest-Queue (JSQ) policy since it is similar in spirit to the Join-the-Smallest-Workload (JSW) policy, which in turn is equivalent to the redundancy c.o.s. policy as noted earlier. Indeed, similarities between JSQ and JSW in terms of process-level limits and (full) state space collapse results have been observed in CRP scenarios~\cite{Atar2019,Atar2019replicate}.  It would be interesting to explore whether this extends to non-CRP scenarios, and whether some asymptotic equivalence property for a wider class of policies might hold. In this regard it is worth recalling the close resemblance with the limiting queue length distribution for the MaxWeight scheduling algorithm in the special case of the N-model with equal service rates. While it is difficult to extrapolate from such a highly special case, the striking commonality suggests that the independent exponentially distributed random variables associated with the critically loaded subsystems may arise more universally. We conjecture however that the specific form of the cone and the relative proportions of the various job types are in general policy-dependent.

\section*{Acknowledgments}
The authors thank Guido Janssen for showing an inductive proof of the identity in Remark~\ref{rem:Guido_1}, after which the proof of Lemma~\ref{lem:Guido_extension} has been modeled, and for providing the argument of the proof of Lemma~\ref{lem:moments_Guido}.

\bibliographystyle{abbrv}
\bibliography{references}

\appendix

\section{Preliminaries - auxiliary results and proofs}

\subsection{Redundancy c.o.s.}\label{app:prelim_cos}
\begin{proposition}\label{prop:pgfcos}
The joint PGF of the numbers of waiting jobs of the various types for the redundancy c.o.s.\ policy is given by 
\begin{equation}\label{eq:pgfcos}
\mathbb{E}\left[\prod\limits_{S \in {\mathcal S}} z_S^{\tilde{Q}_S}\right] = \frac{g(\boldsymbol{z})}{g(\boldsymbol{1})},
\end{equation}
where $\boldsymbol{z}$ and $\boldsymbol{1}$  are $|\mathcal{S}|$-dimensional vectors with entries $|z_S| \le 1$ and
\begin{equation}\label{eq:g}
g(\boldsymbol{z})=\sum\limits_{L=0}^{N}\sum\limits_{\boldsymbol{u}\in \mathcal{E}_L}\prod\limits_{l=1}^L \frac{\mu_{u_l}}{\lambda_{\mathcal{C}(u_1,\dots,u_l)}} 
 \sum\limits_{m=0}^{|\mathcal{S}|}\sum\limits_{\boldsymbol{S}\in \mathcal{S}_m^{\boldsymbol{u}}} \prod\limits_{j=1}^m \frac{N\lambda p_{S_j}z_{S_j}}{\mu(S_1,\dots,S_j)}\left(1{-}\frac{N\lambda}{\mu(S_1,\dots,S_j)}\sum\limits_{i=1}^j p_{S_i}z_{S_i}\right)^{-1}.
\end{equation}

The set consisting of all ordered vectors of $L$ idle servers is denoted by $\mathcal{E}_L$.
The $m$-dimensional vector $\boldsymbol{S}$ consists of $m$ different waiting job types that are not compatible with the idle servers $\boldsymbol{u}$, and the set consisting of all these vectors is denoted by $\mathcal{S}_m^{\boldsymbol{u}}$.
\end{proposition}


\subsection{Proof of Theorem~\ref{th:construction}}\label{app:proof_construction}

The proof of Theorem~\ref{th:construction} follows immediately from combining the following two lemmas. 

\begin{lemma}\label{lem:construction_1}
For any fixed $\sigma\in\Sigma_K$ and $k\in\{1,\dots,K\}$, $T_{\sigma,k}$ obtained via Construction~\ref{construction} is a critical subset of job types as in Definition~\ref{def:critical_subset}.
\end{lemma}

\begin{lemma}\label{lem:construction_2}
If $\mathcal{T}$ is a critical subset of job types as in Definition~\ref{def:critical_subset}, it can be recovered using Construction~\ref{construction}.
\end{lemma}


\begin{proof}[Proof of Lemma~\ref{lem:construction_1}]
Fix $\sigma\in\Sigma_K$ and $k\in\{1,\dots,K\}$. We will show that $T_{\sigma,k}$ is a critical subset of job types. Hence by Definition~\ref{def:critical_subset} it must hold that 
\begin{equation}\label{eq:construction_goal}
p(T_{\sigma,k}) = \frac{\mu(T_{\sigma,k})}{N\mu}.
\end{equation}
Note that $\mu(T_{\sigma,k})$ refers to the aggregate service rate of \emph{all} servers compatible with job types $T_{\sigma,k}$.

By definition of the topological ordering, none of the compatible servers of the job types $T_{\sigma,k}$ can be outside of $Z_{\sigma,k}$. Otherwise there exists an arc from one of the components in $\{\mathbb{C}_{\sigma(1)},\dots,\mathbb{C}_{\sigma(k)}\}$ to one of the components in $\{\mathbb{C}_{\sigma(k+1)},\dots,\mathbb{C}_{\sigma(K)}\}$, which implies that the latter component must be positioned earlier in the topological ordering. This yields a contradiction. Hence the compatible servers of $T_{\sigma,k}$ form a subset of $Z_{\sigma,k}$, so $\mu(T_{\sigma,k}) \le \mu(Z_{\sigma,k})$.

Since the CRP components form a partition of the job types and servers, we can write
\begin{equation}
p(T_{\sigma,k}) = p\left(\bigcup\limits_{i=1}^k \mathcal{C}_{\sigma(i)}\right) = \sum\limits_{i=1}^k p\left(\mathcal{C}_{\sigma(i)}\right).
\end{equation}
From \citep[Lemma~3(i)]{Afeche2021} we know that $p\left(\mathcal{C}_{\sigma(k)}\right) = \mu\left(\mathcal{Z}_{\sigma(k)}\right) / (N\mu)$. Hence,
\begin{equation}
p(T_{\sigma,k}) = \frac{1}{N\mu}\sum\limits_{i=1}^k \mu\left(\mathcal{C}_{\sigma(i)}\right)
=\frac{1}{N\mu} \mu\left(\bigcup\limits_{i=1}^k \mathcal{Z}_{\sigma(i)}\right) = \frac{\mu(Z_{\sigma,k})}{N\mu}.
\end{equation}
Combining the above observations results in
\begin{equation}
p(T_{\sigma,k}) \le \frac{\mu(T_{\sigma,k})}{N\mu} \le \frac{\mu(Z_{\sigma,k})}{N\mu} = p(T_{\sigma,k}),
\end{equation}
where the first inequality is due to the stability conditions in Condition~\ref{cond:stab}.
This yields~\eqref{eq:construction_goal}, and we can conclude that $T_{\sigma,k}$ is indeed a critical subset of job types.
\end{proof}

\begin{proof}[Proof of Lemma~\ref{lem:construction_2}]
Assume that $\mathcal{T}$ is a critical subset of job types according to Definition~\ref{def:critical_subset}. We will argue that $\mathcal{T}$ can always be recovered using Construction~\ref{construction}. 

We first introduce some notation. Let 
\begin{equation}
X\coloneqq \bigcup\limits_{S\in\mathcal{T}} S
\end{equation}
denote the set of all compatible servers of job types in $\mathcal{T}$.
Note that earlier we used $\mathcal{T}$ to denote both the set of job types and the compatible servers (as each job type is defined by its compatible servers), and we like to emphasize here the difference between the two.

Recall that the compatibility relations between the job types and servers can be represented by a matching~$M$ with $m_{S,n} = 1$ whenever job type $S$ is compatible with server~$n$ (or $n\in S$), and 0 otherwise.
For any subset of servers $Y\subseteq\{1,\dots,N\}$, the job types that can uniquely be served by them are denoted by $\mathcal{U}(Y)$.  Formally, 
\begin{equation}
\mathcal{U}(Y) = \left\{S \in \mathcal{S}\colon \sum_{n\notin Y} m_{S,n} = 0\right\}.
\end{equation}

In the remainder of the proof we will use the following two observations:\\
\textbf{Observation~1:} $\mathcal{T}\subseteq \mathcal{U}(X)$. Let $S\in \mathcal{T}$ and assume by contradiction that $S\notin \mathcal{U}(X)$. So $\sum_{n\notin X} m_{S,n} >0$, and hence there is a server $\tilde{n}\in\{1,\dots,N\}\setminus X$ such that $m_{S,\tilde{n}} = 1$. However, as $m_{S,\tilde{n}} =1 $, it follows by construction of~$X$ that $\tilde{n}\in~X$. This yields a contradiction.\\
\textbf{Observation~2:} $\mathcal{U}_{M}(Y) \subseteq \mathcal{U}_{\tilde{M}}(Y)$ for any $Y\subseteq\{1,\dots,N\}$ where $M$ and $\tilde{M}$ represent the original and residual matching, respectively. This follows directly from the fact that $m_{S,n} \ge \tilde{m}_{S,n}$ for all job types $S$ and servers $n$. \\

We focus on the set of servers~$X$, and distinguish three cases that each will be investigated separately:
\begin{enumerate}
\item For some $\sigma\in\Sigma_K$ and $k\in\{1,\dots,K\}$, it holds $X = Z_{\sigma,k}$;
\item The set~$X$ can be written as a union of (a subset of) the server sets of the CRP components, i.e., $\{\mathcal{Z}_k \colon k=1,\dots,K\}$. However there exists no permutation~$\sigma\in\Sigma_K$ such that $X = Z_{\sigma,k}$ for some~$k$;
\item The set~$X$ can not be written as a union of server sets of CRP components.
\end{enumerate}
We now show that the first case indeed results in the critical set of job types equal to~$\mathcal{T}$ and that the remaining two cases cannot occur by showing a contradiction.

\textbf{Case 1:} We aim to show that $\mathcal{T}=T_{\sigma,k}$. From Lemma~\ref{lem:construction_1} we know that $T_{\sigma,k}$ is a critical subset of job types and that $p(T_{\sigma,k}) = \mu(Z_{\sigma,k})/(N\mu) = \mu(X)/(N\mu)$. By assumption, $\mathcal{T}$ is a critical subset of job types and hence $p(T) = \mu(X)/(N\mu)$. It then follows that $p(T) = p(T_{\sigma,k})$.

Moreover, from~\cite[Lemma~3(iii)]{Afeche2021} we know that $T_{\sigma,k} = \mathcal{U}_M(X)$. Observation~1 shows that $\mathcal{T}\subseteq \mathcal{U}_M(X)$. So, $\mathcal{T}\subseteq T_{\sigma,k}$.

After combining the above two results, we can indeed conclude that $\mathcal{T} = T_{\sigma,k}$. 

\textbf{Case 2:} Assume, after possibly relabeling the CRP components, that $X = \cup_{i=1}^k \mathcal{Z}_i$ and that there exists no permutation $\sigma$ or topological ordering such that $X = Z_{\sigma,k}$. This implies the existence of a CRP component $\tilde{\mathbb{C}} = (\tilde{\mathcal{C}},\tilde{\mathcal{Z}})$ such that there is an arc $(\mathbb{C}_i,\tilde{\mathbb{C}})$ in the DAG for some $i\in\{1,\dots,k\}$, and hence $\tilde{\mathbb{C}}$ would have to be positioned \emph{before} $\mathbb{C}_i$ in any valid topological order. Hence there must exist a job type $S\in\mathcal{C}_i$ and a server $n\in\tilde{\mathcal{Z}}$ such that $m_{S,n} > 0 $. So, $S \notin \mathcal{U}_M(X)$, then 
\begin{equation}
N\lambda p (\mathcal{U}_M(X) ) \le \sum_{i=1}^k N\lambda p(\mathcal{C}_i ) - N\lambda p_S.
\end{equation}  
From Observation~1, we have that $\mathcal{T}\subseteq \mathcal{U}_M(X)$ and so
\begin{equation}
\begin{array}{rcl}
\mu(X) - N\lambda p(\mathcal{T}) & \ge & \mu(X) -N\lambda p(\mathcal{U}_M(X)) \\\
& = & \sum\limits_{i=1}^k \mu(\mathcal{Z}_i) - N\lambda p(\mathcal{U}_M(X)) \\
& \ge & \sum\limits_{i=1}^k  \left( \mu(\mathcal{Z}_i) - N\lambda p(\mathcal{C}_i) \right)  + N\lambda p_S.
\end{array}
\end{equation}
As $\mathcal{T}$ is a critical subset of job types, $\mu(X) - N\lambda p(\mathcal{T})  \rightarrow 0$ as $\lambda\uparrow\mu$. From \cite[Lemma~3(i)]{Afeche2021} we know that also $\mu(\mathcal{Z}_i) - N\lambda p(\mathcal{C}_i)  \rightarrow 0$ as $\lambda\uparrow\mu$. Hence it follows that $p_S =0$, which yields a contradiction. We conclude that Case~2 can not occur.

\textbf{Case 3:}
Assume that $X$ can not be written as a union of server sets of CRP components. From Observation~2 we deduce that 
\begin{equation}
\mu(X) - N\lambda p(\mathcal{U}_M(X)) \ge \mu(X) - N\lambda p(\mathcal{U}_{\tilde{M}}(X)).
\end{equation}
Now define $X_k \coloneqq X \cap \mathcal{Z}_k$ for $k=1,\dots,K$. Then, 
\begin{equation}
\mu(X) = \sum\limits_{k=1}^K \mu(X_k). 
\end{equation}
Since the CRP components are disconnected under $\tilde{M}$, we also have that
\begin{equation}
N\lambda p(\mathcal{U}_{\tilde{M}}(X)) = \sum\limits_{k=1}^K N\lambda p(\mathcal{U}_{\tilde{M}}(X_k)). 
\end{equation}
Since $X$ is not equal to the union of server sets of CRP components, there is a $\tilde{k}$ such that $\emptyset \neq X_{\tilde{k}} \subsetneq \mathcal{Z}_{\tilde{k}}$. From \cite[Lemma~3(ii)]{Afeche2021} it is known that $\mu(X_k) - N\lambda p(\mathcal{U}_{\tilde{M}}(X_k)) >0$.

Combining the above results with Observations~1 and 2 yields
\begin{equation}
\begin{array}{rcl}
\mu(X) - N\mu  p(\mathcal{T}) & \ge & \mu(X) - N\mu p(\mathcal{U}_M(X)) \\
& \ge & \mu(X) - N\mu p(\mathcal{U}_{\tilde{M}}(X)) \\
& = & \sum\limits_{k=1}^K \mu(X_k) - N\mu p(\mathcal{U}_{\tilde{M}}(X_k)) \\
& \ge & \mu(X_{\tilde{k}}) - N\mu p(\mathcal{U}_{\tilde{M}}(X_{\tilde{k}})) > 0.
\end{array}
\end{equation}
This contradicts the fact that $\mathcal{T}$ is a critical subset of job types, and we conclude that Case~3 cannot occur.

This concludes the proof of Lemma~\ref{lem:construction_2}.

\end{proof}

\section{Main results - additional comments}




\subsection{Comment on the used notation in Subsection~\ref{subsec:proof_part2}} \label{app:assumption_lam_star_mu}

In the notation in the proofs of Subsection~\ref{subsec:proof_part2} it is assumed that all job types belong to (at least) one of the critical subsets and hence the whole system experiences criticality, so that $\lambda^* = \mu$. 
However, if there had been job types that do not experience criticality, these job types could occur in the $K$-critical vectors. In particular, these vectors would be of the form $\boldsymbol{T} = \left[
\boldsymbol{C}_{\sigma(1)},\dots,\boldsymbol{C}_{\sigma(K)}, \boldsymbol{NC}
\right]$,
for some $\sigma\in\Sigma_K$ where $\boldsymbol{C}_{\sigma(k)}$ denotes a permutation of the job types in the CRP component $\mathbb{C}_{\sigma(k)}$ for all $k$, and $\boldsymbol{NC}$ a permutation of (a subset of) the job types that do not experience criticality. Notice that these non-critical job types must always occur at the end of the critical vector, otherwise the $K$-criticality of the vector would be violated. Moreover, for any topological ordering $\sigma\in\Sigma_K$, the same collection of permutations of non-critical subsets would occur. This would yield an adaptation of the definition of the $\sigma$-critical vectors (Definition~\ref{def:sigma_critcal_vectors}), and a few additional steps in the proof of Proposition~\ref{prop:prob_t_sigma}. Especially in the last step (see~\eqref{eq:last_step_P_star_T_sigma}) one needs to use the above observation that the non-critical job types will contribute to each $\mathbb{P}^*(\boldsymbol{T}^{\sigma})$ equally much, such that the normalization constant $\alpha'''$ can be adapted and the same expression as in Proposition~\ref{prop:prob_t_sigma} would be obtained. This small technicality would have made the notation in the proofs of the above intermediate results even more cumbersome.

Moreover, we would like to emphasize that the main results still hold if not all job types experience criticality, even using the notation as in Subsection~\ref{subsec:proof_part2}. The job types that will not experience criticality will not occur in the DAG either, so that eventually their (scaled) number of jobs will vanish in a heavy-traffic regime.

\subsection{Comments on Assumption~\ref{cond:root}} \label{app:root_partition}
In this subsection we provide some examples of systems that satisfy the root partitioning assumption as formulated in Assumption~\ref{cond:root}. We emphasize that the list of examples is not exhaustive, and only serves to illustrate the wide range of models for which Assumption~\ref{cond:root} holds.\\

All systems for which the model parameters satisfy the \emph{CRP condition} yield a DAG that consists of a single node (Condition~\ref{cond:crp}) for which Assumption~\ref{cond:root} trivially holds.

 Assumption~\ref{cond:root} also holds for systems with a \emph{unique topological ordering}. By contradiction, it can be argued that if a node belongs to the subgraphs induced by two root nodes $r_1$ and $r_2$, then there exists a topological ordering where $r_1$ is positioned before $r_2$ and vice versa.
Alternatively, we can observe that all critical vectors yield the same expression for the corresponding mixture components when there is a single topological ordering. Hence, $(X_S)_{S\in\mathcal{S}}$ and $(Y_S)_{S\in\mathcal{S}}$ have the same distribution, and the statement of Theorem~\ref{th:general_results} follows directly from Theorem~\ref{cor:mixture_dist}.
(A more formal argument can be found in the proof of Lemma~\ref{lem:mixture_components}.)

\emph{Nested systems} are systems where the compatibility graph satisfies the following property: let $S_1$ and $S_2$ be two distinct job types, if they have a non-empty intersection of compatible servers, i.e., $S_1\cap S_2 \neq \emptyset$, then one of them must be contained in the other, i.e., $S_1\subsetneq S_2$ or $S_2\subsetneq S_1$ \cite{Gardner2017scheduling,Gardner2020}. By contradiction, it can again be argued that if a node belongs to the subgraphs induced by two root nodes $r_1$ and $r_2$, then there exists a node $v$ where the subgraphs meet for the first time and nodes $v_1$ and $v_2$ on the path from $r_1$ and $r_2$ respectively to node $v$ such that there is one job type in node $v_1$ and one job type in $v_2$ for which the nested property is violated. 

As the system in Example~\ref{ex:example_main_result} (Figure~\ref{fig:general_result}) illustrates, Assumption~\ref{cond:root} can be satisfied for systems that are neither nested, nor have a unique topological ordering. 
In particular, for arbitrarily structured compatibility graphs, whether or not Assumption~\ref{cond:root} is satisfied depends on the delicate interplay between the compatibility constraints and the model parameters. To illustrate this, consider a system with three servers, each operating at speed $\mu$ and define the job types as $\mathcal{S} = \{\{1,2\},\{2,3\},\{2\}\}$. 
\begin{itemize}
\item[-] Let $\lambda_S \equiv \lambda$ for all $S\in\mathcal{S}$ and some $\lambda<\mu$. Then the critical components are given by $\mathbb{C}_1 = (\{1,2\},\{1\})$, $\mathbb{C}_2 = (\{2,3\},\{3\})$ and $\mathbb{C}_3 = (\{2\},\{2\})$ as $\lambda\uparrow\mu$, and the DAG consists of two edges, namely $(\mathbb{C}_1,\mathbb{C}_3)$ and $(\mathbb{C}_2,\mathbb{C}_3)$. This DAG clearly cannot be partitioned according to its root nodes, so Assumption~\ref{cond:root} is not satisfied.
\item[-] Let $(\lambda_{\{1,2\}},\lambda_{\{2,3\}},\lambda_{\{2\}}) = \lambda(\frac{4}{3},1,\frac{2}{3})$ for some $\lambda<\mu$. Then the critical components are given by $\mathbb{C}_1 = (\{\{1,2\},\{2\}\},\{1,2\})$ and $\mathbb{C}_2 = (\{2,3\},\{3\})$ as $\lambda\uparrow\mu$, and the DAG consists of a single edge, namely $(\mathbb{C}_2,\mathbb{C}_1)$. This DAG can clearly be partitioned according to its root node, so Assumption~\ref{cond:root} is satisfied.
\end{itemize}
Hence, even if Assumption~\ref{cond:root} is not satisfied for a particular set of model parameters for a given system, it might be for other values of model parameters.

\section{Main results - generalization}\label{subsec:extension}

To derive the main result in Theorem~\ref{th:general_results}, we assumed that $(\lambda_S)_{S\in\mathcal{S}} = N\lambda (p_S)_{S\in\mathcal{S}}$, and in particular that the fractions of jobs of each type remain fixed as the system approaches its heavy-traffic limit, i.e., when $\lambda\uparrow\lambda^*$. 
Similar to the heavy-traffic analysis conducted by Hillas \emph{et al.}~\cite{Hillas2023}, we can relax his assumption by selecting any positive constants $(\gamma_S)_{S\in\mathcal{S}}$ and set
\begin{equation}\label{eq:HT_trajectory}
\lambda_S^{\epsilon} = N\lambda^*p_S - \epsilon\gamma_S,
\end{equation}
for all $S\in\mathcal{S}$ and $0\le \epsilon \le \epsilon^{+}$ for some $\epsilon^{+}$ small enough such that all the arrival rates are positive. The heavy-traffic behavior of the system can then be studied by taking $\epsilon\downarrow 0$. 
To ease the notation, we will omit the $\epsilon$ dependence of $\lambda_S$. 

Notice that this definition coincides with the setting in the previous subsections when we set $\gamma_S = N\lambda^*p_S$, where $\epsilon = 1-\frac{\lambda}{\lambda^*}$. In terms of the stability region, this allows us to approach any point on the boundary from any interior point, while the previous set-up only allowed the boundary to be approached in a direction starting from the origin. 

\begin{example} \label{example:N_model_gen_conv}
Consider the N-model as depicted in Figure~\ref{fig:Nmodel}. Assuming that $(p_{\{1,2\}},p_{\{2\}}) = (\mu_1,\mu_2)/(\mu_1+\mu_2)$, the critical arrival rate vector is then given by $(N\lambda^*p_{\{1,2\}},N\lambda^*p_{\{2\}}) = (\mu_1,\mu_2)$. Let $ (\gamma_{\{1,2\}},\gamma_{\{2\}})  = (\mu_1 +\delta,\mu_2 - \delta)$ for some $\delta \ge 0$. Hence, $(\lambda_{\{1,2\}}^{\epsilon},\lambda_{\{2\}}^{\epsilon}) =  \left((1-\epsilon)\mu_1 - \epsilon\delta,(1-\epsilon)\mu_2 + \epsilon\delta\right)$ for $\epsilon\downarrow 0$. Notice that the setting with $\delta = 0$, as depicted by a dotted line in Figure~\ref{fig:stability_Nmodel_generalization}, fits the framework of the previous subsections. The result in this subsection also subsumes the settings for any $\delta> 0$, depicted by a dashed line in Figure~\ref{fig:stability_Nmodel_generalization}. Observe that the heavy-traffic trajectories of the latter cases are closer to the boundary of the stability region induced by the constraint $\lambda_{\{2\}} < \mu_2$ compared to the trajectory with $\delta = 0$.
\end{example}

\begin{figure}[h]
\centering
\includegraphics[scale=1]{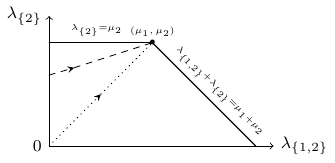}
\caption{Visual representation of the stability region of the N-model (Figure~\ref{fig:Nmodel}) with different directions of convergence of the arrival rate vectors $(\lambda_{\{1,2\}}^{\epsilon},\lambda_{\{2\}}^{\epsilon})$ to the critical arrival rate vector $(N\lambda^*p_{\{1,2\}},N\lambda^*p_{\{2\}}) = (\mu_1,\mu_2)$. The heavy-traffic trajectories are induced by $(\gamma_{\{1,2\}},\gamma_{\{2\}}) = (\mu_1 + \delta,\mu_2 - \delta)$, and for some $\delta > 0$  depicted by a dashed line while for $\delta = 0$ depicted by a dotted line. }
\label{fig:stability_Nmodel_generalization}
\end{figure}

Using this general direction of convergence of the arrival rate vector, we can prove a generalization of Theorem~\ref{th:general_results}. Similarly to~\eqref{eq:abbreviated}, we first define
 \begin{equation}
 \gamma(\boldsymbol{T},j) \coloneqq \sum\limits_{i=1}^j \gamma_{T_i}
 \end{equation}
 for any ordered vector of job types~$\boldsymbol{T}$ and $j=1,\dots, |\boldsymbol{T}|$.
 Moreover, let $\gamma(V_k)$ represent the sum of those elements of $(\gamma_S)_{S\in\mathcal{S}}$ such that the corresponding job types occur in the nodes of the subgraph rooted at critical component $\mathbb{C}_k$, denoted by $V_k$.
Then, we define a slightly altered version of the (deterministic) coefficients in~\eqref{eq:prob_type_rooted_tree}, i.e.,
\begin{equation}
\label{eq:prob_type_rooted_tree_generalized}
\hat{\gamma}_{k,S} \coloneqq \frac{N\lambda^*p_S}{\gamma(V_k)} \mathds{1}\{S\in V_k\}
\end{equation}
for all $k=1,\dots,K$ and $S\in\mathcal{S}$.

\begin{theorem}[Main result - generalization]
\label{th:general_results_extension}
If Assumption~\ref{cond:root} holds, then the vector of \textup{(}scaled\textup{)} queue lengths, $\epsilon(Q_S)_{S\in\mathcal{S}}$, converges in distribution to 
$(Y_S)_{S\in\mathcal{S}}$, i.e.,
\begin{equation}
\epsilon \left( Q_S\right)_{S\in\mathcal{S}}
\overset{d}{\rightarrow}
\left( Y_S\right)_{S\in\mathcal{S}},
\end{equation}
 when $\epsilon\downarrow 0$, with
\begin{equation}\label{eq:general_result_RV_generalization}
\left( Y_S\right)_{S\in\mathcal{S}}
\overset{d}{=}
\left(\sum\limits_{k=1}^K \hat{\gamma}_{k,S}   U_k \right)_{S\in\mathcal{S}},
\end{equation}
$\hat{\gamma}_{k,S}$ as defined in~\eqref{eq:prob_type_rooted_tree_generalized} and
$U_1,\dots,U_K$ independent and exponentially distributed random variables with unit mean. The result holds for both the c.o.c.\ and c.o.s.\ mechanism.
\end{theorem}

If the CRP condition (Condition~\ref{cond:crp}) is satisfied, then the associated DAG consists of a single node with all critical job types. Denote this subset of critical job types by $\mathcal{T}^*$, then the scaled numbers of jobs of the various types converge to
\begin{equation}
\left( U \frac{N\lambda^* }{\gamma(\mathcal{T}^*)}\left(p_S \right)_{S\in\mathcal{T}^*}, (0)_{S\notin\mathcal{T}^*}\right)
\end{equation}
as $\epsilon\downarrow 0$ with $U$ a unit mean exponential random variable. From this we can deduce that, as long as the total speed $\gamma(\mathcal{T}^*)$ at which a point of the boundary of the stability region is approached is unaltered, the direction from which this point is approached plays no role when the CRP condition is satisfied.

This observation is no longer true if the CRP condition is not satisfied, as illustrated in the following example.

\begin{example}
Consider the N-model with the heavy-traffic trajectories as in Example~\ref{example:N_model_gen_conv}. Applying the above theorem yields
\begin{equation}
\begin{array}{rcl}
\left( Y_{\{1,2\}},Y_{\{2\}}\right) & \overset{d}{=} &
\left( \frac{N\lambda^*p_{\{1,2\}}}{\gamma(\mathcal{S})}U_1, 
\frac{N\lambda^*p_{\{2\}}}{\gamma(\mathcal{S})}U_1
+ \frac{N\lambda^*p_{\{2\}}}{\gamma(\{2\})}U_2 \right) \\
& = & \left( \frac{\mu_1}{\mu_1+\mu_2}U_1, \frac{\mu_2}{\mu_1+\mu_2}U_1 + \frac{\mu_2}{\mu_2-\delta}U_2\right). 
\end{array}
\end{equation}
We observe that the number of type-$\{2\}$ jobs is larger when $\delta >0$ while the number of type-$\{1,2\}$ jobs remains the same compared to a setting with $\delta = 0$. 
\end{example}

We now consider a slightly larger example to illustrate the impact of $(\gamma_S)_{S\in\mathcal{S}}$.

\begin{example}\label{exam:extN}
Consider an extended version of the N-model as depicted in Figure~\ref{fig:vb_extN} with $N=4$ servers and assume that 
$
(p_{\{1,2\}},p_{\{2,3\}},p_{\{3,4\}},p_{\{4\}}) = (\mu_1, \mu_2, \mu_3, \mu_4 ) / (N\mu).
$
The critical arrival rate vector is then given by 
$$
N\lambda^*(p_{\{1,2\}},p_{\{2,3\}},p_{\{3,4\}},p_{\{4\}}) = (\mu_1, \mu_2, \mu_3, \mu_4 ),
$$
with 
$
\mathcal{CR}(\mathcal{S}) = \left\{
\left\{4\right\}
\left\{\{3,4\},\{4\}\right\}
\left\{\{2,3\},\{3,4\},\{4\}\right\}
\left\{\mathcal{S}\right\}
\right\}
$
being the collection of all critical subsets of job types. The associated DAG is a line graph with four CRP components, $\mathbb{C}_1 = (\{1,2\},\{1\})$, $\mathbb{C}_2 = (\{2,3\},\{2\})$, $\mathbb{C}_3 = (\{3,4\},\{3\})$ and $\mathbb{C}_4 = (\{4\},\{4\})$. Under the conditions of Section~\ref{sec:main_results} \textup{(}Theorem~\ref{th:general_results}\textup{)}, we observe the following convergence result
\begin{equation}
\left(1-\frac{\lambda}{\lambda^*}\right)\left[
\begin{array}{c}
Q_{\{1,2\}}\\
Q_{\{2,3\}}\\
Q_{\{3,4\}}\\
Q_{\{4\}}
\end{array}
\right]
\overset{d}{\rightarrow}
\left[
\begin{array}{l}
p_{\{1,2\}}U_1\\
p_{\{2,3\}}U_1 + \frac{p_{\{2,3\}}}{p_{234}}U_2\\
p_{\{3,4\}}U_1 + \frac{p_{\{3,4\}}}{p_{234}}U_2 + \frac{p_{\{3,4\}}}{p_{34}}U_3\\
p_{\{4\}}U_1 + \frac{p_{\{4\}}}{p_{234}}U_2  + \frac{p_{\{4\}}}{p_{34}}U_3 + U_4
\end{array}
\right]
\end{equation}
as $\lambda\uparrow\mu$ where we used the abbreviations $p_{234} \coloneqq p_{\{2,3\}}+p_{\{3,4\}}+p_{\{4\}}$ and $p_{34} \coloneqq p_{\{3,4\}}+p_{\{4\}}$. Consider some vector $(\gamma_S)_{S\in\mathcal{S}}$ such that $\sum_{S\in\mathcal{S}} \gamma_S= N\lambda^* = N\mu$. Appying Theorem~\ref{th:general_results_extension}, then yields
\begin{equation}
\epsilon\left[
\begin{array}{c}
Q_{\{1,2\}}\\
Q_{\{2,3\}}\\
Q_{\{3,4\}}\\
Q_{\{4\}}
\end{array}
\right]
\overset{d}{\rightarrow}
\left[
\begin{array}{l}
p_{\{1,2\}}U_1\\
p_{\{2,3\}}U_1 + \frac{N\mu p_{\{2,3\}}}{\gamma_{234}}U_2\\
p_{\{3,4\}}U_1 + \frac{N\mu p_{\{3,4\}}}{\gamma_{234}}U_2 + \frac{N\mu p_{\{3,4\}}}{\gamma_{34}}U_3\\
p_{\{4\}}U_1 + \frac{N\mu p_{\{4\}}}{\gamma_{234}}U_2  + \frac{N\mu p_{\{4\}}}{\gamma_{34}}U_3 + \frac{N\mu p_{\{4\}}}{\gamma_{\{4\}}}U_4
\end{array}
\right]
\end{equation}
as $\epsilon\downarrow 0$ where we used the abbreviations $\gamma_{234} \coloneqq \gamma_{\{2,3\}}+\gamma_{\{3,4\}}+\gamma_{\{4\}}$ and $\gamma_{34} \coloneqq \gamma_{\{3,4\}}+\gamma_{\{4\}}$.
Hence the job types which belong to the CRP components further away from the root of the line graph are more affected by direction along which the boundary point of the stability region is approached. Roughly speaking the \emph{excess} jobs of a job type-$\{i,i+1\}$ can only be processed by server $i+1$ if server $i$ is too busy. This influences the allocation of type-$\{i+1,i+2\}$ jobs, which leads to a cascading event.

\end{example}

\begin{figure}[h]
\centering
\begin{subfigure}[b]{0.49\textwidth}
\centering
\includegraphics[scale=1]{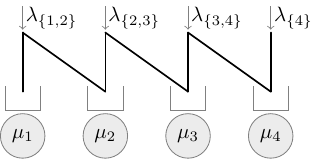}
\caption{Sketch of the system.}\label{fig:vb_extN}
\end{subfigure}
\begin{subfigure}[b]{0.49\textwidth}
\centering
\includegraphics[scale = 1]{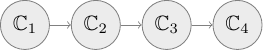}
\subcaption{Associated DAG.}\label{fig:DAG_extN}
\end{subfigure}
\caption{An illustrative example with model parameters and CRP components as defined in Example~\ref{exam:extN}.}\label{fig:extN}
\end{figure}

For larger examples, with more complicated DAG structures, it becomes less obvious how the choice of the parameters $(\gamma_S)_{S\in\mathcal{S}}$ influences the limiting representation of the vector of queue lengths compared to the settings discussed in Section~\ref{sec:main_results}.

The proof of Theorem~\ref{th:general_results_extension} for the c.o.c.\ mechanism follows a similar two-step framework as the proof of Theorem~\ref{th:general_results}.
Therefore we will only give a broad outline of the proof and mention some of the differences. The proof outline for the c.o.s. mechanism is provided in Appendix~\ref{app:main_results_cos}.\\

First, a convergence result similar to Proposition~\ref{th:limit_gf} must be established in terms of the PGF of the numbers of jobs of the various types. 
We define
\begin{equation}
\label{eq:omega_T}
\omega(\boldsymbol{T}) \coloneqq \beta(\boldsymbol{T}) \prod\limits_{j\in\text{CR}(\boldsymbol{T})} \frac{\mu(\boldsymbol{T},j)}{\gamma(\boldsymbol{T},j)}
\quad
\text{and} 
\quad
\omega(\mathcal{N}_K) \coloneqq \sum\limits_{\boldsymbol{T}\in\mathcal{N}_K} \omega(\boldsymbol{T})
\end{equation}
for any $\boldsymbol{T}\in\mathcal{N}_K$.

\begin{theorem}[Convergence of PGFs - generalization]\label{th:limit_gf_generalization}
With the definitions as in Section~\ref{sec:preliminaries} and the \textup{(}pre-limit\textup{)} arrival rates as in~\eqref{eq:HT_trajectory},
the joint PGF of the \textup{(}scaled\textup{)} numbers of jobs of the various types, i.e., $\mathbb{E}\left[\prod_{S\in\mathcal{S}} z_S^{Q_S}\right]$ with $z_S = \exp\left(-\epsilon t_S\right)$ converges to
\begin{equation}\label{eq:GF_gen}
\sum\limits_{\boldsymbol{T}\in\mathcal{N}_K} \frac{\omega(\boldsymbol{T})}{\omega(\mathcal{N}_K)}
\prod\limits_{i\in \text{CR}(\boldsymbol{T})}\left(1+ \sum_{j=1}^i t_{T_j}\frac{N\lambda^*p_{T_j}}{\gamma(\boldsymbol{T},i)}\right)^{-1}
\end{equation}
as $\epsilon\downarrow 0$ with $t_S \ge 0$ for all $S\in\mathcal{S}$.
\end{theorem}

The proof is similar to the proof of Proposition~\ref{th:limit_gf}, where we now use $\epsilon$ instead of $1-\frac{\lambda}{\lambda^*}$ and the fact that $\gamma(\boldsymbol{T},i)$ no longer has to be equal to $\mu(\boldsymbol{T},i) = N\lambda^* p (\boldsymbol{T},i)$ for $i\in\text{CR}(\boldsymbol{T})$.

Second, from the above theorem we can deduce that $(X_S)_{S\in\mathcal{S}}$ follows a mixture distribution where the mixture weights are determined by $\omega(\boldsymbol{T})/\omega(\mathcal{N}_K)$ for $\boldsymbol{T}\in\mathcal{N}_K$ and the mixture components consist of linear combinations of independent exponentially distributed random variables. 

Define the indices $i_1,\dots,i_K$ such that  $\{i_1,\dots,i_K\} = \mathrm{CR}(\boldsymbol{T})$ with $i_1<i_2<\dots<i_K$ and $j_{S}(\boldsymbol{T})$ the position of $S$ in $\boldsymbol{T}$ such that $T_{j_{S}(\boldsymbol{T})} = S$ for any $S\in\mathcal{S}$ and $\boldsymbol{T}\in\mathcal{N}_K$.
Let $\left(I_{\boldsymbol{T}}\right)_{\boldsymbol{T}\in\mathcal{N}_K}
= \boldsymbol{e}(\tilde{\boldsymbol{T}})$ with probability $\frac{\omega(\tilde{\boldsymbol{T}})}{\omega(\mathcal{N}_K)}$ for any $\tilde{\boldsymbol{T}}\in\mathcal{N}_K$ and define $\boldsymbol{e}(\tilde{\boldsymbol{T}})$ as a $|\mathcal{N}_K|$-dimensional unit vector with a one entry at the location corresponding to the ordered vector $\tilde{\boldsymbol{T}}$, $\boldsymbol{T}\in\mathcal{N}_K$. Then, the \textup{(}random\textup{)} coefficients are given by
\begin{equation}
\Gamma_{k,S} \coloneqq \sum\limits_{\boldsymbol{T}\in\mathcal{N}_K} I_{\boldsymbol{T}} \frac{N\lambda^*p_S}{\gamma(\boldsymbol{T},i_k)}\mathds{1}\{i_k \ge j_S(\boldsymbol{T})\},
\end{equation}
for any $S\in\mathcal{S}$ and $k=1,\dots,K$, such that
\begin{equation}\label{eq:X_s_Gen}
(X_S)_{S\in\mathcal{S}}
\overset{d}{=}
\left( \sum\limits_{k=1}^K \Gamma_{k,S} U_k 
\right)_{S\in\mathcal{S}}
\end{equation}
with $U_1,\dots,U_K$ independent and exponentially distributed random variables with unit mean.

\begin{proposition}[Convergence to a mixture distribution - generalization] \label{cor:mixture_dist_gen}
The \textup{(}scaled\textup{)} vector of queue lengths, $\epsilon(Q_S)_{S\in\mathcal{S}}$, converges in distribution to a random vector associated with the mixture distribution as specified above, i.e.,
\begin{equation}
\epsilon \left( Q_S \right)_{S\in\mathcal{S}}
\overset{d}{\rightarrow}
(X_S)_{S\in\mathcal{S}}.
\end{equation}
\end{proposition}

We emphasize that, even though the (pre-limit) arrival rates change, the critical arrival rate to the system, $\lambda^*$ will remain the same as in the previous sections. Hence, the sets of critical job types, $\mathcal{CR}(\mathcal{S})$, the CRP components and the associated DAG will be also be the same, even when $(\gamma_S)_{S\in\mathcal{S}}$ differs from $(N\lambda^* p_S)_{S\in\mathcal{S}}$.\\

After establishing the mixture distribution, we show that it (and its Laplace transform) can be rewritten or simplified, as formalized in the following proposition.

\begin{proposition}\label{prop:equality_GF_generalization}
For $t_S\ge 0$ for all $S\in\mathcal{S}$ and under Assumption~\ref{cond:root}, it holds that
\begin{equation}
\mathbb{E}\left[ \prod\limits_{S\in\mathcal{S}} \mathrm{e} ^{-t_S X_S}\right] = \mathbb{E}\left[ \prod\limits_{S\in\mathcal{S}} \mathrm{e} ^{-t_S Y_S}\right],
\end{equation}
with 
the former Laplace transform as in~\eqref{eq:GF_gen} and 
\begin{equation}\label{eq:equality_GF_generalization}
\mathbb{E}\left[ \prod\limits_{S\in\mathcal{S}} \mathrm{e} ^{-t_S Y_S}\right] 
= \prod\limits_{k=1}^K \left( 1+ \sum\limits_{S\in V_k} t_S \frac{N\lambda^*p_S}{\gamma(V_k)}\right)^{-1},
\end{equation}
such that 
\begin{equation}
\left(X_S \right)_{S\in\mathcal{S}} \overset{d}{=} \left(Y_S \right)_{S\in\mathcal{S}}.
\end{equation}
\end{proposition}

Lemmas~\ref{lem:N_K_versus_Sigma} and~\ref{lem:mixture_components} are still applicable without any alterations and will be the building blocks of the proof. However, the expression for the aggregation of the mixture weights according to the $\sigma$-ordered  vectors, $\boldsymbol{T}^{\sigma}$, does depend on $(\gamma_S)_{S\in\mathcal{S}}$. Indeed,
\begin{equation}
\sum\limits_{\boldsymbol{T}\in \boldsymbol{T}^{\sigma}} \frac{\omega(\boldsymbol{T})}{\omega(\mathcal{N}_K)}
=
\prod_{k=1}^K
\frac{\gamma(V_k)}{\gamma(\mathcal{C}_{\sigma(1)},\dots,\mathcal{C}_{\sigma(k)})}
\end{equation}
for any $\sigma\in\Sigma_K$, which serves as an alternative result for Proposition~\ref{prop:prob_t_sigma}. This, together with the proof outline of Proposition~\ref{prop:equality_GF}, is sufficient to establish~\eqref{eq:equality_GF_generalization}.

\section{Main results - redundancy c.o.s.}\label{app:main_results_cos}

Since the result in Theorem~\ref{th:general_results} is a special case of the result in Theorem~\ref{th:general_results_extension}, we will only focus on the latter to show its validity in case of the c.o.s.\ mechanism. The proof follows the same outline as the proof for the c.o.c.\ mechanism. 
The main difference is the limiting form of the PGF of the numbers of waiting jobs of the various types, which will be provided below for completeness.

Let $\boldsymbol{u}$ be an ordered vector of a subset of the servers which represents the idle servers and the order in which they became idle, then 
\begin{equation}\label{eq:alpha}
\alpha(\boldsymbol{u})\coloneqq \prod\limits_{l=1}^{|\boldsymbol{u}|} \mu_{u_l}\left(N\lambda^* \sum\limits_{S\colon S\cap\{u_1,\dots,u_l\}\neq \emptyset} p_S \right)^{-1}.
\end{equation}
The set of all ordered idle servers with no compatibilities among the job types in $\boldsymbol{T}$ is denoted by $\mathcal{E}(\boldsymbol{T})$.

\begin{theorem}[Convergence of PGFs - c.o.s., generalization] \label{th:limit_gf_cos_gen_HT}
With the definitions as in Section~\ref{sec:preliminaries} and the \textup{(}pre-limit\textup{)} arrival rates as in~\eqref{eq:HT_trajectory}, the joint PGF of the \textup{(}scaled\textup{)} numbers of \emph{waiting} jobs of the various types, i.e., $\mathbb{E} \left[ \prod_{S\in\mathcal{S}} z_S^{\tilde{Q}_S}\right]$ with $z_S = \exp \left(-\epsilon t_S\right)$ converges to
\begin{equation}\label{eq:gf_cos_gen}
\sum\limits_{\boldsymbol{T}\in\mathcal{N}_K}  
\sum\limits_{\boldsymbol{u}\in\mathcal{E}(\boldsymbol{T})}
 \frac{\alpha(\boldsymbol{u})
  \omega(\boldsymbol{T})}{\hat{\omega}}
 \prod\limits_{i\in\text{CR}(\boldsymbol{T})}
 \left(1+\sum\limits_{j=1}^i t_{T_j}\frac{N\lambda^* p_{T_j}}{\gamma(T_1,\dots,T_i)}
 \right)^{-1}
\end{equation}
as $\lambda\uparrow\mu$ with $t_S\ge 0$ for all $S\in\mathcal{S}$ where $\mathcal{E}(\boldsymbol{T})$ represents the collection of vectors of servers which are not compatible to any of the job types in~$\boldsymbol{T}$.
We define $\omega(\boldsymbol{T})$ and $\alpha(\boldsymbol{u})$ as in~\eqref{eq:omega_T} and~\eqref{eq:alpha}, respectively,
and
\begin{equation}
\hat{\omega} = 
\sum\limits_{\boldsymbol{T}\in\mathcal{N}_K}  
\sum\limits_{\boldsymbol{u}\in\mathcal{E}(\boldsymbol{T})}
 \alpha(\boldsymbol{u})
 \omega(\boldsymbol{T}).
\end{equation}
\end{theorem}

The proof uses similar arguments as the proof of Proposition~\ref{th:limit_gf}.

\section{Alternative proofs of the main results}\label{sec:alternative proof}
\subsection{Pre-limit characterization of the system} \label{sec:interpretation}
Before turning our attention to the heavy-traffic limit, we will investigate  in this section the pre-limit stationary behavior of the system. Specifically, we will derive a characterization of the total number of jobs of the various types in terms of weighted sums of geometrically distributed random variables whose parameters depend on the model parameters and the order in which the different job types occur. 
It is worth emphasizing that simple, closed-form expressions for the stationary distribution at the level of the number of jobs (of the various types) still seem out of reach, despite this characterization.

We introduce the following two concepts to keep track of the number of jobs (of the various types) in the system given an ordered vector~$\boldsymbol{T}$.
\begin{definition}[Jobs and segments]
Given an ordered vector of job types~$\boldsymbol{T}$, define $Q^j\mid\boldsymbol{T}$ as the number of jobs between the first occurrences of a type-$T_j$ and a type-$T_{j+1}$ job, for $j=1,\dots,|\boldsymbol{T}|-1$.
Analogously, $Q^{|\boldsymbol{T}|}\mid\boldsymbol{T}$ denotes the number of jobs that arrived after the first $T_{|\boldsymbol{T}|}$-type of job.
We refer to $Q^j\mid\boldsymbol{T}$ as the numbers of jobs in the $j$th segment given $\boldsymbol{T}$. Moreover, define $Q^{j}_{T_i}\mid  \boldsymbol{T}$ as the number of type-$T_i$ jobs in the $j$th segment, for $i=1,\dots,j$ and $j=1,\dots,|\boldsymbol{T}|$.
\end{definition}

\begin{example}  \label{example:pre_limit_def}
Consider the system in Example~\ref{ex:example_main_result} and let us illustrate the above definition by fixing the ordered vector $\boldsymbol{T} = [T_1,T_2,T_3,T_4] = [ \{1\},\{3\},\{3,4\},\{1,2,3\}]$. The state of the system could be given by
\begin{equation}
\boldsymbol{c} = (T_1,\underbrace{T_1}_{\boldsymbol{R}^1},T_2,
T_3, \underbrace{T_1,T_2,T_2,T_3,T_1}_{\boldsymbol{R}^3},
T_4, \underbrace{T_1,T_4}_{\boldsymbol{R}^4}),
\end{equation}
where the type labels are ordered such that the first label corresponds to the oldest job in the system, etc. There are four segments of which the second is of size zero and the other three are denoted by $\boldsymbol{R}^1$, $\boldsymbol{R}^3$ and $\boldsymbol{R}^4$. So the numbers of jobs in the various segments are given by
$\left(Q^1,Q^2,Q^3,Q^4 \mid \boldsymbol{T},\boldsymbol{c} \right) = (|\boldsymbol{R}^1|,0,|\boldsymbol{R}^3|,|\boldsymbol{R}^4|) = (1,0,5,2).$
Considering each segment separately, the numbers of jobs of the various types are given by
\begin{equation}
\begin{array}{rcl}

(Q^1_{T_1} \mid \boldsymbol{T},\boldsymbol{c} ) & = & (1)\\
(Q^2_{T_1}, Q^2_{T_2} \mid \boldsymbol{T},\boldsymbol{c}) & = & (0,0)\\
(Q^3_{T_1}, Q^3_{T_2}, Q^3_{T_3} \mid \boldsymbol{T},\boldsymbol{c}) & = & (2,2,1)\\
(Q^4_{T_1}, Q^4_{T_2}, Q^4_{T_3},Q^4_{T_4} \mid \boldsymbol{T},\boldsymbol{c}) & = & (1,0,0,1).
\end{array}
\end{equation}
\end{example}

\begin{theorem}\label{th:prelimit_distributions}
For a given $\boldsymbol{T}\in\mathcal{N}$, the following \textup{(}pre-limit\textup{)} characterizations hold.
\begin{enumerate}[label=\alph*.]
\item Job types in a segment: 
$Q^j_{T_i}\mid\boldsymbol{T}$ is geometrically distributed with parameter
\begin{equation}\label{eq:param_geometric_type}
p^{\boldsymbol{T},j,i} \coloneqq 
\frac{\frac{N\lambda p_{T_i}}{\mu(\boldsymbol{T},j)}}{1-\frac{N\lambda p(\boldsymbol{T},j)}{\mu(\boldsymbol{T},j)} + \frac{N\lambda p_{T_i}}{\mu(\boldsymbol{T},j)}},
\end{equation}
for all $i=1,\dots,j$ and $j=1,\dots,|\boldsymbol{T}|$.
\item Jobs in a segment: $Q^j\mid\boldsymbol{T}$ is geometrically distributed with parameter
\begin{equation}\label{eq:param_geometric}
p^{\boldsymbol{T},j} \coloneqq \frac{N\lambda p(\boldsymbol{T},j)}{\mu(\boldsymbol{T},j)},
\end{equation}
for all $j=1,\dots,|\boldsymbol{T}|$. 
\item State configurations: $\mathbb{P}(\boldsymbol{T}) = C\cdot h(\boldsymbol{T},\boldsymbol{1})$ with normalization constant $C^{-1} = \sum_{\boldsymbol{\tilde{T}}\in\mathcal{N}} h(\boldsymbol{\tilde{T}},\boldsymbol{1})$ and $h(\boldsymbol{T},\boldsymbol{1})$ as in~\eqref{eq:def_h_T_z}.
\item The numbers of jobs in different segments are independent.
\item The numbers of jobs of the various types in a particular segment given $\boldsymbol{T}$ and the total number of jobs in that segment, i.e., $(Q^j_{T_i} \mid Q^j, \boldsymbol{T})_{i=1,\dots,j}$, follows a multinomial distribution with parameters $Q^j$ and $(p_{T_i} / p(\boldsymbol{T},j))_{i=1,\dots,j}$ for all $j=1,\dots,|\boldsymbol{T}|$.
\end{enumerate}
\end{theorem}

Note that $p^{\boldsymbol{T},j},p^{\boldsymbol{T},j,i} \in(0,1)$ for all $i=1,\dots,j$, $j=1\dots,|\boldsymbol{T}|$ and $\boldsymbol{T}\in\mathcal{N}$ due to the stability conditions (Condition~\ref{cond:stab}). The proof of Theorem~\ref{th:prelimit_distributions} relies on the detailed stationary distribution provided in~\citep{Gardner2016queueing} and a sequential aggregation of the states.

\begin{proof}[Proof of Theorem~\ref{th:prelimit_distributions}]
To prove the statements in Theorem~\ref{th:prelimit_distributions} we start from the detailed stationary distribution provided in~\citep{Gardner2016queueing}
and sequentially aggregate the states. The product-form expression is given by
\begin{equation} \label{eq:PF_coc}
\pi (\boldsymbol{c}) = C\prod\limits_{i=1}^{|\boldsymbol{c}|} \frac{N\lambda p_{c_i}}{\mu(c_1,\dots,c_i)},
\end{equation}
with $\boldsymbol{c} = (c_1,\dots,c_n)$ representing the state of the system. 
Note that for a given ordered vector of job types~$\boldsymbol{T}$, a state $\boldsymbol{c}$ satisfying this order can be rewritten as
\begin{equation} \label{eq:enumeration_states}
\boldsymbol{c} =  (T_1,\boldsymbol{R}^1,T_2,\boldsymbol{R}^2,\dots,T_{|\boldsymbol{T}|},\boldsymbol{R}^{|\boldsymbol{T}|}), 
\end{equation}
where $ \boldsymbol{R}^j = (R^j_i,\dots,R_{n^j}^{j})$ with $R_n^j \in\{T_1,\dots,T_j\}$ for all $n\le n^j$, $n^j\in\mathbb{N}$ and $j\le |\boldsymbol{T}|$. Hence, we can also write
\begin{equation}
\pi (\boldsymbol{c}) = C\prod\limits_{j=1}^{|\boldsymbol{T}|} \frac{N\lambda p_{T_j}}{\mu(\boldsymbol{T},j)}  \prod\limits_{i=1}^{|\boldsymbol{R}^j|} \frac{N\lambda p_{R^j_i}}{\mu(\boldsymbol{T},j)},
\end{equation}
where we subdivide the factors of each job type in~\eqref{eq:PF_coc} into different categories, where either the corresponding job type label occurs for the first time or the corresponding job is part of the $j$th segment for some~$j$.

Consider a (non-empty) subset $\mathcal{J}$ of the indices $\{1,\dots,|\boldsymbol{T}|\}$.  Let $\boldsymbol{n}^j\in\mathbb{N}^j$ for all $j\in\mathcal{J}$. To obtain an expression for the joint stationary probability that $\left( \left( Q^j_{T_i} \mid \boldsymbol{T}\right)_{i=1}^j\right)_{j\in\mathcal{J}} =  \left( \left( n^j_i \right)_{i=1}^j\right)_{j\in\mathcal{J}}$, we have to aggregate over all states with
$\boldsymbol{R}^j$ such that $\sum_{k=1}^{|\boldsymbol{R}^j|} \mathds{1}\{R^j_k =  T_i\} = n^j_i$ for all $i=1,\dots,j$ and $j\in\mathcal{J}$ and $\boldsymbol{R}^j$ can be anything for all $j\notin\mathcal{J}$, which is a shorthand notation for the indices $j\in \{1,\dots,|\boldsymbol{T}|\}\setminus \mathcal{J}$. Moreover, we define $|\boldsymbol{n}^j|\coloneqq n^j_1+n^j_2+\dots + n^j_j$.

Hence,
\begin{equation}\label{eq:prelimit_full}
\begin{array}{rl}
& \mathbb{P} \left( \left( \left( Q^j_{T_i} \right)_{i=1,\dots,j}\right)_{j\in\mathcal{J}} =  \left( \left( n^j_i \right)_{i=1,\dots,j}\right)_{j\in\mathcal{J}} , \boldsymbol{T}  \right)\\

= & C \prod\limits_{j=1}^{|\boldsymbol{T}|} \frac{N\lambda p_{T_j}}{\mu(\boldsymbol{T},j)}
\prod\limits_{j\in\mathcal{J}}
\left[ \binom{|\boldsymbol{n}^j|}{n^j_1,\dots,n^j_j}  \prod\limits_{i=1}^j \left(\frac{N\lambda p_{T_i}}{\mu(\boldsymbol{T},j)}\right)^{n^j_i} \right]

\prod\limits_{j\notin\mathcal{J}}
\sum\limits_{n^j = 0}^{\infty} 
\left[ \sum\limits_{n^j_1+\dots+n^j_j= n^j}\binom{n^j}{n^j_1,\dots,n^j_j}  \prod\limits_{i=1}^j \left(\frac{N\lambda p_{T_i}}{\mu(\boldsymbol{T},j)}\right)^{n^j_i} \right] \\

= & C \prod\limits_{j=1}^{|\boldsymbol{T}|} \frac{N\lambda p_{T_j}}{\mu(\boldsymbol{T},j)}
\prod\limits_{j\in\mathcal{J}}
\left[ \binom{|\boldsymbol{n}^j|}{n^j_1,\dots,n^j_j}  \prod\limits_{i=1}^j \left(\frac{N\lambda p_{T_i}}{\mu(\boldsymbol{T},j)}\right)^{n^j_i} \right]
 \prod\limits_{j\notin\mathcal{J}}
\sum\limits_{n^j = 0}^{\infty} 
 \left(\frac{N\lambda p(\boldsymbol{T},j)}{\mu(\boldsymbol{T},j)}\right)^{n^j}\\

= & C \prod\limits_{j=1}^{|\boldsymbol{T}|} \frac{N\lambda p_{T_j}}{\mu(\boldsymbol{T},j)}
\prod\limits_{j\in\mathcal{J}}
\left[ \binom{|\boldsymbol{n}^j|}{n^j_1,\dots,n^j_j}  \prod\limits_{i=1}^j \left(\frac{N\lambda p_{T_i}}{\mu(\boldsymbol{T},j)}\right)^{n^j_i} \right]
 \prod\limits_{j\notin\mathcal{J}}
  \left(1-\frac{N\lambda p(\boldsymbol{T},j)}{\mu(\boldsymbol{T},j)}\right)^{-1}  .
\end{array}
\end{equation}

Let $n^j\in\mathbb{N}$ for $j\in\mathcal{J}$. In order to obtain the (joint) stationary probability that
$\left( Q^j\mid \boldsymbol{T} \right)_{j\in\mathcal{J}} =\left( n^j \right)_{j\in\mathcal{J}}$, we aggregate the above states further. In particular, we sum over all vectors $\boldsymbol{n}^j$ such that $|\boldsymbol{n}^j| = n^j$ to obtain
\begin{equation}\label{eq:prelimit_segment}
\begin{array}{rl}
& \mathbb{P} \left( \left( Q^j \right)_{j\in\mathcal{J}} =  \left(  n^j \right)_{j\in\mathcal{J}} , \boldsymbol{T}  \right)\\

= & C \prod\limits_{j=1}^{|\boldsymbol{T}|} \frac{N\lambda p_{T_j}}{\mu(\boldsymbol{T},j)}
\prod\limits_{j\in\mathcal{J}}
\sum\limits_{\substack{\boldsymbol{n}^j\in\mathbb{N}^j\colon\\ |\boldsymbol{n}^j| = n^j}}
\left[ \binom{|\boldsymbol{n}^j|}{n^j_1,\dots,n^j_j}  \prod\limits_{i=1}^j \left(\frac{N\lambda p_{T_i}}{\mu(\boldsymbol{T},j)}\right)^{n^j_i} \right]
 \prod\limits_{j\notin\mathcal{J}}
  \left(1-\frac{N\lambda p(\boldsymbol{T},j)}{\mu(\boldsymbol{T},j)}\right)^{-1} \\
  
  = & C \prod\limits_{j=1}^{|\boldsymbol{T}|} \frac{N\lambda p_{T_j}}{\mu(\boldsymbol{T},j)}
\prod\limits_{j\in\mathcal{J}}
\left(  \sum\limits_{i=1}^j \frac{N\lambda p_{T_i}}{\mu(\boldsymbol{T},j)}\right)^{n^j}
 \prod\limits_{j\notin\mathcal{J}}
  \left(1-\frac{N\lambda p(\boldsymbol{T},j)}{\mu(\boldsymbol{T},j)}\right)^{-1} \\
  
= & C \prod\limits_{j=1}^{|\boldsymbol{T}|} \frac{N\lambda p_{T_j}}{\mu(\boldsymbol{T},j)}
\prod\limits_{j\in\mathcal{J}}
\left(   \frac{N\lambda p(\boldsymbol{T},j)}{\mu(\boldsymbol{T},j)}\right)^{n^j}
 \prod\limits_{j\notin\mathcal{J}}
  \left(1-\frac{N\lambda p(\boldsymbol{T},j)}{\mu(\boldsymbol{T},j)}\right)^{-1} .
\end{array}
\end{equation}
Next, we add all the above states for all $n^j\in\mathbb{N}$ for all $j\in\mathcal{J}$ to obtain the stationary probability that the system state complies with the ordered vector $\boldsymbol{T}$, i.e.,
\begin{equation}
\begin{array}{rcl}
 \mathbb{P} \left(  \boldsymbol{T}  \right) & =& 

 C \prod\limits_{j=1}^{|\boldsymbol{T}|} \frac{N\lambda p_{T_j}}{\mu(\boldsymbol{T},j)}
\prod\limits_{j\in\mathcal{J}} 
\sum\limits_{n^j = 0}^{\infty}
\left(   \frac{N\lambda p(\boldsymbol{T},j)}{\mu(\boldsymbol{T},j)}\right)^{n^j}
 \prod\limits_{j\notin\mathcal{J}}
  \left(1-\frac{N\lambda p(\boldsymbol{T},j)}{\mu(\boldsymbol{T},j)}\right)^{-1} \\
  
  & =& 
 C \prod\limits_{j=1}^{|\boldsymbol{T}|} \frac{N\lambda p_{T_j}}{\mu(\boldsymbol{T},j)}
\prod\limits_{j\in\mathcal{J}} 

\left(  1-  \frac{N\lambda p(\boldsymbol{T},j)}{\mu(\boldsymbol{T},j)}\right)^{-1}
 \prod\limits_{j\notin\mathcal{J}}
  \left(1-\frac{N\lambda p(\boldsymbol{T},j)}{\mu(\boldsymbol{T},j)}\right)^{-1} \\
  
  & =& 
 C \prod\limits_{j=1}^{|\boldsymbol{T}|} \frac{N\lambda p_{T_j}}{\mu(\boldsymbol{T},j)} 
\left(  1-  \frac{N\lambda p(\boldsymbol{T},j)}{\mu(\boldsymbol{T},j)}\right)^{-1}.
\end{array}
\end{equation}
This proves part~c of Theorem~\ref{th:prelimit_distributions}. Using this part and the computations in~\eqref{eq:prelimit_segment} we obtain
\begin{equation}
\begin{array}{rl}
 \mathbb{P} \left( \left( Q^j \right)_{j\in\mathcal{J}} =  \left(  n^j \right)_{j\in\mathcal{J}} \mid \boldsymbol{T}  \right)
=  
\prod\limits_{j\in\mathcal{J}}
 \left(1-\frac{N\lambda p(\boldsymbol{T},j)}{\mu(\boldsymbol{T},j)}\right)
\left(   \frac{N\lambda p(\boldsymbol{T},j)}{\mu(\boldsymbol{T},j)}\right)^{n^j}
=  
\prod\limits_{j\in\mathcal{J}}
 \left(1-p^{\boldsymbol{T},j}\right)
\left(  p^{\boldsymbol{T},j}\right)^{n^j}.
\end{array}
\end{equation}
We observe that the latter expression corresponds to the joint probability distribution of $|\mathcal{J}|$ independent geometric random variables with parameters $(p^{\boldsymbol{T},j})_{j\in\mathcal{J}}$. This concludes the proofs of parts b and d.

Using part~c and~\eqref{eq:prelimit_full} we observe that
\begin{equation}
 \mathbb{P} \left( \left( \left( Q^j_{T_i} \right)_{i=1,\dots,j}\right)_{j\in\mathcal{J}} =  \left( \left( n^j_i \right)_{i=1,\dots,j}\right)_{j\in\mathcal{J}} \mid \boldsymbol{T}  \right)\\
= 
\prod\limits_{j\in\mathcal{J}}
\left(1-\frac{N\lambda p(\boldsymbol{T},j)}{\mu(\boldsymbol{T},j)}\right)
 \binom{|\boldsymbol{n}^j|}{n^j_1,\dots,n^j_j}  \prod\limits_{i=1}^j \left(\frac{N\lambda p_{T_i}}{\mu(\boldsymbol{T},j)}\right)^{n^j_i},
\end{equation}
which holds for any $\mathcal{J} \subseteq \{1,\dots,|\boldsymbol{T}|\}$ and in particular for $\mathcal{J} = \{j\}$ with $1\le j \le |\boldsymbol{T}|$ such that
\begin{equation} \label{eq:uitdruk2}
 \mathbb{P} \left( \left( Q^j_{T_i} \right)_{i=1,\dots,j} =  \left(  n^j_i \right)_{i=1,\dots,j}\mid \boldsymbol{T}  \right)\\
= 
\left(1-\frac{N\lambda p(\boldsymbol{T},j)}{\mu(\boldsymbol{T},j)}\right)
 \binom{|\boldsymbol{n}^j|}{n^j_1,\dots,n^j_j}  \prod\limits_{i=1}^j \left(\frac{N\lambda p_{T_i}}{\mu(\boldsymbol{T},j)}\right)^{n^j_i}.
\end{equation}
It is worth emphasizing that the numbers of jobs of the various types within one segment are not independent as can be deduced from the above expression. In order to obtain $\mathbb{P}(Q^j_{T_i} = n^j_i \mid \boldsymbol{T})$ for some $1\le i \le j \le |\boldsymbol{T}|$, we sum over all vectors $(  n^j_{\tilde{i}} )_{{\tilde{i}}=1,\dots,j}$ with $n^j_i$ fixed. Hence,
\begin{equation}\label{eq:uitdruk1}
\begin{array}{rcl}
\mathbb{P}(Q^j_{T_i} = n^j_i \mid \boldsymbol{T})
& = & 
\left(1-\frac{N\lambda p(\boldsymbol{T},j)}{\mu(\boldsymbol{T},j)}\right)
\left(\frac{N\lambda p_{T_i}}{\mu(\boldsymbol{T},j)}\right)^{n^j_i}
\sum\limits_{n=0}^{\infty}
\sum\limits_{n_1^j + \dots + n^j_j = n - n^j_i}
 \binom{n+n^j_i}{n^j_1,\dots,n^j_j}  \prod\limits_{\substack{\tilde{i}=1\\\tilde{i}\neq i}}^j \left(\frac{N\lambda p_{T_{\tilde{i}}}}{\mu(\boldsymbol{T},j)}\right)^{n^j_{\tilde{i}}}\\
 & = & 
\left(1-\frac{N\lambda p(\boldsymbol{T},j)}{\mu(\boldsymbol{T},j)}\right)
\left(\frac{N\lambda p_{T_i}}{\mu(\boldsymbol{T},j)}\right)^{n^j_i}
\sum\limits_{n=0}^{\infty}
 \binom{n+n^j_i}{n,n^j_i}   \left(\frac{N\lambda (p(\boldsymbol{T},j) - p_{T_i}) }{\mu(\boldsymbol{T},j)}\right)^{n}\\
 
  & = & 
\left(1-\frac{N\lambda p(\boldsymbol{T},j)}{\mu(\boldsymbol{T},j)}\right)
\left(\frac{N\lambda p_{T_i}}{\mu(\boldsymbol{T},j)}\right)^{n^j_i}
 \left(1-\frac{N\lambda (p(\boldsymbol{T},j) - p_{T_i}) }{\mu(\boldsymbol{T},j)}\right)^{-(n^j_i+1)}\\
   & = & 
\left(1-p^{\boldsymbol{T},j,i}\right)
\left(p^{\boldsymbol{T},j,i}\right)^{n^j_i},
\end{array}
\end{equation}
which concludes the proof of part~a.

From \eqref{eq:uitdruk2} and \eqref{eq:uitdruk1}, we observe that
\begin{equation}
\mathbb{P}\left( (Q^j_{T_1},\dots,Q^j_{T_j}) = (n_1,\dots,n_j) \mid  Q^j = n^j, \boldsymbol{T} \right)  = \binom{n^j}{n_1,\dots,n_j} \prod\limits_{i=1}^{j}\left(\frac{ p_{T_i}}{p(\boldsymbol{T},j)}\right)^{n_i},
\end{equation}
with $n^j, n_1,\dots,n_i\in\mathbb{N}$ such that $n_1+\dots +n_j=n^j$ for some $1\le j \le |\boldsymbol{T}|$, which concludes the proof of part~e. 

\end{proof}


\begin{example}
Consider Example~\ref{example:pre_limit_def}, applying Theorem~\ref{th:prelimit_distributions} yields
\begin{equation}
\begin{array}{rcl}
\mathbb{P}(\boldsymbol{T}) & = & 
\frac{N\lambda p_{T_1}}{\mu}\frac{N\lambda p_{T_2}}{2\mu}\frac{N\lambda p_{T_3}}{3\mu}\frac{N\lambda p_{T_4}}{4\mu}
\bigl(1-\frac{N\lambda p_{T_1}}{\mu}\bigl)^{-1}
\bigl(1-\frac{N\lambda (p_{T_1}+p_{T_2})}{2\mu}\bigl)^{-1}
\bigl(1-\frac{N\lambda (p_{T_1}+p_{T_2}+p_{T_3})}{3\mu}\bigl)^{-1}
\bigl(1-\frac{N\lambda}{4\mu}\bigl)^{-1}\\

& = & \frac{1}{27} \bigl(\frac{\lambda}{\mu}\bigl)^4 \bigl(1-\frac{\lambda}{\mu}\bigl)^{-3} \bigl(1-\frac{5\lambda}{6\mu}\bigl)^{-1}.
\end{array}
\end{equation}
Moreover, we have, for instance, that $Q^2\mid\boldsymbol{T}$ is geometrically distributed with parameter $p^{\boldsymbol{T},2} = N\lambda (p_{T_1}+p_{T_2})/(2\mu) = \frac{5\lambda}{6\mu}$,
and 
$Q^3_{T_1}\mid\boldsymbol{T}$ is geometrically distributed with parameter
\begin{equation}
p^{\boldsymbol{T},3,1} = \frac{N\lambda p_{T_1}}{3\mu - N\lambda(p_{T_2}+p_{T_3})} = \frac{\lambda}{3\mu-2\lambda}.
\end{equation}
\end{example}

Observe that Theorem~\ref{th:prelimit_distributions} allows us to count the number of jobs of a fixed type $S\in\mathcal{S}$ given the ordered vector of job types~$\boldsymbol{T}$. Define $j_S(\boldsymbol{T})$ as the position of $S$ in $\boldsymbol{T}$ if $S \in \boldsymbol{T} $ and as $\infty$ otherwise, then
\begin{equation}\label{eq:observation_limit_Q_S}
Q_S \mid \boldsymbol{T}  = \mathds{1}\left\{ S \in \boldsymbol{T} \right\} + \sum_{j=j_S(\boldsymbol{T})}^{|\boldsymbol{T}|} Q^j_{S}.
\end{equation}

We now combine all of the above results to describe the distribution of the numbers of jobs of the various types $(Q_S)_{S\in\mathcal{S}}$, irrespective of the ordered vector of job types.

First we introduce the random vector $\boldsymbol{\alpha} = (\alpha_{\boldsymbol{T}})_{\boldsymbol{T}\in\mathcal{N}}$ to distinguish between the different ordered vectors~$\boldsymbol{T}$ that may occur. In particular,
$
\boldsymbol{\alpha} \equiv \boldsymbol{e}_{\boldsymbol{T}}$, 
 a $|\mathcal{N}|$-dimensional unit vector with a one entry at the position corresponding to $\boldsymbol{T}$, with probability $\mathbb{P}(\boldsymbol{T})$.
 
 Then we construct a matrix $M(\boldsymbol{T})\in\mathbb{N}^{|\boldsymbol{T}|\times|\boldsymbol{T}|}$ with random entries corresponding to the numbers of jobs of the various types in each of the $|\boldsymbol{T}|$ segments, 
 \begin{equation}\label{eq:M_T}
 M(\boldsymbol{T})_{i,j} =
 \left\{
 \begin{array}{lcl}
 Q^j_{T_i} \mid \boldsymbol{T} & & \text{if~} i\le j\\
 0 & & \text{if~} i> j
 \end{array}
 \right.
 \end{equation}
 for a given $\boldsymbol{T}\in\mathcal{N}$. Notice that the different columns correspond to the various (independent) segments (Theorem~\ref{th:prelimit_distributions}, parts~b and~d), and the rows correspond to the numbers of jobs of the various types in those segments (Theorem~\ref{th:prelimit_distributions}, part~a). The triangular part below the diagonal will always contain zeros.
 
Since $|\boldsymbol{T}|$ is possibly smaller than $|\mathcal{S}|$, $M(\boldsymbol{T})$ will be embedded into a larger matrix with the desired dimensions. Let $\tilde{M}(\boldsymbol{T}) \in \mathbb{N}^{|\mathcal{S}|\times|\boldsymbol{T}|}$ be such that
 \begin{equation}
 \label{eq:tilde_M_T}
 \tilde{M}(\boldsymbol{T}) = \left[
 \begin{array}{c}
 M(\boldsymbol{T})\\
 O
 \end{array}
 \right],
 \end{equation}
 where $O$ is an all-zeros matrix with dimensions $(|\mathcal{S}|-|\boldsymbol{T}|)\times |\boldsymbol{T}|$. Now, the rows in $\tilde{M}(\boldsymbol{T})$ are arranged according to the order of the job types in~$\boldsymbol{T}$, and ideally they would be arranged according to the set of all job types $(S)_{S\in\mathcal{S}}$. This can be accomplished by multiplying $\tilde{M}(\boldsymbol{T})$ to the left with a carefully designed permutation matrix $P(\boldsymbol{T})$ of dimensions $|\mathcal{S}|\times|\mathcal{S}|$, i.e.,
 \begin{equation}\label{eq:P_T}
 P(\boldsymbol{T})_{S,j} = \left\{
 \begin{array}{lcl}
 1 & & \text{if~} j=1,\dots,|\boldsymbol{T}|~\text{and}~T_j = S\\
 1 & & \text{if~} j=|\boldsymbol{T}|+1,\dots,|\mathcal{S}|~\text{and}~\bar{T}_{j-|\boldsymbol{T}|-1} = S\\
 0 & & \text{otherwise},
 \end{array}
 \right.
 \end{equation}
 permuting the $j$th row to the $S$th row for $j=1,\dots,|\mathcal{S}|$ and $S\in\mathcal{S}$. The first case ensures that $T_j=S$ is associated with the correct job type. The second case takes care of the job types that are not present in~$\boldsymbol{T}$ as $\bar{\boldsymbol{T}}$ is defined as an arbitrary, but fixed, order of the job types in~$\mathcal{S}\setminus\boldsymbol{T}$.

 \begin{example}
  With the notation as above, we can define $\tilde{M}(\boldsymbol{T})$ and $P(\boldsymbol{T})$ with $\boldsymbol{T}$ as in Example~\ref{example:pre_limit_def}, i.e., 
 \begin{equation}
 \tilde{M}(\boldsymbol{T}) = 
 \left[
 \begin{array}{cccc}
 Q^1_{T_1}\mid \boldsymbol{T} &  Q^2_{T_1}\mid \boldsymbol{T} &  Q^3_{T_1}\mid \boldsymbol{T}&  Q^4_{T_1}\mid \boldsymbol{T} \\
0 &  Q^2_{T_2}\mid \boldsymbol{T} &  Q^3_{T_2}\mid \boldsymbol{T} &  Q^4_{T_2}\mid \boldsymbol{T}\\
0 &  0 &  Q^3_{T_3}\mid \boldsymbol{T} &  Q^4_{T_3}\mid \boldsymbol{T}\\
0 & 0 & 0 &  Q^4_{T_4}\mid \boldsymbol{T}
 \end{array}
 \right]
 \quad
 \text{and}
 \quad
 P(\boldsymbol{T}) = 
 \left[
 \begin{array}{cccc}
1& 0 &  0 &  0 \\
0 &  0 & 0 & 1\\
0 &  1 &  0  & 0\\
0 &  0 &  1  & 0\\
 \end{array}
 \right].
 \end{equation}
 Hence,
 \begin{equation}
  P(\boldsymbol{T}) \tilde{M}(\boldsymbol{T}) = 
  \left[
 \begin{array}{cccc}
 Q^1_{T_1}\mid \boldsymbol{T} &  Q^2_{T_1}\mid \boldsymbol{T} &  Q^3_{T_1}\mid \boldsymbol{T}&  Q^4_{T_1}\mid \boldsymbol{T} \\
 0 & 0 & 0 &  Q^4_{T_4}\mid \boldsymbol{T}\\
 0 &  Q^2_{T_2}\mid \boldsymbol{T} &  Q^3_{T_2}\mid \boldsymbol{T} &  Q^4_{T_2}\mid \boldsymbol{T}\\
 0 &  0 &  Q^3_{T_3}\mid \boldsymbol{T} &  Q^4_{T_3}\mid \boldsymbol{T}\\
 \end{array}
 \right].
 \end{equation}
 Relying on the observation in~\eqref{eq:observation_limit_Q_S}, we can write
\begin{equation} 
 \left[
 \begin{array}{c}
 Q_{\{1\}}\mid \boldsymbol{T} \\
Q_{\{1,2,3\}}\mid \boldsymbol{T} \\
Q_{\{3\}}\mid \boldsymbol{T} \\
Q_{\{3,4\}}\mid \boldsymbol{T} \\
 \end{array}
 \right]
 = 
 \left[
 \begin{array}{c}
1 \\
1 \\
1 \\
1 \\
 \end{array}
 \right]+
P(\boldsymbol{T}) \tilde{M}(\boldsymbol{T})
  \left[
 \begin{array}{c}
1 \\
1 \\
1 \\
1 \\
 \end{array}
 \right]
 = 
  \left[
 \begin{array}{c}
 1 + Q^1_{T_1} + Q^2_{T_1}+Q^3_{T_1}+Q^4_{T_1}\mid \boldsymbol{T}\\
 1 + Q^4_{T_4}\mid \boldsymbol{T}\\
 1 + Q^2_{T_2} + Q^3_{T_2} + Q^4_{T_2} \mid \boldsymbol{T}\\
1+Q^3_{T_3} + Q^4_{T_3} \mid \boldsymbol{T} 
 \end{array}
 \right].
\end{equation}
  \end{example}
 
Define
 the vector $\mathds{1}(\boldsymbol{T}) \in \{0,1\}^{|\mathcal{S}|}$ such that $\mathds{1}(\boldsymbol{T})_S = \mathds{1}\{S\in\boldsymbol{T}\}$ and an all-ones vector of length~$|\boldsymbol{T}|$, $E(\boldsymbol{T})$.
 
\begin{theorem}
\label{th:probabilistic_interpretation}
With the notation as above, the numbers of jobs of the various types can be written in terms of geometrically distributed random variables defined in Theorem~\ref{th:prelimit_distributions},
\begin{equation}\label{eq:probabilistic_representation}
(Q_S)_{S\in\mathcal{S}} \overset{d}{=} \sum\limits_{T\in\mathcal{N}} \alpha_{\boldsymbol{T}} \left( \mathds{1}(\boldsymbol{T}) + P(\boldsymbol{T}) \tilde{M}(\boldsymbol{T}) E(\boldsymbol{T}) \right).
\end{equation}
\end{theorem} 
The above theorem gives a probabilistic representation of the numbers of jobs of the various types which agrees with the PGF in Proposition~\ref{prop:pgf}. We emphasize that $P(\boldsymbol{T})$ and $E(\boldsymbol{T})$ consist solely of 0 and 1 entries, while $\tilde{M}(\boldsymbol{T})$ contains geometrically distributed random variables where the various columns are independent (Theorem~\ref{th:prelimit_distributions}, parts~a and~d) and the sum of these columns yields again geometrically distributed random variables (Theorem~\ref{th:prelimit_distributions}, part~b). Hence $(Q_S)_{S\in\mathcal{S}}$ follows a mixture distribution, where each mixture component corresponds to a sum of geometrically distributed random variables.

\begin{remark}
The above probabilistic interpretation focuses on the redundancy c.o.c.\ policy. However, many of the above observations and results are also true for the redundancy c.o.s.\ policy with some minor adjustments. In particular, the results in parts a, b and d of Theorem~\ref{th:prelimit_distributions} still hold for the \emph{waiting} numbers of jobs in the various segments and part~c is alternatively formulated as
\begin{equation}\label{eq:P_T_cos}
\mathbb{P}_{\text{c.o.s.}}(\boldsymbol{T}) = C'h(\boldsymbol{T},\boldsymbol{1})k(\boldsymbol{T})
\end{equation}
with normalization constant $(C')^{-1} = \sum_{\boldsymbol{T}\in\mathcal{N}}h(\boldsymbol{T},\boldsymbol{1})k(\boldsymbol{T})$
and 
\begin{equation}
k(\boldsymbol{T}) \coloneqq \sum\limits_{\boldsymbol{u} \in \mathcal{E}(\boldsymbol{T})} \prod\limits_{l=1}^{|\boldsymbol{u}|} \frac{\mu_{u_l}}{\lambda_{\mathcal{C}(u_1, \dots, u_l)}}
\end{equation}
with $\mathcal{E}(\boldsymbol{T})$ the set of all ordered vectors of idle servers which are not compatible with any of the job types in~$\boldsymbol{T}$.
Consequently the statement in Theorem~\ref{th:probabilistic_interpretation} is also true for the numbers of waiting jobs of the various types, i.e., $(\tilde{Q}_S)_{S\in\mathcal{S}}$, when the weights of the unit vectors $\boldsymbol{\alpha}$ are updated according to~\eqref{eq:P_T_cos}.
\end{remark}

\subsection{Heavy-traffic characterization and Theorem~\ref{cor:mixture_dist}}
\label{subsec:alternative_proof_HT}

We now use the stochastic characterization of the system as derived in the previous subsection to investigate the heavy-traffic behavior. This analysis will result in an alternative proof of Theorem~\ref{cor:mixture_dist} along with a probabilistic interpretation.

We first identify the dominant state configurations in the heavy-traffic regime.
As it turns out, in general only the states with configurations with a maximum number of nested critical subsets will occur with non-zero probability in the heavy-traffic regime. 
\begin{lemma}[State configurations in heavy traffic]\label{lem:limit_stateprob}
It holds that
\begin{equation}\label{eq:limit_stateprob}
\mathbb{P}^*(\boldsymbol{T}) \coloneqq \lim\limits_{\lambda\uparrow\lambda^*} \mathbb{P}(\boldsymbol{T}) = 
\left\{
\begin{array}{ll}
\frac{\beta(\boldsymbol{T})}{\beta(\mathcal{N}_K)} & \text{if~} \boldsymbol{T} \in\mathcal{N}_K, \\
0 & \text{if~} \boldsymbol{T} \in\mathcal{N}_k, k<K,
\end{array}
\right.
\end{equation}
for any $\boldsymbol{T} \in\mathcal{N}_k, k=0,1,\dots,K$, with $\beta(\boldsymbol{T})$ and $\beta(\mathcal{N}_K)$ as in~\eqref{eq:beta} and~\eqref{eq:beta_N_K}, respectively.
\end{lemma} 
As already suggested by the notation, the above expression for the heavy-traffic limit of the probabilities $\mathbb{P}(\boldsymbol{T})$ coincides with~\eqref{eq:P_star_T} for $\boldsymbol{T}\in\mathcal{N}_K$.
 The proof of Lemma~\ref{lem:limit_stateprob} follows similar arguments as the proof of Proposition~\ref{th:limit_gf}. 
\begin{proof}
Fix any $\boldsymbol{T}\in\mathcal{N}_{k}$ for some $k=0,\dots,K$. Then we know from Theorem~\ref{th:prelimit_distributions} that
$\mathbb{P}(\boldsymbol{T})  
 = C \cdot h(\boldsymbol{T},\boldsymbol{1})$,
with $h(\boldsymbol{T},\boldsymbol{1})$ as in~\eqref{eq:def_h_T_z} and $C$ a normalization constant. From Lemma~\ref{lem:intermediate_result} (with $\boldsymbol{z}\equiv \boldsymbol{1}$ or $\boldsymbol{t}\equiv \boldsymbol{0}$),
we observe that 
\begin{equation}
\begin{array}{rcccc}
\lim\limits_{\lambda\uparrow\lambda^*} \mathbb{P}(\boldsymbol{T})
&=&  \lim\limits_{\lambda\uparrow\lambda^*} \mathbb{P}(\boldsymbol{T}) \times \frac{\left(1-\frac{\lambda}{\lambda^*}\right)^K}{\left(1-\frac{\lambda}{\lambda^*}\right)^K}
&=& \frac{h^*(\boldsymbol{T},\boldsymbol{0}) \lim\limits_{\lambda\uparrow\lambda^*} \left(1-\frac{\lambda}{\lambda^*}\right)^{K-k}}{\sum\limits_{\tilde{k}=0}^K \sum\limits_{\tilde{\boldsymbol{T}}\in\mathcal{N}_{\tilde{k}}} h^*(\tilde{\boldsymbol{T}},\boldsymbol{0}) \lim\limits_{\lambda\uparrow\lambda^*} \left(1-\frac{\lambda}{\lambda^*}\right)^{K-\tilde{k}} } 
 \\ 
 &=& \frac{h^*(\boldsymbol{T},\boldsymbol{0})\mathds{1}\left\{ k = K \right\}}{\sum\limits_{\tilde{k}=0}^K \mathds{1}\left\{ \tilde{k} = K \right\} \sum\limits_{\tilde{\boldsymbol{T}}\in\mathcal{N}_{\tilde{k}}} h^*(\tilde{\boldsymbol{T}},\boldsymbol{0})  } 
 &=& \frac{h^*(\boldsymbol{T},\boldsymbol{0})}{\sum\limits_{\tilde{\boldsymbol{T}}\in\mathcal{N}_K} h^*(\tilde{\boldsymbol{T}},\boldsymbol{0})  } \mathds{1}\left\{ k = K \right\}
 \\
&=& \frac{\beta(\boldsymbol{T}) }{\tilde{\beta}(\mathcal{N}_K)} \mathds{1}\left\{ k = K \right\}. & &
\end{array}
\end{equation}
We used that $h^*(\boldsymbol{T},\boldsymbol{0}) = \beta(\boldsymbol{T})$ for any $\boldsymbol{T}$.
This concludes the proof.
\end{proof}

\begin{example}\label{exam:with_parameters} 
Consider the system in Example~\ref{ex:example_main_result}.
 Using Lemma~\ref{lem:limit_stateprob}, it can be computed that
\begin{equation}
\begin{array}{rcl}
\mathbb{P}^*([\{1\},\{3\},\{3,4\},\{1,2,3\}])& =& \frac{4}{9} \\
 \mathbb{P}^*([\{1\},\{3,4\},\{3\},\{1,2,3\}]) & =& \frac{2}{9} \\
\mathbb{P}^*([\{3\},\{3,4\},\{1\},\{1,2,3\}]) & =& \frac{2}{9} \\
\mathbb{P}^*([\{3,4\},\{3\},\{1\},\{1,2,3\}]) & =& \frac{1}{9}. \\
\end{array}
\end{equation}
\end{example}

Second, we focus on the number of jobs in each segment, i.e., $(Q^j\mid \boldsymbol{T})_{j=1}^{|\boldsymbol{T}|}$ for a given ordered vector $\boldsymbol{T}$.
As formalized below, the scaled number of jobs between the first occurrences of a type $T_j$ and $T_{j+1}$ job, $Q^j\mid \boldsymbol{T}$, will either converge to an exponentially distributed random variable or vanish in the heavy-traffic regime.

\begin{lemma}[Convergence - jobs in a segment]\label{lem:limit_segment}
Let $\boldsymbol{T}\in\mathcal{N}_k$ for some $k=0,1,\dots,K$ and consider $\text{CR}(\boldsymbol{T}) = \{i_1,\dots,i_k\}$ as in Definition~\ref{def:crit_vectors}. 
It holds for $l = 1,\dots, k$ that
\begin{equation}
 \Bigl(1-\frac{\lambda}{\lambda^*} \Bigl) Q^{i_l} \mid \boldsymbol{T} \overset{d}{\rightarrow} U_{i_l} 
\end{equation}
as $\lambda\uparrow\lambda^*$, with $(U_{i_l})_{l=1}^k$ independent exponentially distributed random variables with unit mean. For $j \in \{1,\dots, |\boldsymbol{T}|\} \setminus \text{CR}(\boldsymbol{T})$, we have
\begin{equation}
\Bigl(1-\frac{\lambda}{\lambda^*} \Bigl) Q^{j} \mid \boldsymbol{T} \overset{d}{\rightarrow} 0
\end{equation}
as $\lambda\uparrow\lambda^*$.
\end{lemma}

\begin{proof}
Notice that $p^{\boldsymbol{T},j} = \lambda/\lambda^*$ when $j\in \text{CR}(\boldsymbol{T})$, observe then the following limiting behavior of the parameter $p^{\boldsymbol{T},j}$ in \eqref{eq:param_geometric}: 
\begin{equation}\label{eq:limit_param_geom}
\lim_{\lambda\uparrow\lambda^*} p^{\boldsymbol{T},j} = 
\left\{
\begin{array}{l l}
1 & \text{if~} j \in \text{CR}(\boldsymbol{T}),\\
\frac{N\lambda^* p(\boldsymbol{T},j)}{\mu(\boldsymbol{T},j)} < 1 & \text{otherwise}.
\end{array}
\right.
\end{equation}
 Hence, \begin{equation}
\lim\limits_{\lambda\uparrow\lambda^*} \mathbb{E}\left[ \mathrm{e}^{-t\left(1-\lambda/\lambda^*\right)Q^j} \mid \boldsymbol{T}\right]
=
\lim\limits_{\lambda\uparrow\lambda^*} \frac{1-p^{\boldsymbol{T},j}}{1-p^{\boldsymbol{T},j}\mathrm{e}^{-t\left(1-\lambda/\lambda^*\right)}}
=\left\{
\begin{array}{l l}
\frac{1}{1+t} & \text{if~} j \in \text{CR}(\boldsymbol{T}),\\
1 & \text{otherwise},
\end{array}
\right.
\end{equation}
for any $t\ge 0$.
Noticing that $\frac{1}{1+t}$ is the Laplace transform of a unit-mean exponential random variable, the results follow from applying L\'evy's Continuity Theorem~\cite{Jacod2000}.
\end{proof}

Since some of the segments, i.e., $(Q^j \mid \boldsymbol{T} )_{j\notin \text{CR}(\boldsymbol{T})}$, will vanish after scaling, the numbers of jobs of the types $T_1,\dots,T_j$ in that segment will also vanish. For the numbers of jobs of the various types in the remaining segments, we can show the following convergence results. 

\begin{lemma}[Convergence - job types in a segment]\label{lem:limit_segment_type_geom}
Let $\boldsymbol{T}\in\mathcal{N}_k$ for some $k=0,1,\dots,K$ and $j=1,\dots,|\boldsymbol{T}|$, then
\begin{equation}
\Bigl(1-\frac{\lambda}{\lambda^*}\Bigl)( Q^j_{T_i} \mid \boldsymbol{T} )_{i=1}^j \overset{d}{\rightarrow} 
\left\{
\begin{array}{lcl}
U\bigl(\frac{p_{T_i}}{p(\boldsymbol{T},j)}\bigl)_{i=1}^j & & \text{if~} j\in\text{CR}(\boldsymbol{T}),\\
\boldsymbol{0} & & \text{otherwise,} 
\end{array}
\right.
\end{equation}
as $\lambda\uparrow\lambda^*$, for $i=1,\dots,j$ where $U$ denotes a unit-mean exponential random variable.
\end{lemma}

\begin{proof}
Using Theorem~\ref{th:prelimit_distributions}, we can derive the PGF of $\bigl(1-\frac{\lambda}{\lambda^*}\bigl)( Q^j_{T_i}\mid \boldsymbol{T} )_{i=1}^j$:
\begin{equation}
\begin{array}{rcl}
\mathbb{E}\Bigr[ \prod\limits_{i=1}^j z_{T_i}^{Q^j_{T_i}}\mid\boldsymbol{T}\Bigr]
& = & \sum\limits_{q=0}^{\infty}\mathbb{E}\Bigr[ \prod\limits_{i=1}^j z_{T_i}^{Q^j_{T_i}}\mid Q^j = q,\boldsymbol{T}\Bigr] \mathbb{P}( Q^j\mid\boldsymbol{T})\\
& = & \sum\limits_{q=0}^{\infty} \bigl( \sum_{i=1}^j \frac{p_{T_i}}{p(\boldsymbol{T},j)}z_{T_i}\bigl)^q (1-p^{\boldsymbol{T},j})\left(p^{\boldsymbol{T},j}\right)^q\\
& = & \frac{1-p^{\boldsymbol{T},j}}{1-\sum\limits_{i=1}^j p^{\boldsymbol{T},j}\frac{p_{T_i}}{p(\boldsymbol{T},j)}z_{T_i}},
\end{array}
\end{equation}
with $z_{T_i} = \exp\bigl(-\bigl(1-\frac{\lambda}{\lambda^*} \bigl)t_{T_i} \bigl)$. Relying on~\eqref{eq:limit_param_geom}, we observe that
\begin{equation}
\lim\limits_{\lambda\uparrow\lambda^*}\frac{1-p^{\boldsymbol{T},j}}{1-\sum\limits_{i=1}^j p^{\boldsymbol{T},j}\frac{p_{T_i}}{p(\boldsymbol{T},j)}z_{T_i}} = 
\left\{
\begin{array}{lcl}
\Bigl( 1+\sum\limits_{i=1}^j \frac{p_{T_i}}{p(\boldsymbol{T},j)} t_{T_i} \Bigl)^{-1} & & \text{if~} j\in\text{CR}(\boldsymbol{T}),\\
1 & & \text{otherwise.} 
\end{array}
\right.
\end{equation}
Noticing that $\bigl( 1+\sum_{i=1}^j \frac{p_{T_i}}{p(\boldsymbol{T},j)} t_{T_i} \bigl)^{-1}$ is the Laplace transform corresponding to $U\bigl(\frac{p_{T_i}}{p(\boldsymbol{T},j)}\bigl)_{i=1}^j$, the results follow by applying L\'evy's Continuity Theorem~\cite{Jacod2000}.
\end{proof}

From the above lemmas and the observation in~\eqref{eq:observation_limit_Q_S}, it follows that
\begin{equation}
\Bigl(1-\frac{\lambda}{\lambda^*}\Bigl) \left(Q_S \mid \boldsymbol{T} \right)_{S\in\mathcal{S}}
\overset{d}{\rightarrow} \sum\limits_{\substack{j = j_S(\boldsymbol{T})\\ j\in\text{CR}(\boldsymbol{T})}}^{|\boldsymbol{T}|} \frac{p_S}{p(\boldsymbol{T},j)}U_j
\end{equation}
as $\lambda \uparrow \lambda^*$, for any $\boldsymbol{T}\in\mathcal{N}_k$ for some $k=0,1,\dots,K$. Note that some of the terms in the above summations have no contribution, i.e., $j\notin \mathrm{CR}(\boldsymbol{T})$, as they may correspond to vanishing segments. In fact, at most $k$ of the terms are nonzero. Therefore we can rewrite the summation while focusing on those non-zero elements. Let $\text{CR}(\boldsymbol{T}) = \{i_1,\dots,i_k\}$ as in Definition~\ref{def:crit_vectors} with $i_1<\dots <i_k$, then
\begin{equation}\label{eq:convergence_type_level}
\Bigl(1-\frac{\lambda}{\lambda^*}\Bigl) \left(Q_S \mid \boldsymbol{T} \right)_{S\in\mathcal{S}}
\overset{d}{\rightarrow} 
\biggl(\sum\limits_{\substack{l = 1\\ i_l \ge j_S(\boldsymbol{T})}}^{k} \frac{p_S}{p(\boldsymbol{T},i_l)}U_{i_l}\biggl)_{S\in\mathcal{S}}.
\end{equation}

\begin{example}
Consider the system as outlined in Example~\ref{ex:example_main_result} and $\boldsymbol{T}= [\{1\},\{3\},\{3,4\},\{1,2,3\}]$, then $\text{CR}(\boldsymbol{T}) = \{i_1,i_2,i_3\} = \{1,3,4\}$. According to Lemma~\ref{lem:limit_segment}, we observe 
\begin{equation}
\Bigl(1-\frac{\lambda}{\lambda^*}\Bigl) \left( Q^1,Q^2,Q^3,Q^4 \mid \boldsymbol{T} \right) \overset{d}{\rightarrow} (U_1,0,U_2,U_3)
\end{equation}
as $\lambda\uparrow\lambda^*$ with $U_1, U_2$ and $U_3$ three independent unit-mean exponential random variables. The second segment is vanishing for the given ordered vector $\boldsymbol{T}$, since the limiting arrival rate of type-$\{1\}$ and type-$\{3\}$ jobs combined is strictly less than the aggregate service rate of their compatible servers, i.e., $N\lambda^*(p_{\{1\}}+p_{\{3\}}) = \frac{5}{3}\mu < 2\mu$.

Moreover, the scaled total number of jobs of the various types behaves as
\begin{equation}\label{eq:example_HT}
\Bigl(1-\frac{\lambda}{\lambda^*}\Bigl) \left( Q_{\{1\}},Q_{\{1,2,3\}},Q_{\{3\}}, Q_{\{3,4\}} \mid \boldsymbol{T} \right) \overset{d}{\rightarrow} \Bigl(
U_1 + \frac{1}{3}U_2 + \frac{1}{4}U_3,
\frac{1}{4}U_3,
 \frac{2}{9}U_2 + \frac{1}{6}U_3,
\frac{4}{9}U_2 + \frac{1}{3}U_3
 \Bigl),
\end{equation}
as $\lambda\uparrow\lambda^*$ by relying on the observation in~\eqref{eq:convergence_type_level}.
\end{example}

We will now combine the results in this subsection with those in the previous subsection to describe the heavy-traffic limiting distribution of the scaled numbers of jobs of the various types, irrespective of the ordered vector of job types.

First we note that only few of the possible values of the $\mathcal{N}$-dimensional random vector $\boldsymbol{\alpha} = (\alpha_{\boldsymbol{T}})_{\boldsymbol{T}\in\mathcal{N}}$ can occur in the heavy-traffic regime with non-zero probability. Specifically, Lemma~\ref{lem:limit_stateprob} implies that 
\begin{equation}
\boldsymbol{\alpha} = (\alpha_{\boldsymbol{T}})_{\boldsymbol{T}\in \mathcal{N}} \overset{d}{\rightarrow} \boldsymbol{\alpha}^* =  (\alpha^*_{\boldsymbol{T}})_{\boldsymbol{T}\in \mathcal{N}}
\end{equation}
as $\lambda\uparrow\lambda^*$, where
$\boldsymbol{\alpha}^*= \boldsymbol{e}_{\boldsymbol{T}}$ with probability $\mathbb{P}^*(\boldsymbol{T})$ if  $\boldsymbol{T} \in \mathcal{N}_K$ and $\boldsymbol{\alpha}^*= \boldsymbol{0}$ otherwise.

Second, we observe that, irrespective of the choice of $\boldsymbol{T} \in \mathcal{N}$, $\mathds{1}(\boldsymbol{T})\in\{0,1\}^{|\mathcal{S}|}$ becomes an all-zeros vector in the limit when scaled by $(1-\lambda/\lambda^*)$.
In other words,
the contribution of the first occurrence of a job type to the (scaled) total number of jobs of that type will be negligible in the heavy-traffic regime.

Hence, we are now interested in 
\begin{equation}
\lim\limits_{\lambda\uparrow\lambda^*} \Bigl( 1-\frac{\lambda}{\lambda^*} \Bigl) P(\boldsymbol{T})\tilde{M}(\boldsymbol{T}) = P(\boldsymbol{T}) \cdot \Bigl\{ \lim\limits_{\lambda\uparrow\lambda^*}  \Bigl( 1-\frac{\lambda}{\lambda^*} \Bigl) \tilde{M}(\boldsymbol{T}) \Bigl\}   \eqqcolon P(\boldsymbol{T}) \cdot \tilde{M}^*(\boldsymbol{T}).
\end{equation}
Thanks to Lemmas~\ref{lem:limit_segment} and~\ref{lem:limit_segment_type_geom} we know how to characterize the limiting random variables which are the entries of the matrix $\tilde{M}^*(\boldsymbol{T})$. With $\boldsymbol{T}\in\mathcal{N}_k$, only $k$ out of the $|\boldsymbol{T}|$ columns will contain non-zero elements (Lemma~\ref{lem:limit_segment}), with which we can associate $k$ independent unit-mean exponential random variables $U_1,\dots,U_k$. 

Combining the above results, the (scaled) limiting distribution of~\eqref{eq:probabilistic_representation} is given by
\begin{equation}
 \Bigl( 1-\frac{\lambda}{\lambda^*} \Bigl) \left( Q_S \right)_{S\in\mathcal{S}} \overset{d}{\rightarrow} 
 \sum\limits_{\boldsymbol{T}\in\mathcal{N}_K} \alpha^*_{\boldsymbol{T}}P(\boldsymbol{T})  \tilde{M}^*(\boldsymbol{T}) E(\boldsymbol{T}),
\end{equation}
as $\lambda\uparrow\lambda^*$. Note that the expression on the right-hand side is indeed a vector with components indexed in $S\in\mathcal{S}$.

Furthermore, rewriting $ \tilde{M}^*(\boldsymbol{T}) E(\boldsymbol{T})$ can further simplify the representation due to the $|\boldsymbol{T}|-K$ all-zeros columns. Define the $|\mathcal{S}|\times K$ dimensional matrix $W(\boldsymbol{T})$, which contains the weights or fractions each of the job types receives of the independent exponential random variables $U_1,\dots,U_k$,
i.e.,
\begin{equation}\label{eq:W_T}
 W(\boldsymbol{T})_{i,l} = \left\{
 \begin{array}{lcl}
\frac{p_{T_i}}{p(\boldsymbol{T},i_l)} & & \text{if~} i \le i_l,\\
 0 & & \text{otherwise},
 \end{array}
 \right.
 \end{equation}
 with $\text{CR}(\boldsymbol{T}) = \{i_1,\dots,i_K\}$ for all $i = 1,\dots,|\mathcal{S}|$ and $l=1,\dots,K$.
 Note that we will not alter the permutation matrix~$P(\boldsymbol{T})$, hence the indices in the above-defined matrix are resembling the job types, and hence the order, of the vector $\boldsymbol{T}$. Moreover, with 
 \begin{equation}
 U(K) = 
 \left[
 \begin{array}{c}
 U_1\\ U_2 \\ \vdots \\U_K
 \end{array}
 \right],
 \end{equation}
 we have that $\tilde{M}^*(\boldsymbol{T}) E(\boldsymbol{T}) = W(\boldsymbol{T})U(K)$.
Indeed, for $i=1,\dots,|\boldsymbol{T}|$ it holds that
\begin{equation}
\begin{array}{rcccccc}
[\tilde{M}^*(\boldsymbol{T}) E(\boldsymbol{T})]_i 
& = & \sum\limits_{j=1}^{\boldsymbol{T}}  \tilde{M}^*(\boldsymbol{T})_{i,j}E(\boldsymbol{T})_j 
& = & \sum\limits_{j=1}^{\boldsymbol{T}}  \tilde{M}^*(\boldsymbol{T})_{i,j} 
& = & \sum\limits_{j=1}^{\boldsymbol{T}} \lim\limits_{\lambda\uparrow\lambda^*} \left(1-\frac{\lambda}{\lambda^*} \right) Q^{j}_{T_i} \\
& = & \sum\limits_{j=i}^{\boldsymbol{T}} \lim\limits_{\lambda\uparrow\lambda^*} \left(1-\frac{\lambda}{\lambda^*} \right) Q^{j}_{T_i} 
& = & \sum\limits_{\substack{l=1 \\ i_l \ge i}}^{K} \frac{p_{T_i}}{p(\boldsymbol{T},i_l)}U_l 
& = & \sum\limits_{l=1}^{K}W(\boldsymbol{T})_{i,l}U(K)_l \\ 
& = & [W(\boldsymbol{T})U(K)]_i. & & & &
\end{array}
\end{equation}
If $i=|\boldsymbol{T}|+1,\dots,|\mathcal{S}|$, then $[\tilde{M}^*(\boldsymbol{T}) E(\boldsymbol{T})]_i = [W(\boldsymbol{T})U(K)]_i = 0$.

Consequently, we conclude that 
\begin{equation}
 \Bigl( 1-\frac{\lambda}{\lambda^*} \Bigl) \left( Q_S \right)_{S\in\mathcal{S}} \overset{d}{\rightarrow} 
 \sum\limits_{\boldsymbol{T}\in\mathcal{N}_K} \alpha^*_{\boldsymbol{T}}P(\boldsymbol{T})  W(\boldsymbol{T}) U(K) 
\end{equation}
as $\lambda\uparrow\lambda^*$,
which is consistent with the result in Theorem~\ref{cor:mixture_dist}. We emphasize that the matrices $P(\boldsymbol{T})$ and $W(\boldsymbol{T})$ consist solely of real-valued and non-negative entries given the vector $\boldsymbol{T}$, and only $U(K)$ contains random variables. Hence, $\left( Q_S \right)_{S\in\mathcal{S}}$, properly scaled, will follow a mixture distribution in the limiting regime, where each mixture component consists of a weighted sum of independent exponential random variables. Note that the expression on the right-hand side of the above equation is again a vector with components indexed by $S\in\mathcal{S}$.

\begin{example}
Consider Example~\ref{ex:example_main_result} and $\boldsymbol{T} = [\{1\},\{3\},\{3,4\},\{1,2,3\}]$, then 
\begin{equation}
P(\boldsymbol{T})\tilde{M}^*(\boldsymbol{T})
=
\left[
\begin{array}{cccc}
U_1 & 0 & \frac{1}{3}U_2 & \frac{1}{4}U_3\\
 0 & 0 & 0 &  \frac{1}{4}U_3\\
 0 & 0 & \frac{2}{9}U_2 & \frac{1}{6}U_3\\
0 & 0 & \frac{4}{9}U_2 & \frac{1}{3}U_3\\
\end{array}
\right].
\end{equation}
Alternatively,
\begin{equation}
P(\boldsymbol{T})W(\boldsymbol{T})
=
\left[
\begin{array}{ccc}
 1 & \frac{1}{3} & \frac{1}{4}\\
0 & 0 & \frac{1}{4}\\
0 & \frac{2}{9} & \frac{1}{6}\\
0 & \frac{4}{9} & \frac{1}{3}\\
\end{array}
\right]
\quad
\text{and}
\quad
U(3)
=
\left[
\begin{array}{c}
U_1  \\
U_2 \\
U_3
\end{array}
\right],
\end{equation}
such that~\eqref{eq:example_HT} holds. Moreover, taking into account all critical ordered vectors in~$\mathcal{N}_3$, the \textup{(}unconditional\textup{)} mixture distribution of the numbers of jobs of the various types is given by 
\begin{equation}
\Bigl(1-\frac{\lambda}{\lambda^*} \Bigl) \left( Q_S \right)_{S\in\mathcal{S} } \overset{d}{\rightarrow} (X_S)_{S\in\mathcal{S}}
\end{equation}
as $\lambda\uparrow\lambda^*$ with
$
(X_S)_{S\in\mathcal{S}}$ as in Table~\ref{tab:mixture_dist}.
\end{example}

\begin{remark}
The above heavy-traffic characterization focuses on the redundancy c.o.c.\ policy. However, the same conclusions hold for the c.o.s.\ mechanism. The key observation is that 
\begin{equation}
\lim\limits_{\lambda\uparrow\lambda^*} \mathbb{P}_{\text{c.o.s.}}(\boldsymbol{T}) = \frac{h^*(\boldsymbol{T},0)k(\boldsymbol{T})}{\sum\limits_{\tilde{\boldsymbol{T}}\in\mathcal{N}_K}h^*(\tilde{\boldsymbol{T}},0)k(\tilde{\boldsymbol{T}})} = \frac{h^*(\boldsymbol{T},0)}{\sum\limits_{\tilde{\boldsymbol{T}}\in\mathcal{N}_K}h^*(\tilde{\boldsymbol{T}},0)} =   \mathbb{P}^*(\boldsymbol{T}),
\end{equation}
where we used that the value of $k(\boldsymbol{T})$ is the same for all vectors $\boldsymbol{T}\in\mathcal{N}_K$ as they all include the same set of jobs.
\end{remark}

\section{Moments}\label{sec:moments}
In Section~\ref{sec:main_results} we focused on the Laplace transform and joint distribution of the (scaled) numbers of jobs of the various types. In this section we will turn our attention to the $n$th moments of the total number of jobs and those of the individual job types.
The $n$th moments of the total number of (waiting) jobs for a fixed arrival rate were derived in~\cite[Lemma EC.4]{Cardinaels2022} and are stated below for completeness. 
\begin{proposition}\label{prop:nth_moment_total}
Under the stability conditions stated in Condition~\ref{cond:stab}, the $n$th moment of the total number of jobs in the system, $Q$, under the c.o.c.\ mechanism is given by
\begin{equation}\label{eq:nth_moment_total}
\mathbb{E}\left[ Q^n \right] = n! \sum\limits_{\boldsymbol{S}\in \mathcal{N}} \gamma(\boldsymbol{S}) \times \mathbb{P}(\boldsymbol{S}),
\end{equation}
where for any $\boldsymbol{S}\in\mathcal{S}_m$ and $j=1,\dots,m$ we define 
\begin{equation}\label{eq:gamma_S}
\gamma(\boldsymbol{S}) \coloneqq \sum\limits_{\boldsymbol{k}\colon k_0+\dots + k_m=n} \frac{m^{k_0}}{k_0!}\prod\limits_{i=1}^m \sum\limits_{\boldsymbol{m}\in R(k_i)}\binom{m_1+\dots+m_{k_i}}{m_1,\dots,m_{k_i}} \left(1-B(\boldsymbol{S},i)\right)^{-|\boldsymbol{m}|}
\prod\limits_{j=1}^{k_i}\frac{B(\boldsymbol{S},i)^{m_j}}{j!^{m_j}},
\end{equation}
\begin{equation}\label{eq:B_S_j}
B(\boldsymbol{S},j) \coloneqq  \frac{N\lambda p(\boldsymbol{S},j)}{\mu(\boldsymbol{S},j)},
\end{equation}
and $\mathbb{P}(\boldsymbol{S})$ as defined Theorem~\ref{th:prelimit_distributions}.
Furthermore
\begin{equation}
R(k) \coloneqq \{ \boldsymbol{m}\in\mathbb{N}^k\colon 1m_1+2m_2 +\dots +km_k = k \} \quad \text{and} \quad |\boldsymbol{m}| = m_1 +m_2+\dots +m_k.
\end{equation}
Under the stability conditions stated in Condition~\ref{cond:stab}, the $n$th moment of the total number of waiting jobs in the system, $\tilde{Q}$, under the c.o.s.\ mechanism is given by
\begin{equation}
\mathbb{E}\left[ \tilde{Q}^n \right] = n!\sum\limits_{\boldsymbol{S}\in \mathcal{N}} \gamma(\boldsymbol{S}) \times  \mathbb{P}_{\text{c.o.s.}}(\boldsymbol{S}),
\end{equation}
with $\gamma(\boldsymbol{S})$ and $\mathbb{P}_{\text{c.o.s.}}(\boldsymbol{S})$ as defined in~\eqref{eq:gamma_S} and \eqref{eq:P_T_cos}, respectively.
\end{proposition}

Alternatively, one could rely on the probabilistic arguments outlined in Subsection~\ref{sec:interpretation} to derive an expression for the $n$th moments in terms of the individual moments of geometrically distributed random variables. An outline of the proof is deferred to Appendix~\ref{sec:moments_alternative}.

\subsection{Convergence of the moments}

Relying on Proposition~\ref{prop:nth_moment_total}, we can prove the following result.
\begin{theorem}\label{th:n_th_moment_limit}
With the definitions as in Section~\ref{sec:preliminaries}, it holds for both the c.o.c.\ and c.o.s. mechanisms
\begin{equation}\label{eq:n_th_moment_limit}
\lim\limits_{\lambda\uparrow\lambda^*} \mathbb{E}\left[ \left(\left(1-\frac{\lambda}{\lambda^*}\right) Q \right)^n \right] = \frac{(n+K-1)!}{(K-1)!}
\end{equation}
for any $n\ge 1$.
\end{theorem}

Note that Theorem~\ref{th:n_th_moment_limit} aligns with the result in Corollary~\ref{cor:total_num_jobs} since the $n$th moment of an Erlang distribution with parameters 1 and $K$ is precisely given by~\eqref{eq:n_th_moment_limit}.

\begin{proof}
\textbf{C.o.c.\ mechanism:}\\
From Proposition~\ref{prop:nth_moment_total}, we can write
\begin{equation}\label{eq:nthmoment_goal}
\mathbb{E}\left[ \left(\left(1-\frac{\lambda}{\lambda^*}\right) Q \right)^n \right] = \frac{\tilde{f}\cdot\left(1-\frac{\lambda}{\lambda^*}\right)^K}{f(1)\left(1-\frac{\lambda}{\lambda^*}\right)^K},
\end{equation}
with 
\begin{equation}
\tilde{f} = n!\left(1-\frac{\lambda}{\lambda^*}\right)^n
\sum\limits_{k=0}^K \sum\limits_{\boldsymbol{T}\in\mathcal{N}_k}
\gamma(\boldsymbol{T}) 
h(\boldsymbol{T},\boldsymbol{1}),
\end{equation}
and $h(\boldsymbol{T},\boldsymbol{1})$ as in~\eqref{eq:def_h_T_z}.
Note that $f$ and $\tilde{f}$ no longer depend on $(z_S)_{S\in\mathcal{S}}$, but do depend on the arrival rate~$\lambda$. To ease the notation, this dependence is omitted. We also note that the summation index~$m$ in~\eqref{eq:nth_moment_total} ranges from 1 to $|\mathcal{S}|$, hence in the notation above we must exclude the empty vector from~$\mathcal{N}_0$. This is a minor technicality that is negligible in the heavy-traffic limit.
Let us first focus on the limiting behavior of the numerator of~\eqref{eq:nthmoment_goal}, which can be rewritten as
\begin{equation}\label{eq:numerator}
n! \sum\limits_{k=0}^K \sum\limits_{\boldsymbol{T}\in\mathcal{N}_k}
\left(1-\frac{\lambda}{\lambda^*}\right)^{K-k}
\times
h(\boldsymbol{T},\boldsymbol{1})\left(1-\frac{\lambda}{\lambda^*}\right)^{k}
\times
 \gamma(\boldsymbol{T})\left(1-\frac{\lambda}{\lambda^*}\right)^{n}
.
\end{equation}
It can easily be observed that
\begin{equation}
\lim\limits_{\lambda\uparrow\lambda^*}
\left(1-\frac{\lambda}{\lambda^*}\right)^{K-k}
=
\left\{
\begin{array}{lcr}
0 & & \text{if~}k<K, \\
1 & & \text{if~}k=K. \\
\end{array}
\right.
\end{equation}
Moreover, from Lemma~\ref{lem:intermediate_result}, we know that
\begin{equation}
\lim\limits_{\lambda\uparrow\lambda^*} h(\boldsymbol{T},\boldsymbol{1})\left(1-\frac{\lambda}{\lambda^*}\right)^{k} = \beta(\boldsymbol{T}),
\end{equation}
a finite, non-zero constant for each fixed~$\boldsymbol{T}$.
Define
\begin{equation}
{\gamma}(\boldsymbol{T},i)
\coloneqq
\sum\limits_{\boldsymbol{m}\in R(n_i)}
\binom{m_1+\dots+m_{n_i}}{m_1,\dots,m_{n_i}}
\left(
1- \frac{N\lambda p(\boldsymbol{T},i)}{\mu(\boldsymbol{T},i)}
\right)^{-|\boldsymbol{m}|}
\prod\limits_{j=1}^{n_i} \frac{B(\boldsymbol{T},i)^{m_j}}{j!^{m_j}},
\end{equation}
then
\begin{equation}\label{eq:detail_gamma}
\left(1-\frac{\lambda}{\lambda^*}\right)^{n} \gamma(\boldsymbol{T})
 = 
\sum\limits_{\boldsymbol{n}\colon n_0+\dots+n_{|\boldsymbol{T}|} = n}
\left(1-\frac{\lambda}{\lambda^*}\right)^{n_0}\frac{|\boldsymbol{T}|^{n_0}}{n_0!} \times
\prod\limits_{i\notin \text{CR}(\boldsymbol{T})} {\gamma}(\boldsymbol{T},i)
\times
\left( 1-\frac{\lambda}{\lambda^*}\right)^{n-n_0} 
\prod \limits_{i\in \text{CR}(\boldsymbol{T})} {\gamma}(\boldsymbol{T},i)
.
\end{equation}
It can be seen that 
\begin{equation}
\lim\limits_{\lambda\uparrow\lambda^*}
\left(1-\frac{\lambda}{\lambda^*}\right)^{n_0}
=
\left\{
\begin{array}{lcr}
0 & & \text{if~}n_0>0, \\
1 & & \text{if~}n_0=0, \\
\end{array}
\right.
\end{equation}
and 
\begin{equation}
\lim\limits_{\lambda\uparrow\lambda^*} {\gamma}(\boldsymbol{T},i) \in (0,\infty)
\end{equation}
for all $i\notin\text{CR}(\boldsymbol{T})$. Only the last part of~\eqref{eq:detail_gamma} requires some further investigation. Note that it can be rewritten as
\begin{equation} \label{eq:gamma_split}
\prod \limits_{i\notin \text{CR}(\boldsymbol{T})} \left(1-\frac{\lambda}{\lambda^*} \right)^{n_i}
\times
\prod \limits_{i\in \text{CR}(\boldsymbol{T})} 
\left(1-\frac{\lambda}{\lambda^*} \right)^{n_i}
{\gamma}(\boldsymbol{T},i).
\end{equation}
Obviously, the first product has a non-zero limit only if $n_i =0$ for all $i\notin \text{CR}(\boldsymbol{T})$. For each $i\in \text{CR}(\boldsymbol{T})$, the corresponding term in the above product can be rewritten as
\begin{equation}\label{eq:fixed_i_in_CR}
\sum\limits_{\boldsymbol{m}\in R(n_i)}
\binom{m_1+\dots+m_{n_i}}{m_1,\dots,m_{n_i}}
\left\{
\left(1-\frac{\lambda}{\lambda^*} \right)
\left(
1- \frac{N\lambda p(\boldsymbol{T},i)}{\mu(\boldsymbol{T},i)}
\right)^{-1}
\right\}^{|\boldsymbol{m}|}
\left(1-\frac{\lambda}{\lambda^*} \right)^{n_i-|\boldsymbol{m}|}
\prod\limits_{j=1}^{n_i} \frac{B(\boldsymbol{T},i)^{m_j}}{j!^{m_j}}.
\end{equation}
Now, for any $\boldsymbol{m}\in R(n_i)$, 
\begin{equation}
\lim\limits_{\lambda\uparrow\lambda^*} 
\prod\limits_{j=1}^{n_i} \frac{B(\boldsymbol{T},i)^{m_j}}{j!^{m_j}} 
= 
\prod\limits_{j=1}^{n_i} (j!)^{-m_j} \in (0,\infty),
\end{equation}
and 
\begin{equation}
\lim\limits_{\lambda\uparrow\lambda^*} 
\left\{
\left(1-\frac{\lambda}{\lambda^*} \right)
\left(
1- \frac{N\lambda p(\boldsymbol{T},i)}{\mu(\boldsymbol{T},i)}
\right)^{-1}
\right\}^{|\boldsymbol{m}|}
= 1
\end{equation}
as in~\eqref{eq:intermed_res_hopital}.
By definition of $\boldsymbol{m}\in R(n_i)$, we have that $n_i \ge |\boldsymbol{m}|$. However, only terms with precisely $n_i =|\boldsymbol{m}|$ will be non-vanishing in~\eqref{eq:fixed_i_in_CR}. This implies that $m_1 = n_i$ and $m_2 = \dots = m_{n_i} = 0$. Due to the above arguments, \eqref{eq:fixed_i_in_CR} converges to~1 as $\lambda\uparrow\lambda^*$. Hence, \eqref{eq:gamma_split} converges to 1 if all $n_i=0$ for $i\notin\text{CR}(\boldsymbol{T})$, and 0 otherwise. This results in
\begin{equation}
\lim\limits_{\lambda\uparrow\lambda^*} 
\left(
1-\frac{\lambda}{\lambda^*}
\right)^{n}
\gamma(\boldsymbol{T}) =
| \{ \boldsymbol{n}\in \mathbb{N}^{|\boldsymbol{T}|} \colon
n_1 +\dots +n_{|\boldsymbol{T}|} = n, n_i = 0 \text{~for~} i \in \text{CR}(\boldsymbol{T}) \} |
= \binom{n+k-1}{n}
\end{equation}
as $\boldsymbol{T}\in\mathcal{N}_k$. Hence, we can conclude that~\eqref{eq:numerator} converges to
\begin{equation}\label{eq:limit_numer}
n!\sum\limits_{\boldsymbol{T}\in\mathcal{N}_K} \beta(\boldsymbol{T})  \binom{n+K-1}{n} = \frac{(n+K-1)!}{(K-1)!} \sum\limits_{\boldsymbol{T}\in\mathcal{N}_K} \beta(\boldsymbol{T}) = \frac{(n+K-1)!}{(K-1)!} \beta(\mathcal{N}_K). 
\end{equation}
Now, let us focus on the denominator of~\eqref{eq:nthmoment_goal}. From~\eqref{eq:limit_f_k} and~\eqref{eq:theorem_part2} we know that 
\begin{equation}\label{eq:limit_denom}
\lim\limits_{\lambda\uparrow\lambda^*} f(1)\left(1-\frac{\lambda}{\lambda^*}\right)^K = \lim\limits_{\lambda\uparrow\lambda^*} f_K(1)\left(1-\frac{\lambda}{\lambda^*}\right)^K 
= \sum\limits_{\boldsymbol{T}\in\mathcal{N}_K} \beta(\boldsymbol{T})=\beta(\mathcal{N}_K).
\end{equation}
Dividing~\eqref{eq:limit_numer} by~\eqref{eq:limit_denom} concludes the proof for the c.o.c.\ mechanism.\\

\textbf{C.o.s.\ mechanism:}\\
The proof concepts are similar to those in the proof for the c.o.c.\ mechanism. Together with the arguments in the proof of Proposition~\ref{th:limit_gf} we can observe that
\begin{equation}
\begin{array}{l}
\lim\limits_{\lambda\uparrow\lambda^*}
\left(1-\frac{\lambda}{\lambda^*}\right)^{n+K}\sum\limits_{L=0}^N\sum\limits_{\boldsymbol{u}\in\mathcal{E}_L}\alpha(\boldsymbol{u})\sum\limits_{m=1}^{|\mathcal{S}|}\sum\limits_{\boldsymbol{S}\in \mathcal{S}_m^{\boldsymbol{u}}} \gamma(\boldsymbol{S}) h(\boldsymbol{S},\boldsymbol{1})\\
=
\binom{n+K-1}{K-1}\sum\limits_{L=0}^N\sum\limits_{\boldsymbol{u}\in\mathcal{E}_L}\alpha(\boldsymbol{u}) \sum\limits_{\boldsymbol{T}\in\mathcal{N}_K^{\boldsymbol{u}}} \beta(\boldsymbol{T}),
\end{array}
\end{equation}
and
\begin{equation}
\left(1-\frac{\lambda}{\lambda^*}\right)^{K}g(1) = \sum\limits_{L=0}^N\sum\limits_{\boldsymbol{u}\in\mathcal{E}_L}\alpha(\boldsymbol{u}) \sum\limits_{\boldsymbol{T}\in\mathcal{N}_K^{\boldsymbol{u}}} \beta(\boldsymbol{T}).
\end{equation}

In order to prove the convergence result in Theorem~\ref{th:n_th_moment_limit} in terms of the total number of jobs in the system, we use that $\tilde{Q}\le Q\le \tilde{Q}+N$. Hence, 
\begin{equation}
\mathbb{E}\left[\left(\left(1-\frac{\lambda}{\lambda^*}\right)\tilde{Q}\right)^n \right] \le \mathbb{E}\left[\left(\left(1-\frac{\lambda}{\lambda^*}\right)Q \right)^n \right] \le  \mathbb{E}\left[\left(\left(1-\frac{\lambda}{\lambda^*}\right)(\tilde{Q}+N)\right)^n \right],
\end{equation}
where the last term can be rewritten as
\begin{equation}
\mathbb{E}\left[\left(\left(1-\frac{\lambda}{\lambda^*}\right)\tilde{Q}\right)^n \right] 
+
 \sum\limits_{k=0}^{n-1} \binom{n}{k} N^{n-k}\left(1-\frac{\lambda}{\lambda^*}\right)^{n-k} 
 \mathbb{E}\left[\left(\left(1-\frac{\lambda}{\lambda^*}\right)\tilde{Q}\right)^k \right] \rightarrow \binom{n+K-1}{K-1}
\end{equation}
as $\lambda\uparrow\lambda^*$. We used the inductive argument that the $k$th moment of $\left(1-\frac{\lambda}{\lambda^*}\right)\tilde{Q}$ converges to $\binom{k+K-1}{K-1}$, for $k<n$. This concludes the proof.
\end{proof}

Applying Little's law~\cite{Little1961} to the first moments in Theorem~\ref{th:n_th_moment_limit} results in the following corollary for the response time of an arbitrary job.

\begin{corollary}[Expected response time]\label{cor:response_time} With the definitions as in Section~\ref{sec:preliminaries} and $R$ denoting the response time of an arbitrary job, it holds for both the c.o.c.\ and c.o.s.\ mechanisms
\begin{equation}
\lim\limits_{\lambda\uparrow\lambda^*} \mathbb{E}\left[ \left(1-\frac{\lambda}{\lambda^*}\right) R \right] = \frac{K}{N\lambda^*}.
\end{equation}
\end{corollary}

In~\cite[Lemma~EC.5]{Cardinaels2022}, also an expression for the $n$th moment of the number of jobs of a particular job type $S\in\mathcal{S}$ was derived. Using a similar methodology as above we can prove the following theorem.

\begin{theorem}\label{th:n_th_moment_type}
With the definitions as in Sections~\ref{sec:preliminaries} and \ref{sec:main_results}, it holds for a given job type $S\in\mathcal{S}$ for any $n\ge 1$
\begin{equation}
\lim\limits_{\lambda\rightarrow\lambda^*}
\mathbb{E}\left[ \left( 1-\frac{\lambda}{\lambda^*}\right)^n Q_S^n \right]
=
n!p_S^n
\sum\limits_{\sigma\in\Sigma_K} \mathbb{P}^*(\boldsymbol{T}^{\sigma})
\times
\sum\limits_{\boldsymbol{n}\colon n_{1}+\dots n_{K} = n}
\prod\limits_{k=1}^K p(\mathcal{C}_{\sigma(1)},\dots,\mathcal{C}_{\sigma(k)})^{-n_{k}}
\end{equation}
for the c.o.c.\ mechanism and
\begin{equation}
\lim\limits_{\lambda\rightarrow\lambda^*}
\mathbb{E}\left[ \left( 1-\frac{\lambda}{\lambda^*}\right)^n \tilde{Q}_S^n \right]
=
n!p_S^n  
\sum\limits_{\sigma\in\Sigma_K} \mathbb{P}^*(\boldsymbol{T}^{\sigma})
\times
\sum\limits_{\boldsymbol{n}\colon n_{1}+\dots n_{K} = n}
\prod\limits_{k=1}^K p(\mathcal{C}_{\sigma(1)},\dots,\mathcal{C}_{\sigma(k)})^{-n_{k}}
\end{equation}
for the c.o.s.\ mechanism.
\end{theorem} 


Next we will connect our results to the results derived by Af\`eche \emph{et al.}~\cite{Afeche2021}. Their analysis focuses on the redundancy c.o.s.\ mechanism or, as it is referred to there, the FCFS-ALIS (First Come First Served - Assign to the Longest Idle Server) policy. The motivation for the work in~\citep{Afeche2021} is to obtain some insight in the trade-off between the expected waiting times and rewards in parallel-server systems with compatibility constraints where the CRP condition is not necessariliy satisfied. For this purpose they build on the product-form stationary distributions obtained by Adan \& Weiss~\cite{Adan2018}, which involve a slightly different state descriptor than the ones derived in~\cite{Gardner2016queueing}. Instead of focusing on the central queue and the job type labels, this state descriptor is more server-centered and omits the job type labels. The state is of the form
$(s_1,n_1,s_2,n_2,\dots,s_k,n_k,s_{k+1},\dots,s_N)$ with $(s_1,\dots,s_N)$ a permutation of the $N$ servers. Here $s_1,\dots,s_k$ denote the busy servers ordered according to the arrival times of the jobs they are processing while $s_{k+1},\dots,s_N$ represent the ordered idle servers where server $s_{k+1}$ is idle for the longest time. The number of jobs that arrived after the job that is in service at server $s_i$ and can only be served by some servers in $s_1,\dots,s_i$ is denoted by $n_i$ for $i=1,\dots,k$.
 In view of the computational complexity of the stationary distribution, a heavy-traffic perspective is adopted in which it is possible to derive expressions for the expected waiting time of a given job type or the overall expected waiting time. 
In particular, in \cite[Corollary~2]{Afeche2021} they state that 
\begin{equation}\label{eq:results_afeche}
\lim\limits_{\lambda\uparrow\mu} \left(1-\frac{\lambda}{\mu}\right)\mathbb{E}[W] = \frac{K}{N\mu},
\end{equation}
with $W$ the waiting time of an arbitrary job. Obviously this coincides with our result in Corollary~\ref{cor:response_time}, taking into account that the service time of a job becomes negligible after scaling. We note that the `$K$' in their corollary is indeed the depth of the critical sets as formulated in Definition~\ref{def:depthK}, though it has a slightly different interpretation in~\cite{Afeche2021}, as alluded to in Remark~\ref{rem:K}.

It is worth emphasizing that the results in~\citep{Afeche2021} pertain to the first moments of the total number of jobs and the number of jobs of the various types, while our results go beyond the first moment (Theorems~\ref{th:n_th_moment_limit} and~\ref{th:n_th_moment_type}) and even describe the full (joint) stationary distribution (Theorem~\ref{th:general_results}).
Using similar arguments as in Section~\ref{sec:alternative proof}, it should be possible to extend the results in~\citep{Afeche2021} beyond the first moments.

Let us outline their \emph{server}-centered approach.
The states in~\citep{Afeche2021} are aggregated according to ordered vectors of servers. Given this ordering, the observation is made that the number of jobs between two servers is geometrically distributed and each job can be labeled with a particular type label independently~\cite{visschers2012product}. Hence the total number of (waiting) jobs of a particular type is given by a sum of geometric distributions (with different parameters). This yields the expected number of jobs of each type and, after applying Little's law, the expected waiting times as well.
 Though our analysis approach is more \emph{job}-centered, we apply some similar steps in the alternative proof of Theorem~\ref{th:general_results}. After aggregating states according to orderings of the job types (Theorem~\ref{th:prelimit_distributions}, part~c), we observe that the numbers of jobs of the various types in each segment also follow a geometric distribution (Theorem~\ref{th:prelimit_distributions}, parts~a and~b) such that we have a full characterization of the vector of queue lengths of the various types in terms of geometrically distributed random variables. It is noted that this characterization does not give immediate insight into the performance of the system due to the delicate interplay between the job-type orderings and the parameters of the geometric distributions, already computing higher-ordered moments is non-trivial (Proposition~\ref{prop:nth_moment_total_alternative}). Given these geometric distributions, we then analyze the system when $\lambda\uparrow\lambda^*$ in Subsection~\ref{subsec:alternative_proof_HT}. A similar approach should be possible in the \emph{server}-centered approach, using then the distributional form of Little's law~\cite{Little1961}, the waiting time distribution of the various job types could be derived to, for instance, also include the variance in the trade-off analysis conducted in~\cite{Afeche2021,Hillas2023}.

\subsection{Alternative formulation of Proposition~\ref{prop:nth_moment_total}}\label{sec:moments_alternative}

First we will derive an expression for the $n$th moment of the total number of jobs relying on the probabilistic interpretation of Proposition~\ref{prop:nth_moment_total} as outlined in Subsection~\ref{sec:interpretation}. Then we show, in Lemma~\ref{lem:moments_Guido}, that the obtained expression indeed coincides with the result in Proposition~\ref{prop:nth_moment_total}.

\begin{proposition}\label{prop:nth_moment_total_alternative}
Under the stability conditions stated in Condition~\ref{cond:stab}, the $n$th moment of the total number of jobs in the system, $Q$, under the c.o.c.\ mechanism is given by
\begin{equation}
\mathbb{E}\left[ Q^n \right] = n! \sum\limits_{\boldsymbol{T}\in\mathcal{N}}  \hat{\gamma}(\boldsymbol{T}) \times \mathbb{P}(\boldsymbol{T}),
\end{equation}
where for any $\boldsymbol{T}\in\mathcal{N}$ and $j = 1,\dots,|\boldsymbol{T}|$ we define
\begin{equation}
\begin{array}{rcl}
\hat{\gamma}(\boldsymbol{T}) &=& \sum\limits_{\substack{k_0,k_1,\dots,k_{|\boldsymbol{T}|}\in\mathbb{N} \colon \\ k_0+k_1+\dots+k_{|\boldsymbol{T}|}=n}} \frac{|\boldsymbol{T}|^{k_0}}{k_0!}
\prod\limits_{i=1}^{|\boldsymbol{T}|} \frac{1}{k_i!} \hat{\gamma}(\boldsymbol{T},i),\\

\hat{\gamma}(\boldsymbol{T},i) &=&  \left(\frac{p^{\boldsymbol{T},i}}{1-p^{\boldsymbol{T},i}}\right)^{k_i} \sum\limits_{l=0}^{k_i} \left\langle \begin{matrix}k_i\\ l \end{matrix}  \right \rangle \left(p^{\boldsymbol{T},i}\right)^{-l},
\end{array}
\end{equation}
$\mathbb{P}(\boldsymbol{T})$ as in Theorem~\ref{th:prelimit_distributions}, $p^{\boldsymbol{T},i}$ as in~\eqref{eq:param_geometric} and $\left\langle \begin{matrix}k_i\\ l \end{matrix}  \right \rangle$ as the Eulerian number, i.e.,  the number of permutations of $\{1,\dots,k_i\}$ in which exactly $l$ elements are greater than the previous element.
\end{proposition}

\begin{proof}
First we focus on a $\boldsymbol{T}\in\mathcal{N}$, such that $\boldsymbol{T} = [T_1,\dots,T_{|\boldsymbol{T}|}]$ is an ordered vector of job types.
We will use the notation as introduced in Subsection~\ref{sec:interpretation} to compute $\mathbb{E}[Q^n\mid \boldsymbol{T}]$. Conditioning on all ordered vectors will then yield the desired results. \\

We know from Theorem~\ref{th:prelimit_distributions} that the total number of jobs can be counted by adding the total number of jobs in each segment and including the first occurrences of each segment, i.e., 
\begin{equation}
Q\mid \boldsymbol{T} \overset{d}{=} \sum_{j=1}^{|\boldsymbol{T}|} 1+ Q^j\mid \boldsymbol{T},
\end{equation}
where $(Q^j\mid \boldsymbol{T})_{j=1,\dots,|\boldsymbol{T}|}$ are independent and geometrically distributed random variables with parameters $(p^{\boldsymbol{T},j})_{j=1,\dots,|\boldsymbol{T}|}$ as in~\eqref{eq:param_geometric}. Hence, by applying the multinomium of Newton, it holds that
\begin{equation}
\begin{array}{rcl}
\mathbb{E}[Q^n\mid \boldsymbol{T}] & = & \mathbb{E}\left[\left( |\boldsymbol{T}| + \sum\limits_{j=1}^{|\boldsymbol{T}|}  Q^j\right)^n\mid \boldsymbol{T}\right]\\

& = & \sum\limits_{\substack{k_0,k_1,\dots,k_{|\boldsymbol{T}|}\in\mathbb{N} \colon \\ k_0+k_1+\dots+k_{|\boldsymbol{T}|}=n}} \binom{n}{k_0,\dots,k_{|\boldsymbol{T}|}}
\mathbb{E}\left[|\boldsymbol{T}|^{k_0!} \prod\limits_{j=1}^{|\boldsymbol{T}|}  (Q^j)^{k_j}\mid \boldsymbol{T}\right]\\

& = & n! \sum\limits_{\substack{k_0,k_1,\dots,k_{|\boldsymbol{T}|}\in\mathbb{N} \colon \\ k_0+k_1+\dots+k_{|\boldsymbol{T}|}=n}} 
\frac{|\boldsymbol{T}|^{k_0}}{k_0}
\prod\limits_{j=1}^{|\boldsymbol{T}|} 
\frac{1}{k_j!}
\mathbb{E} \left[  (Q^j)^{k_j}\mid \boldsymbol{T}\right].
\end{array}
\end{equation}
Focusing on the moments of the individual geometric distributions then yields
\begin{equation}
\begin{array}{rcl}
\hat{\gamma}(\boldsymbol{T},j) & = & \mathbb{E} \left[  (Q^j)^{k_j}\mid \boldsymbol{T}\right]\\

 & = & (1-p^{\boldsymbol{T},j}) \sum\limits_{i=0}^{\infty} i^{k_j} \left(p^{\boldsymbol{T},j} \right)^{i}  \\
 
  & = & (1-p^{\boldsymbol{T},j}) \frac{1}{(1-p^{\boldsymbol{T},j})^{k_j+1}} \sum\limits_{i=0}^{k_j}\left\langle \begin{matrix}k_j\\ i \end{matrix}  \right \rangle  \left(p^{\boldsymbol{T},j} \right)^{k_j-i}.
\end{array}
\end{equation}
The last equality is based on an alternative representation of the polylogarithmic function $\sum_{i\ge 1} z^i/ i^{-k_j}$ with $z = p^{\boldsymbol{T},j}$, and introduces the Eulerian number to the expression.

This concludes the proof.
\end{proof}

Using the identity in the following lemma, we can verify that the expressions for the $n$th moments in Propositions~\ref{prop:nth_moment_total} and~\ref{prop:nth_moment_total_alternative} indeed coincide.

\begin{lemma}\label{lem:moments_Guido}
Let $k\in\mathbb{N}$ and $p\in(0,1)$, then
\begin{equation}\label{eq:moments_guido}
\frac{1}{k!} \left(\frac{p}{1-p} \right)^k \sum\limits_{j=0}^{k} \left\langle \begin{matrix}k\\ j \end{matrix}  \right \rangle p^{-j}
=
\sum\limits_{\boldsymbol{m}\in R(k)}
\binom{m_1 + \dots + m_k}{m_1,\dots,m_k}
 (1-p)^{-|\boldsymbol{m}|} \prod\limits_{j=1}^{k} \frac{p^{m_j}}{(j!)^{m_j}}.
\end{equation}
\end{lemma}

\begin{proof}
The above result follows from the observation that the expressions can be considered as the coefficients of the same analytical function $F$ defined as
\begin{equation}
F(s) =  \left(1-\sum\limits_{r=1}^{\infty} c_r s^r \right)^{-1},
\end{equation}
with 
\begin{equation}
c_r \coloneqq \frac{1}{r!}\frac{p}{1-p}
\end{equation}
for $r\in\mathbb{N}$ and $s$ in a neighborhood of 0, more precisely $|p\mathrm{e}^s|<1$. Hence,
\begin{equation}
F(s)  =  \left(1-\frac{p}{1-p} (\mathrm{e}^s-1)\right)^{-1} =  \frac{1-p}{1-p\mathrm{e}^s},
\end{equation}
after substitution of $(c_r)_r$. Moreover, we can write
\begin{equation}
\begin{array}{rcl}
F(s) &=& \sum\limits_{m=0}^{\infty} \left( \sum\limits_{r=1}^{\infty} c_r s^r \right)^{m} \\
&=& 1 + \sum\limits_{m=1}^{\infty} \left( \sum\limits_{r=1}^{\infty} c_r s^r \right)^{m}\\
& = &  1+ \sum\limits_{m=1}^{\infty}
\sum\limits_{\substack{m_1,m_2,\dots \in \mathbb{N}\\ m_1+m_2+\dots = m}}
\binom{m}{m_1,m_2,\dots} \prod\limits_{j=1}^{\infty} (c_js^j)^{m_j}\\
& = &  1+ \sum\limits_{m=1}^{\infty}
\sum\limits_{\substack{m_1,m_2,\dots \in \mathbb{N}\\ m_1+m_2+\dots = m}}
\binom{m}{m_1,m_2,\dots} \prod\limits_{j=1}^{\infty} c_j^{m_j} \cdot s^{\sum_{j}jm_j},
\end{array}
\end{equation}
where we used the multinomium of Newton.
We now focus on the coefficient of $s^k$ in the above expression, which is given by
\begin{equation}
\begin{array}{rcl}
\sum\limits_{m=1}^{\infty}  \sum\limits_{\substack{m_1,m_2,\dots \in \mathbb{N}\\ m_1+m_2+\dots = m\\ 1m_1+2m_2 + \dots = k}}
\binom{m}{m_1,m_2,\dots} \prod\limits_{j=1}^{\infty} c_j^{m_j}
&=&
 \sum\limits_{\substack{m_1,m_2,\dots \in \mathbb{N}\\ 1m_1+2m_2 + \dots = k}}
\binom{m_1+m_2+\dots}{m_1,m_2,\dots} \prod\limits_{j=1}^{\infty} c_j^{m_j} \\
&=&
 \sum\limits_{\substack{m_1,m_2,\dots \in \mathbb{N}\\ 1m_1+2m_2 + \dots = k}}
\binom{m_1+m_2+\dots}{m_1,m_2,\dots} (1-p)^{-(m_1+m_2+\dots)} \prod\limits_{j=1}^{\infty} \frac{p^{m_j}}{(j!)^{m_j}}. 
\end{array}
\end{equation}
Notice that this expression of the coefficient of $s^k$ corresponds to the left-hand side of~\eqref{eq:moments_guido}. We now proceed in the opposite direction by setting up a function $\tilde{F}$ with as coefficients the right-hand side of~\eqref{eq:moments_guido} and showing that the resulting function coincides with $F$ by relying on the properties of the Eulerian numbers~\cite[Section~26.14]{Olver}. So, consider
\begin{equation}
\tilde{F}(s) = \sum_{k=1}^{\infty}  \left( \frac{1}{k!} (1-p)^{-k} \sum\limits_{j=0}^{k-1}\left\langle \begin{matrix}k\\ j \end{matrix}  \right \rangle p^{k-j}\right) s^k.
\end{equation}
We can rewrite the inner summation using a variable transformation,
\begin{equation}
\sum\limits_{j=0}^{k-1}\left\langle \begin{matrix}k\\ j \end{matrix}  \right \rangle p^{k-j} = 
\sum\limits_{i=1}^{k}\left\langle \begin{matrix}k\\ k-i \end{matrix}  \right \rangle p^{i}.
\end{equation}
From \cite[Eq. 26.14.9]{Olver} we know that $\left\langle \begin{matrix}k\\ i \end{matrix}  \right \rangle = \left\langle \begin{matrix}k\\ k-i-1 \end{matrix}  \right \rangle$ for $k\in\mathbb{N}\setminus\{0\}$ and $0\le i \le k-1$, such that 
\begin{equation}
\sum\limits_{i=1}^{k}\left\langle \begin{matrix}k\\ k-i \end{matrix}  \right \rangle p^{i}
=
\sum\limits_{i=0}^{k-1}\left\langle \begin{matrix}k\\ k-i-1 \end{matrix}  \right \rangle p^{i+1}
=
p \sum\limits_{i=0}^{k-1}\left\langle \begin{matrix}k\\ i \end{matrix}  \right \rangle p^{i}.
\end{equation}
So, 
\begin{equation}
\begin{array}{rcl}
\tilde{F}(s) & = & p  \sum\limits_{k=1}^{\infty}  \left( \frac{1}{k!} (1-p)^{-k} \sum\limits_{i=0}^{k-1}\left\langle \begin{matrix}k\\ i \end{matrix}  \right \rangle p^{i}\right) s^k\\
& = & p \sum\limits_{k=1}^{\infty} \sum\limits_{i=0}^{\infty}  \left\langle \begin{matrix}k\\ i \end{matrix}  \right \rangle p^{i}   \frac{\left((1-p)^{-1} s\right)^k}{k!},
\end{array}
\end{equation}
where we used that $\left\langle \begin{matrix}k\\ i \end{matrix}  \right \rangle = 0$ for $i\ge k\ge 1$. Notice that
\begin{equation}
p\sum\limits_{i=0}^{\infty} \left\langle \begin{matrix}0\\ i \end{matrix}  \right \rangle p^i \frac{\left((1-p)^{-1} s\right)^0}{0!}  = p,
\end{equation}
since $  \left\langle \begin{matrix}0\\ 0 \end{matrix}  \right \rangle = 1$ and $ \left\langle \begin{matrix}0\\ i \end{matrix}  \right \rangle = 0$ for $i\ge 1$, so that
\begin{equation}
\tilde{F}(s) = p \sum\limits_{k=0}^{\infty} \sum\limits_{i=0}^{\infty}  \left\langle \begin{matrix}k\\ i \end{matrix}  \right \rangle p^{i}   \frac{\left((1-p)^{-1} s\right)^k}{k!} -p.
\end{equation}
Relying on the following identity given in~\cite[Eq. 26.14.4]{Olver} which holds for $|x|<1$ and $|t|<1$
\begin{equation}
\sum\limits_{k,i=0}^{\infty} \left\langle \begin{matrix}k\\ i \end{matrix}  \right \rangle x^i \frac{t^k}{k!} = \frac{1-x}{\exp((x-1)t) - x},
\end{equation}
we can substitute $x=p$ and $t=\frac{s}{1-p}$ (with $|s| < 1-p$) to obtain
\begin{equation}
\tilde{F}(s) = p \frac{1-p}{\exp\left((p-1)\frac{s}{1-p}\right)-p} - p = 
\frac{p(1-p)}{\mathrm{e}^{-s}-p} - p =  \frac{1-p}{1-p\mathrm{e}^s}-1.
\end{equation}
Hence $F(s) = \tilde{F}(s)+1$, so that~\eqref{eq:moments_guido} follows.\\
This concludes the proof.
\end{proof}

\end{document}